\documentclass[11pt]{amsart}

\usepackage{amsthm, amsfonts, rlepsf, amscd}

\newcommand{\twocopy}{\ensuremath{2L}}
\newcommand{\aug}{\ensuremath{\varepsilon}}
\newcommand{\df}{\ensuremath{\partial}}
\newcommand{\dfe}{\ensuremath{\partial^\varepsilon}}
\newcommand{\fdf}{\ensuremath{\widehat{\df}}}
\newcommand{\faug}{\ensuremath{\widehat{\aug}}}
\newcommand{\fdfe}{\ensuremath{\fdf^{\faug}}}
\newcommand{\alg}{\ensuremath{\mathcal{A}}}
\newcommand{\qset}{\ensuremath{\mathcal{Q}}}
\newcommand{\ms}{\ensuremath{\mathcal{M}}}
\newcommand{\ev}{\operatorname{ev}}
\DeclareMathOperator{\area}{Area}

\newcommand{\la}{\langle}
\newcommand{\ra}{\rangle}
\newcommand{\pa}{\partial}

\newcommand{\Spa}{\operatorname{Span}}
\newcommand{\ix}{\operatorname{Index}}
\newcommand{\krn}{\operatorname{Ker}}

\newcommand{\img}{\operatorname{Im}}

\newcommand{\sblv}{{\mathcal{H}}}

\newcommand{\Ordo}{{\mathbf{O}}}

\newcommand{\cc}{\ensuremath{\mathbb{C}}}
\newcommand{\rr}{\ensuremath{\mathbb{R}}}
\newcommand{\zz}{\ensuremath{\mathbb{Z}}}
\newcommand{\qq}{\ensuremath{\mathbb{Q}}}

\theoremstyle{plain}
\newtheorem{thm}{Theorem}[section]
\newtheorem{cor}[thm]{Corollary}
\newtheorem{lem}[thm]{Lemma}
\newtheorem{claim}[thm]{Claim}
\newtheorem{prop}[thm]{Proposition}

\theoremstyle{definition}
\newtheorem{defn}[thm]{Definition}

\theoremstyle{remark}
\newtheorem{rem}[thm]{Remark}
\newtheorem{exam}[thm]{Example}

\numberwithin{equation}{section}

\def\dfn#1{{\em #1}}

\begin{document}

\title[A Duality Exact Sequence]{A Duality Exact Sequence for
  Legendrian Contact Homology}

\author[T. Ekholm]{Tobias Ekholm} \address{Uppsala University, Box 480,
  751 06 Uppsala, Sweden} \email{tobias@math.uu.se}
\author[J. Etnyre]{John B. Etnyre} \address{Georgia Institute of
  Technology, Atlanta, GA 30332} \email{etnyre@math.gatech.edu}
\author[J. Sabloff]{Joshua M. Sabloff} \address{Haverford College,
  Haverford, PA 19041} \email{jsabloff@haverford.edu}

\thanks{TE was supported by the Alfred P. Sloan Foundation, by
  NSF-grant DMS-0505076, and by the Royal Swedish Academy of Sciences,
  the Knut and Alice Wallenberg Foundation. JE was partially supported
  by NSF grants DMS-0804820,  DMS-0707509 and DMS-0244663. \\ 
  An updated version of this article has appeared in the Duke Mathematical Journal, Volume 150 (2009), no.~1, 1--75.}

\begin{abstract}
  We establish a long exact sequence for Legendrian submanifolds
  $L\subset P \times \rr$, where $P$ is an exact symplectic manifold,
  which admit a Hamiltonian isotopy that displaces the projection of
  $L$ to $P$ off of itself. In this sequence, the singular homology
  $H_\ast$ maps to linearized contact cohomology $CH^{\ast}$, which
  maps to linearized contact homology $CH_\ast$, which maps to
  singular homology. In particular, the sequence implies a duality
  between $\krn(CH_{\ast}\to H_\ast)$ and
  $CH^{\ast}/\img(H_\ast)$. Furthermore, this duality is compatible
  with Poincar{\'e} duality in $L$ in the following sense: the
  Poincar{\'e} dual of a singular class which is the image of $a\in
  CH_\ast$ maps to a class $\alpha\in CH^{\ast}$ such that
  $\alpha(a)=1$.


  The exact sequence generalizes the duality for Legendrian knots in
  $\rr^3$ \cite{duality} and leads to a refinement of the Arnold
  Conjecture for double points of an exact Lagrangian admitting a
  Legendrian lift with linearizable contact homology, first proved in
  \cite{ees:ori}.
\end{abstract}

\maketitle


\section{Introduction}
\label{sec:intro}

Legendrian contact homology, originally formulated in \cite{chv,
  yasha:icm}, is a Floer-type invariant of Legendrian submanifolds
that lies within Eliashberg, Givental, and Hofer's Symplectic Field
Theory framework \cite{egh}.  If $(P, d\theta)$ is an exact symplectic
$2n$-manifold with finite geometry at infinity, then Legendrian
contact homology associates to a Legendrian submanifold $L\subset (P
\times \rr, dz-\theta)$, where $z$ is a coordinate in the
$\rr$-factor, the ``stable tame isomorphism'' class of an associative
differential graded algebra (DGA) $(\alg(L), \df)$.  The algebra is
freely generated by the Reeb chords of $L$ --- that is, integral
curves of the Reeb field $\partial_z$ that begin and end on $L$ ---
and is graded using a Maslov index.  The differential comes from
counting holomorphic curves in the symplectization of $(P\times
\rr,L)$; in the present case, this reduces to a count of holomorphic
curves in $P$ with boundary on the projection of $L$; see
\cite{ees:high-d-analysis, ees:pxr} and below.

In general, it is difficult to extract information directly from the
Legendrian contact homology DGA. An important computational technique
is Chekanov's linearization $(Q(L), \df_1)$ of $(\alg(L), \df)$. The
linearization is defined when $\alg(L)$ admits an isomorphism that
conjugates $\pa$ to a differential $\pa'$ which respects the word
length filtration on $\alg(L)$. In this case, the linearized homology
is the $E_1$-term in the corresponding spectral sequence for computing
the homology of $\pa'$. The linearized contact homology may depend on
the choice of conjugating isomorphism, but the \emph{set} of
isomorphism classes of linearized contact homologies is invariant
under deformations of $L$.  Legendrian contact homology and its
linearized version have turned out to be quite effective tools for
providing obstructions to Legendrian isotopies \cite{chv,
  ees:high-d-geometry, epstein-fuchs, kirill, lenny:computable,
  lenny-lisa}.  These results indicate the so-called ``hard''
properties of Legendrian embeddings.

The main result of this paper describes an important structural
feature of linearized Legendrian contact homology. Say that a
Legendrian submanifold $L \subset P \times \rr$ is \emph{horizontally
  displaceable} if the projection of $L$ to $P$ can be completely
displaced off of itself by a Hamiltonian isotopy.  This condition
always holds if $P = \rr^{2n}$ (or, more generally, if $P = M \times
\cc$) or if $L$ is a ``local'' submanifold that lies inside a Darboux
chart.

\begin{thm} \label{thm:duality} Let $L \subset P^{2n} \times \rr$ be a
  closed, horizontally displaceable Legendrian submanifold whose
  projection to $P$ has only transverse double points. If $L$ is spin,
  then let $\Lambda$ be $\zz$, $\qq$, $\rr$, or $\zz_m$. Otherwise,
  let $\Lambda$ be $\zz_2$.

  If the contact homology of $L$ is linearizable, then the linearized
  contact homology and cohomology and the singular homology of $L$
  with coefficients in $\Lambda$ fit into a long exact sequence:
  \begin{equation}\label{e:seq}
    \cdots \rightarrow H_{k+1}(L) \overset{\sigma_*}{\rightarrow}
    H^{n-k-1}(Q(L)) \rightarrow H_k(Q(L))
    \overset{\rho_*}{\rightarrow}
    H_k(L)\rightarrow
    \cdots.
  \end{equation}
  Furthermore, if $\Lambda$ is a field, if $\la\,,\ra$ is the pairing
  between the homology and cohomology of $Q(L)$, and if $\bullet$ is
  the intersection pairing on homology of $L$, then for $\gamma\in
  H_{n-k}(L)$ and $\alpha\in H_{k}(Q(L))$
  \[
  \la \sigma_\ast(\gamma),\alpha\ra=\gamma\bullet\rho_\ast(\alpha).
  \]  
\end{thm}

A version of Theorem \ref{thm:duality} was described for Legendrian
$1$-knots in $J^{1}(\rr)$ in \cite{duality}, where it was proved that
off of a ``fundamental class'' of degree $1$, the linearized
Legendrian contact homology obeys a ``Poincar\'e duality'' between
homology groups in degrees $k$ and $-k$.  Theorem \ref{thm:duality}
can also be interpreted as a Poincar{\'e} duality theorem for
Legendrian contact homology, up to a fixed error term which depends
only on the topology of the Legendrian submanifold. Any class in
$H_k(Q(L))$ that is in the kernel of $\rho_\ast$ has a dual class in
$H^{n-k-1}(Q(L))$ determined up to the image of $\sigma_\ast$. In
particular, there is a ``duality'' between $\ker (\rho_\ast)$ and
$\text{coker}(\sigma_\ast).$ We call the quotient
$H_k(Q(L))/\ker(\rho_\ast)$ the {\em manifold classes of $L$}. The
last equation of Theorem \ref{thm:duality} implies the following: if
\[
V_{k;0}=\ker(\sigma_\ast)=\img(\rho_\ast)\subset H_{k}(L)
\]
then  
\[
V_{n-k;0}=V_{k;0}^{\perp_{\bullet}},
\]
where $V_{k;0}^{\perp_{\bullet}}$ denotes the annihilator of $V_{k;0}$
with respect to the intersection pairing. In particular, if
$\{\beta_1,\dots,\beta_r\}$ is a basis in $V_{k;0}$ which is extended
to a basis $\{\beta_1,\dots\beta_r,\gamma_1,\dots,\gamma_s\}$ and if
$\{\beta_1',\dots,\beta_r',\gamma_1',\dots,\gamma_s'\}$ is the dual
basis in $H_{n-k}(L)$ under the intersection pairing then
$\{\gamma_1',\dots,\gamma_s'\}$ is a basis in $V_{n-k;0}$. The map
$\rho_\ast$ induces an isomorphism between $V_{k;0}$ and
$H_{k}(L)/\ker(\rho_\ast)$. One may think of this isomorphism as a
correspondence between pairs of Poincar{\'e} dual classes,
$(\beta_j,\beta_j')\in H_k(L)\times H_{n-k}(L)$ and
$(\gamma_l',\gamma_l)\in H_{n-k}(L)\times H_{k}(L)$, and manifold
classes in $H_k(Q(L))/\ker(\rho_\ast)$ and
$H_{n-k}(Q(L))/\ker(\rho_\ast)$, respectively. We prove in
Theorem~\ref{thm:mfld-class} that, over $\zz_2$, the fundamental class
$[L] \in H_n(L)$ is always in the image of $\rho_\ast$. In other
words, in the pair of Poincar{\'e} dual classes $([{\rm point}],[L])$,
it is $[L]$ which is hit by $\rho_\ast$. For Poincar{\'e} dual pairs
of homology classes other than this one, however, it is impossible to
say {\em a priori} which class is in the image of $\rho_\ast$, as
Example~\ref{exam:superspun} shows.

An important application of Legendrian contact homology was to the
Arnold conjecture for Legendrian submanifolds, which states that the
number of Reeb chords of a generic Legendrian submanifold $L \subset
J^{1}(\rr^{n})$ is bounded from below by half the sum of Betti numbers
of $L$. For Legendrian submanifolds with linearizable contact
homology, this conjecture was proved in \cite{ees:ori}. Theorem
\ref{thm:duality} gives the following refinement of that result.


\begin{thm}\label{thm:arnold} Let $L\subset P \times \rr$  and
  $\Lambda$ be as in Theorem~\ref{thm:duality}, with the additional
  assumptions that the first Chern class of $TP$ vanishes and that the
  Maslov number of $L$ equals zero.  Let $c_m$ denote the number of
  Reeb chords of $L$ of grading $m$ and let $b_k$ denote the $k^{\rm
    th}$ Betti number of $L$ over $\Lambda$.  If the Legendrian
  contact homology of $L$ admits a linearization, then
  \[
  c_m+c_{n-m}\ge b_m
  \]
  for $0\le m\le n$.
\end{thm}

The key to the proof of Theorem \ref{thm:duality} is to study the
Legendrian contact homology of the ``two-copy'' link $2L$ consisting
of $L$ and another copy of $L$ shifted high up in the
$z$-direction. Holomorphic disks of $2L$ admit a description in terms
of holomorphic disks of $L$ together with negative gradient flow lines
of a Morse function on $L$. This description of holomorphic disks
allows for a particularly nice characterization of the linearized
contact homology of $2L$ in terms of the linearized contact homology
of $L$ and its Morse homology.  The observation that $2L$ is isotopic
to a link with no Reeb chord connecting different components then
yields the exact sequence.

The paper is organized as follows. In Section~\ref{sec:background}, we
set notation and review the basic definitions and constructions in
Legendrian contact homology. The algebraic framework for the statement
and proof of Theorem \ref{thm:duality} is established in
Section~\ref{sec:algebra} using the description of holomorphic disks
with boundary on $2L$. The analysis necessary for this description is
carried out in Section~\ref{sec:disks}. In Section~\ref{sec:duality},
the duality theorem is established.  Section~\ref{sec:ex} contains the
proof of Theorem \ref{thm:arnold}, and discusses, via examples, the
relationship between the manifold classes in linearized contact
homology and the singular homology of the Legendrian submanifold. In
Appendix~\ref{apdx:framings}, we establish some results related to
gradings in contact homology.

{\em Acknowledgments:} The authors thank Fr\'ed\'eric Bourgeois for
useful discussions which led to the current formulation of our main
result.

\section{Background Notions}
\label{sec:background}

This section
sets terminology and gives a brief review of the basic ideas of
Legendrian submanifold theory and Legendrian contact homology. See
\cite{ees:high-d-analysis, ees:high-d-geometry, ees:ori, ees:pxr} as
well as Section 6 for details.

\subsection{The Lagrangian Projection into $P$}
\label{ssec:projs}
Throughout the paper we let $P$ denote an exact symplectic manifold
with symplectic $2$-form $d\theta$, primitive $1$-form $\theta$, and
finite geometry at infinity. We consider $P\times\rr$ as a contact
manifold with contact $1$-form $dz-\theta$, where $z$ is a coordinate
along the $\rr$-factor.  We use the projection $\Pi_P: P \times \rr
\to P$ and call it the \dfn{Lagrangian projection}. If $L\subset
P\times \rr$ is a Legendrian submanifold then $\Pi_P(L) \subset P$ is
a Lagrangian immersion with respect to the symplectic form
$d\theta$.

As mentioned in the introduction, the Reeb vector field of $dz-\theta$ is
$\pa_z$. Consequently, if $L$ is a Legendrian submanifold then there is a
bijective correspondence between Reeb chords of $L$ and double points
of $\Pi_P(L)$. We will use ``Reeb chord'' and ``double point''
interchangeably, depending on context.

\subsection{Legendrian Contact Homology}
\label{ssec:lch}

The Legendrian contact homology of a Legendrian submanifold $L \subset
P \times \rr$ is defined to be the homology of a differential graded
algebra (DGA) denoted by $(\alg(L), \df)$. In this paper, we will take
the coefficients of $\alg(L)$ to be  $\Lambda$ where  $\Lambda$ is $\zz$, 
$\qq$, $\rr$, or $\zz_n$ if $L$ is spin and $\zz_2$ otherwise. In general, it 
is possible to define the DGA with coefficients in the group ring
$\Lambda[H_1(L)]$.

\subsubsection{The Algebra}
\label{ssec:algebra}
Let $L$ be a
spin\footnote{This condition is unnecessary if we work with $\zz_2$
  coefficients.} Legendrian submanifold of $(P \times \rr, dz-\theta)$
such that all self-intersections of the Lagrangian projection of $L$
are transverse double points; such a Legendrian will be called 
\emph{chord generic}. (Note that if $L \subset P\times\rr$ is any
Legendrian submanifold then there exists arbitrarily small Legendrian
isotopies $L_t$, $0\le t\le1$ such that $L_0=L$ and such that $L_1$ is
chord generic.)  Label the double points of the Lagrangian projection
of $L$ with the set $\qset = \{q_1, \ldots, q_N\}$.  Above each $q_i$,
there are two points $q_i^+$ and $q_i^-$ in $L$, with $q_i^+$ having
the larger $z$ coordinate.

Let $Q(L)$ be the free $\Lambda$-module generated by the set of double
points $\qset$, and let $\alg(L)$ be the free unital tensor algebra of
$Q(L)$.  This algebra should be considered to be a \emph{based}
algebra, i.e.\ the generating set \qset\ is part of the data.  For
simplicity, we will frequently suppress the $L$ in the notation for
$Q(L)$ and $\alg(L)$, and we will write elements of $\alg$ as sums of
words in the elements of \qset.

\subsubsection{The Grading}
\label{ssec:grading}

Before defining the grading, we need to make some preliminary
definitions and choices.  Let $c_1(P)$ denote the first Chern class of
$TP$ equipped with an almost complex structure compatible with the
symplectic form $d\theta$.  Define the \dfn{greatest
divisor} $g(P)$ as follows: if $c_1(P) = 0$, then $g(P) = 0$;
otherwise, if $c_1(P) \neq 0$, then let $g(P)$ be the largest positive
integer such that $c_1(P) = g(P) a$ for some $a \in H^2(P;\zz)$ 
such that $g' a\ne 0$ if $0<g'<g(P)$.

The set of complex trivializations of $TP$ over a closed curve
$\gamma$ in $P$ is a principal homogeneous space over $\zz$. In
particular, given two trivializations $Z$ and $Z'$ of $TP$ along
$\gamma$, there is a well-defined distance $d(Z,Z')\ge 0$ which is the
absolute value of the class in $\pi_1\bigl(GL(k,\cc)\bigr)\cong\zz$
determined by $Z'$ if $Z$ is considered as the reference framing.

\begin{defn}\label{d:Z/gZ-framing} Let $g\in\zz$, $g\ge 0$. A {\em
    $\zz_g$-framing} of $TP$ along a closed curve $\gamma\subset
  P$ is an equivalence class of complex trivializations of $TP$
  along $\gamma$, where two trivializations $Z_0$ and $Z_1$ belong to
  the same equivalence class if $d(Z_0,Z_1)$ is divisible by $g$.
\end{defn}

In Appendix~\ref{apdx:framings}, we show how choices of sections in
certain frame bundles of $TP$ over the $3$-skeleton of some fixed
triangulation of $P$ induce a {\em loop $\zz_g$-framing of $P$} where
$g=g(P)$, i.e., a $\zz_g$-framing of $TP|_\gamma$ for any loop
$\gamma$ in $P$. We also show that such loop $\zz_g$-framings are
unique up to an action of $H^{1}(M;\zz_g)$. If $g=0$, then
$\zz_g$-framings are ordinary framings and the framings are unique up
to action of $H^{1}(M;\zz)$. In the special case $P = T^*M$, it is
straightforward to see that $c_1(P)=0$ and that there is a canonical
loop framing; see Remark~\ref{r:cotangent1}. In what follows, we
assume that a loop $\zz_g$-framing for $P$ has been fixed.

If $\gamma$ is a loop in $L$, then the tangent planes of $L$ give a
loop of Lagrangian subspaces of $TP$ along $\gamma$. Using a
trivialization $Z$ representing the $\zz_g$-framing of $TP|_\gamma$
(see Lemma~\ref{l:indframe}), we get a Maslov index
\[
\mu(\gamma,Z)\in \zz
\]
of the loop of tangent planes along $\gamma$. Since a change of
complex trivialization by one unit changes the Maslov index by $2$
units, this gives a homomorphism
\[ H_1(L;\zz)\to \zz_{2g}.
\] A generator $m(L)$ of the image of this homomorphism is called a
{\em Maslov number of $L$}.

We are now ready to define the grading.  Assume that $L \subset P
\times \rr$ is connected.  For each of the double points in $\qset$,
choose a \dfn{capping path} $\gamma_i$ in $L$ that runs from $q_i^+$
to $q_i^-$.  The Lagrangian projections of tangent planes to $L$ along
$\gamma_i$ gives a bundle of Lagrangian subspaces over $\gamma_i$ in
$TP$.  Pick a trivialization $Z$ of $TP$ over $\gamma_i$ representing
its $\zz_g$-framing. This gives a path $\Gamma_i$ of Lagrangian
subspaces in $\cc^{n}\approx\rr^{2n}$.  Let $\widehat\Gamma_i$ be the
result of closing the path $\Gamma_i$ to a loop using a positive
rotation (see \cite{ees:high-d-geometry}) along the complex angle
between the endpoints of $\Gamma_i$. The \dfn{Conley Zehnder index}
$\nu(\gamma_i,Z)$ is the Maslov index of $\widehat\Gamma_i$. Let
\[
|q_i|_Z=\nu(\gamma_i,Z)-1\in\zz.
\]
The grading $|q_i|$ of $q_i$ is
\begin{equation} \label{eqn:grading}
  |q_i| = \pi(|q_i|_Z\!\!\!\mod{2g})\in\zz_{lcm(2g,m(L))},
\end{equation}
where $\pi\colon\zz_{2g}\to \zz_{lcm(2g,m(L))}$ is the
projection. Note that this grading is independent of the choice of
capping path and of the choice of representative framing $Z$. It does,
however, depend on the choice of loop $\zz_g$-framing.

The grading of a word in $\alg$ is the sum of the gradings of its
constituent letters.  This grading will be extended to the two-copy
$2L$ in Section~\ref{ssec:2-copy}.

\subsubsection{The Differential}
\label{ssec:diffl}
The differential on $\alg$ is a degree $-1$ endomorphism that comes
from counting rigid $J$-holomorphic disks in $P$ with boundary on
$\Pi_P (L)$. Here, we first give a general discussion of
$J$-holomorphic disks and the moduli spaces they constitute. After
that, we give brief definitions of properties of the almost complex
structure $J$ and of $L$ that we will use to prove that the moduli
spaces have certain regularity properties, see Lemmas~\ref{lem:tvmspc}
and~\ref{lem:2pos}. The proofs of these lemmas and more details on the
definitions are found in Subsection \ref{s:cmdli}.  Finally, we define
the differential.

A \dfn{marked disk} $D_{m+1}$ is the unit disk in $\cc$ together with
$m+1$ marked points $\{x_0, x_1, \ldots, x_m \}$ in counter-clockwise
order along its boundary.  Let $\partial \widehat{D}_{m+1}$ be the
boundary of $D_{m+1}$ with the marked points removed.  Over a double
point $q$ of $\Pi_P(L)$ lie two points $q^+$ and $q^-$ in $L$; let
$W^+$ be a neighborhood of $q^+$ in $L$, and similarly for $W^-$.
Given a continuous map $u: (D_{m+1}, \partial D_{m+1}) \to (P,
\Pi_P(L))$ with $u(x_j) = q$, say that $u$ has a \dfn{positive
  puncture} (resp.  \dfn{negative puncture}) at $q$ if, as the
boundary of $D_{m+1}$ near $x_j$ is traversed counter-clockwise, its
image under $u$ lies in $\Pi_P(W^-)$ (resp.  $\Pi_P(W^+)$) before
$x_j$ and in $\Pi_P(W^+)$ (resp.  $\Pi_P(W^-)$) after.

Now we can define the moduli spaces of holomorphic disks with boundary
on $L$. Fix an almost complex structure $J$ on $P$ compatible with
$d\theta$.  For $a, b_1, \ldots, b_k \in \qset$ and $A \in H_1(L)$,
define the moduli space $\ms_A(a; b_1, \ldots, b_k)$, to be the set of
maps $u: D_{m+1} \to P$ that are $J$-holomorphic, i.e. that satisfy
$du + J\circ du\circ i=0$, and so that:
\begin{itemize}
\item The boundary of the punctured disk is mapped to $\Pi_P (L)$,
\item The map $u$ has a positive puncture at $x_0$ with $u(x_0) =
  a$,
\item The map $u$ has negative punctures at $x_j$, $j>0$, with $u(x_j)
  = b_j$, and
\item The restriction $u|_{\partial \widehat{D}_{m+1}}$ admits a
  continuous lift into $L$ which, together with the capping paths
  $-\gamma_a$ and $\gamma_{b_j}$, gives a loop that represents $A$.
\end{itemize}
These disks should be taken modulo holomorphic reparametrization. Note
that
\[
\ms_A(a;b_1,\dots,b_k)=\bigcup_\alpha\ms_A^{\alpha}(a;b_1,\dots,b_k),
\]
where $\alpha$ ranges over the homotopy classes represented by the
disk maps.  Consider a map of the boundary of a punctured disk
corresponding to $A\in H_1(L)$ and fix a homotopy class of disk maps
$\alpha$ with this boundary condition. Fix a trivialization $TP$ along
the boundary of the disk map which extends over the disk and fix
representatives $Z_a, Z_{b_1},\dots, Z_{b_k}$ of the $\zz_g$-framings
of the capping paths of $a,b_1,\dots,b_k$ which agree with this
trivialization at the double points. This gives a trivialization of
$TP$ along the closed-up loop representing $A$. Consider the loop of
Lagrangian planes tangent to $\Pi_P(L)$ along this closed-up loop and
let $\mu(A,\alpha)$ denote its Maslov index measured using the
trivialization just described.


In order to show that moduli spaces of holomorphic disks are nice
spaces we require the Legendrian $L$ and the almost complex structure
$J$ to have certain properties. We use the following notation. As in
\cite{ees:pxr}, we say that an almost complex structure $J$ on $P$,
compatible with $d\theta$, is {\em adjusted to $L$} if there are
coordinate neighborhoods near all double points of $\Pi_P(L)$ where
$J$ looks like the standard complex structure on $\cc^{n}$. We call
such coordinates on a neighborhood of a double point {\em double point
  coordinates}. Given an almost complex structure adjusted to $L$, we
say that $L$ is {\em normalized at crossings} if its Lagrangian
projection consists of linear subspaces near each double
point. Coordinates on $L$ near an endpoint of a Reeb chord that turn
$\Pi_P$ into a linear map into double point coordinates on $P$ are
called {\em Reeb chord coordinates}. Finally, a {\em symplectic
  neighborhood map of $L$} is a symplectic immersion $\Phi$ from a
neighborhood of the $0$-section in $T^{\ast}L$ which extends the
Lagrangian immersion $\Pi_P\colon L\to P$. We say that an almost
complex structure on $P$ is {\em standard in a neighborhood of
  $\Pi_P(L)$} if its pulls back under some symplectic neighborhood map
to an almost complex structure induced by a Riemannian metric $g$ on
$L$; see Remark \ref{rem:acsfrommetr}.

Let $L$ be a chord generic Legendrian submanifold and let $J_0$ be an
almost complex structure adjusted to $L$. Assume that $L$ is
normalized at crossings and that $J_0$ is standard in a neighborhood
of $\Pi_P(L)$. Let $N$ be a small closed regular neighborhood of
$\Pi_P(L)$ which is included in the region where $J_0$ is
standard. Write ${\mathcal J}(N)$ for the space of almost complex
structures, compatible with $d\theta$, which agrees with $J_0$ in $N$
and equip it with the $C^{2}$-topology. The following transversality
result is closely related to Proposition 2.3 in \cite{ees:pxr}.

\begin{lem}\label{lem:tvmspc}
  Let $L$ be a chord generic Legendrian submanifold. Then, after
  arbitrarily $C^{1}$-small deformation of $L$, there exists an almost
  complex structure $J_0$ adjusted to $L$, $L$ is normalized at
  crossings, and $J_0$ is standard in a neighborhood of
  $\Pi_P(L)$. Furthermore there exists a closed neighborhood $N$ of
  $\Pi_P(L)$ such that for an open dense set of almost complex
  structures $J\in {\mathcal J}(N)$ the moduli space
  $\ms_A^{\alpha}(a;b_1,\dots,b_k)$ of $J$-holomorphic disks is a
  transversely cut out manifold of dimension $d$, where
  $$d=|a|_{Z_a}-\sum_j|b_j|_{Z_{b_j}}+\mu(A,\alpha)-1.$$
\end{lem}
Lemma \ref{lem:tvmspc} is proved in Subsection \ref{s:cmdli}.


We are now ready to discuss the moduli spaces involved in the
definition of the differential. Fix an almost complex structure $J$ so
that Lemma \ref{lem:tvmspc} holds. (Later, in order to deal with the
two-copy $2L$ of $L$, see Subsection \ref{ssec:2-copy}, we also need
to fix an Morse function $f$ so that Lemma \ref{lem:2pos} discussed
below holds.)

If $L$ is spin, the moduli spaces can be consistently oriented (see
\cite{ees:ori}).  If $\ms_A(a;b_1, \ldots, b_k)$ is zero dimensional,
then by Gromov compactness it is compact and the algebraic count of
points $\# \ms_A(a; b_1, \ldots, b_k)$ makes sense.  This count allows
us to define the differential on generators as follows:
\begin{equation}
  \df a = \sum_{\dim(\ms_A(a; b_1, \ldots, b_k))=0} \bigl(\# \ms_A(a; b_1, \ldots, b_k)
  \bigr) b_1 \cdots b_k.
\end{equation}
We then extend the differential to $\alg(L)$ via linearity and the
Leibniz rule.

The following lemma combines the definition of $\df$ with
Stokes' Theorem.

\begin{lem} \label{lem:stokes} Let $\ell(q_i)$ be the length of the Reeb
  chord lying above $q_i$.  If $u \in \ms(a; b_1, \ldots, b_k)$, then:
  \begin{equation*}
    \ell(a) - \sum_k \ell(b_j)\ge C\area(u) > 0,
  \end{equation*}
  for some constant $C>0$.
\end{lem}

The central result of the theory is the following theorem.

\begin{thm}[\cite{ees:high-d-analysis, ees:high-d-geometry, ees:ori,
    ees:pxr}] \label{thm:lch} The differential \df\ satisfies $\df^2 =
  0$ and the ``stable tame isomorphism class'' (and hence the
  homology) of the DGA $(\alg, \df)$ is invariant under Legendrian
  isotopy.
\end{thm}

See \cite{chv} for the definition of ``stable tame isomorphism''; we
will need only some straightforward consequences of the definition,
not its precise formulation, in this paper.

\subsubsection{Another moduli space}

We will also need to consider moduli spaces of rigid holomorphic disks
with boundary on $L$ with exactly two positive punctures $a_1,a_2$ and
an arbitrary number of negative punctures. There are new
transversality issues in this case, as the setup allows for multiple
covers. In order to deal with such issues, we study closely related
holomorphic disks, which in a sense give resolutions of multiple
covers. More precisely, we take two copies $L_0$ and $L_1$ of $L$ and
let $\ms_A(a_1,a_2;b_1,\dots,b_k; c_1, \ldots c_l)$ denote the moduli
space of holomorphic disks which are maps $u: S\to P$, where $S$ is a
strip $\rr\times [0,1]$ with boundary punctures $x_1,\dots,x_k$
appearing in order along $\rr \times \{0\}$ and $y_1, \ldots, y_l$
appearing in order along $\rr \times \{1\}$ with the following
properties.  The map $u$ takes the puncture at $-\infty$ to $a_1$, where $a_1$ is a Reeb chord from $L_1$ to $L_0$, the puncture at $+\infty$ to $a_2$ where $a_2$ is a Reeb chord from $L_0$ to $L_1$, 
the puncture $x_j$ to $b_j$, where the $b_j$'s  Reeb chords from $L_0$ to $L_0$, and the puncture
$y_j$ to $c_j$, where the $c_j$'s are Reeb chords from $L_1$ to $L_1$. It has positive punctures at $\pm\infty$ and negative punctures elsewhere. It takes the boundary components in
$\rr\times\{0\}$ to $L_0$ and those in $\rr\times\{1\}$ to $L_1$, and
$A$ encodes the homology class of the boundary data as before. In
complete analogy with the one-punctured case, we have
\[
\ms_A(a_1,a_2;b_1,\dots,b_k; c_1, \ldots
c_l)=\bigcup_\alpha\ms_A^{\alpha}(a_1,a_2;b_1,\dots,b_k; c_1, \ldots
c_l),
\]
where $\alpha$ ranges over the homotopy classes of disk maps. Again,
we choose a trivialization $TP$ for the boundary of the disk which
extends over the disk and is compatible with trivializations $Z_{c}$
representing the $\zz_g$-framings of the capping paths of the Reeb
chords of the disk.

In order to achieve transversality for $\ms_A(a_1,a_2;b_1,\dots,b_k;
c_1, \ldots c_l)$, we push $L_1$ off of $L=L_0$.  To state this more
precisely, we require $L$ to be normalized at crossings with respect
to some almost complex structure $J$ adjusted to $L$. Let $f\colon L\to
\rr$ be a Morse function. We say that $f$ is {\em admissible} if its
critical points lie outside Reeb chord coordinate neighborhoods in $L$
and if $f$ is real analytic in Reeb chord coordinates near every Reeb
chord endpoint.

Fix an admissible Morse function $f\colon L\to\rr$ which is
sufficiently small so that the graph of $df$ in $T^{\ast}L$ lies in
the domain of definition of a symplectic neighborhood map $\Phi$ of
$L$. Let $L_1(f)$ be the image of the graph of $df$ under $\Phi$. We
write $\ms_A(a_1,a_2;b_1,\dots,b_k; c_1, \ldots c_l;f)$ for the moduli
space of $J$-holomorphic curves with boundary on $L_0\cup L_1(f)$ and
punctures as described above.  We have the following result.

\begin{lem}\label{lem:2pos}
  Let $\hat f$ be a sufficiently small admissible Morse function and
  let $k\ge 2$. Then there are admissible Morse functions $f$
  arbitrarily $C^{k}$-close to $\hat f$ such that if
  $$d=|a_1|_{Z_{a_1}}+|a_2|_{Z_{a_2}}-\sum_{j=1}^k|b_j|_{Z_{b_j}}
  -\sum_{j=1}^l|c_j|_{Z_{c_j}} - n + 1 + \mu(A,\alpha) \le 0,$$ then
  $\ms_A^{\alpha}(a_1,a_2;b_1,\dots,b_k,c_1,\dots, c_l)$ is a
  transversely cut out manifold of dimension $d$.
\end{lem}
Lemma~\ref{lem:2pos} is proved in Subsection~\ref{s:cmdli}.


\begin{rem}
  It is essential in the proof of Lemma \ref{lem:2pos} that the
  perturbations near one of the positive punctures are independent
  from those at the other. This is not true for multiply covered disks
  if the boundary of the disk lies on a single immersed Lagrangian
  $L$, but using two Lagrangian submanifolds resolves this problem:
  since the Reeb chord $a_1$ starts on $L_1$ and ends on $L_0$ and the
  Reeb chord $a_2$ starts at $L_0$ and ends on $L_1$, there are no
  multiply covered disks in $\ms(a_1,a_2;b_1,\dots,b_k;c_1,\dots
  c_r)$. Furthermore, since both positive punctures map to mixed Reeb
  chords, ``boundary bubbling'' is not possible for topological
  reasons, i.e., broken disks which arises as limits of a sequence of
  disks in $\ms(a_1,a_2;b_1,\dots,b_k;c_1,\dots c_r)$ must be joined
  at Reeb chords, not at boundary points; see \cite[Definition
  2.1]{ekholm:rat-sft-ex-lag}.
\end{rem}

\begin{rem}
  Note that Lemma \ref{lem:2pos} gives one way of resolving the
  transversality problems arising from multiply covered disks. In
  order to compute the number of disks which would arise on the
  one-copy version, one would have to study the obstruction bundles
  over certain multiply covered disks on $L$.
\end{rem}

\subsection{Linearization}
\label{ssec:linearization}

Legendrian contact homology is not an easy invariant to use, as it is
a non-commutative algebra given by generators and relations.  One
important tool in extracting useful information from Legendrian
contact homology is Chekanov's \dfn{linearized contact homology}
\cite{chv}.  To define it, break the differential \df\ into components
$\df = \sum_{r=0}^\infty \df_r$, where the image of $\df_0$ lies in
the ground ring and $\df_r$ maps $Q$ to $Q^{\otimes r}$.  If $\df_0 =
0$ --- that is, if there are no constant terms in the differential
\df\ --- then the equation $\df^2=0$ implies that $\df_1^2=0$ as well,
and hence that $(Q, \df_1)$ is a chain complex in its own right.

Even if the DGA does not satisfy $\df_0=0$, it may be tame isomorphic
to one that does; in this case, we call the DGA \dfn{good}. The
existence of such an isomorphism is equivalent to the existence of an
\dfn{augmentation} of the DGA $(\alg, \df)$, i.e.\ a graded algebra
map $\aug$ from the algebra to the ground ring such that $\aug(1) = 1$
and $\aug \circ \df = 0$.  To see why this is true, define an
isomorphism $\Phi^\aug: \alg \to \alg$ on the generators of \alg\ by:
\begin{equation}
  \Phi^\aug(q_i) = q_i + \aug(q_i).
\end{equation}
This map conjugates $(\alg, \df)$ to a good DGA $(\alg,
\df^\aug)$.  It is easy to check that $\df^\aug$ may be
computed from \df\ by replacing $q_i$ with $q_i + \aug(q_i)$ in
the expression for the differential. 

For any given augmentation $\aug$, it is straightforward to compute
the homology $H_*(Q, \df^\aug_1)$.  The set $\mathcal{H}(\alg, \df)$
of all such ``linearized'' homologies taken with respect to all
augmentations of $(\alg, \df)$ is an invariant of the stable tame
isomorphism class of $(\alg, \df)$ \cite{chv}.  It is this set of
linearized homologies for which we will prove duality.

Below, if $\epsilon$ is an augmentation and if $q\in Q$ is a generator then we say that $q$ is {\em augmented} if $\epsilon(q)\ne 0$. 

\subsection{Basic Examples}
\label{ssec:basic-ex}

Before proceeding to the proof of duality, let us look at a few basic
examples in $J^1(\rr^n) = \rr^{2n+1}$.  In these examples, we use the
\dfn{front projection} $\Pi_F: J^1(\rr^n) \to J^{0}(\rr^n)=
\rr^{n+1}$. In the front projection, a Reeb chord occurs when two
tangent planes to $\Pi_F (L)$ over some point in $\rr^n$ are parallel.

In terms of the front projection, the Conley-Zehnder index of a
capping path $\gamma_i$ is calculated as follows. Let $x_i\in \rr^n$
be the projection of the double point $q_i$ to $\rr^n$.  Around
$q_i^\pm$, the front projection $\Pi_F(L)$ is the graph of functions
$h_i^\pm$ with domain in a neighborhood of $x_i$.  Define the
difference function for $q_i$ to be $h_i = h_i^+ - h_i^-$.  Since the
tangent planes to $\Pi_F(L)$ are parallel at $q_i^\pm$, the difference
function $h_i$ has a critical point at $x_i$ and, if $L$ is chord
generic, the Hessian $d^2h_i$ is non-degenerate at $x_i$.  For the
proof of the following lemma, see \cite{ees:high-d-geometry}.

\begin{lem} \label{lem:front-grading}
  If $\gamma_i$ is a generic capping path for $q_i$, then
  \begin{equation}\label{e:frontnu}
    \nu(\gamma_i) = D(\gamma_i) - U (\gamma_i) + I_{h_i}(x_i),
  \end{equation}
  where $D(\gamma_i)$ (resp.  $U(\gamma_i)$) is the number of cusps
  that $\gamma_i$ traverses in the downward (resp. upward)
  $z$-direction and $I_{h_i}(x_i)$ is the Morse index of $h_i$ at $x_i$.
\end{lem}

\begin{exam}[The flying saucer]\label{ex:flying-saucer} The front
  projection of the simplest Legendrian submanifold of $\rr^{2n+1}$ is
  shown in Figure~\ref{fig:flying-saucer}.  This Legendrian $L$ is an
  $n$-sphere and $\Pi_P(L)$ has exactly one double point $c$. The
  grading on this double point is $|c|=n.$ Thus, $\alg(L)$ is the
  $\zz$-algebra generated by $c$. Moreover, for $n \geq 2$, one easily
  sees that $\df c=0$ for degree reasons. Thus, the differential is
  good and the linearized contact homology is:
  \begin{equation}
    H_k(Q(L))=\begin{cases}
      \zz & \text{if $k=n$, }\\
      0& \text{otherwise}.
    \end{cases}
  \end{equation}
\end{exam}

\begin{figure}
  {\epsfxsize=4in\centerline{\epsfbox{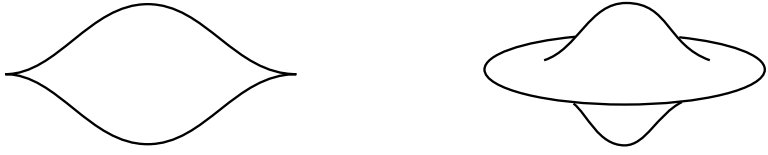}}}
  \caption{Front projection of the Flying Saucer in dimension 3, on
    the left, and 5, on the right.}
  \label{fig:flying-saucer}
\end{figure}

\begin{exam} \label{ex:stab-ex} Consider the two flying saucers in
  Figure~\ref{fig:stab-ex}.  Connect them by a curve $c(s)$
  parametrized by $s \in [-1,1]$ that runs from the cusp edge of one
  flying saucer to the cusp edge of the other.  Take a small
  neighborhood of $c$ whose cross-sections are round balls whose radii
  decrease from $s=-1$ to $s=0$, and increase from $s=0$ to $s=1$.
  Finally, introduce cusps along the sides of the neighborhood and
  join it smoothly to the two flying saucers.  If $n$ is even, the
  resulting Legendrian sphere $L$ has the same classical invariants as
  the flying saucer.  Note that this is Example 4.12 of
  \cite{ees:high-d-geometry}.
  \begin{figure}[ht]
    \relabelbox \small {\epsfxsize=4.2in\centerline{\epsfbox{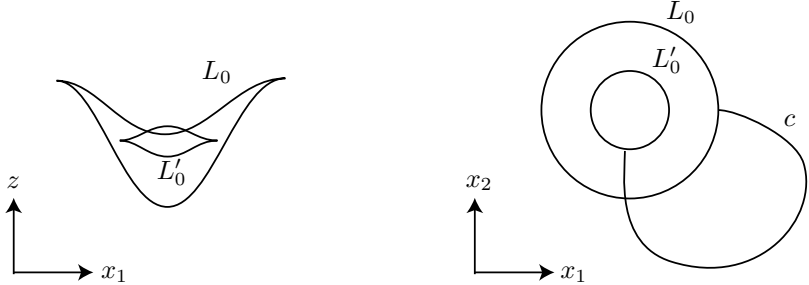}}}
    \relabel{1}{$x_1$}
    \relabel{2}{$x_2$}
    \relabel{3}{$x_1$}
    \relabel{z}{$z$}
    \relabel{0}{$L_0$}
    \relabel{p}{$L_0'$}
    \relabel{t}{$L_0$}
    \relabel{l}{$L_0'$}
    \relabel{c}{$c$}
    \endrelabelbox
    \caption{On the left side, the $x_1z$-slice of part of $L_1$ is
      shown. To see this portion in $\rr^3$, rotate the figure about
      its center axis. On the right side, we indicate the arc $c$
      connecting the two copies of $L_0.$}
    \label{fig:stab-ex}
  \end{figure}

  The connecting tube can be chosen such that there are six Reeb
  chords $q_1, \ldots, q_6$ that come from the two flying saucers and
  exactly one more Reeb chord $q$ at the center of the connecting
  tube.  The degrees of these chords are:
  \begin{align*}
    |q_1| &= |q_2| = |q_5| = n \\
    |q_3| &= |q_6| = |q| = n-1 \\
    |q_4| &= 0.
  \end{align*}
  For $n>2$, the fact that there are no chords of degree $1$ implies
  that $\alg(L)$ must be a good algebra.  Further, it is clear that
  $H_0(Q(L)) \cong \zz$.  We will calculate the remainder of the
  linearized cohomology using Theorem~\ref{thm:duality} in
  Section~\ref{ssec:explan-ex}.
\end{exam}

\begin{rem}
  In contrast to dimension three, where DGAs with $\df_0 = 0$ almost
  never arise directly from the geometry of a Legendrian knot, this is
  not infrequently the case in higher dimensions.  In dimension three,
  however, the theory of the existence of augmentations has been
  well-studied (see \cite{fuchs:augmentations, fuchs-ishk,
    ns:augm-rulings, rutherford:kauffman, rulings}), while less work
  has been done in higher dimensions.  One exception is Ng's use of
  augmentation ideals in his study of Knot Contact Homology
  \cite{lenny:knots3}, which is based on the technology of Legendrian
  contact homology.
\end{rem}

\section{Algebraic Framework}
\label{sec:algebra}

In this section we present the algebraic setup for the linearized
contact homology of the two-copy $2L$ of a Legendrian submanifold
$L\subset P$. The chain complex $Q(2L)$ encodes the chain complexes
for the linearized contact homology, the linearized contact
cohomology, and the Morse homology of a function $f$ on $L$.  Similar
setups for the two-copy of a Legendrian submanifold were used in
\cite{ees:ori} for the proof of double-point estimates of exact
Lagrangians as well as in \cite{duality} for the proof of duality for
Legendrian contact homology in three dimensions.

\subsection{The Two-Copy Algebra}
\label{ssec:2-copy}

Let $\phi_t\colon P\times\rr\to P\times\rr$ denote the time $t$ Reeb
flow. Note that $\phi_t(p,t_0)=(p,t_0+t)$.  The \dfn{two-copy}
$\twocopy$ of a Legendrian submanifold $L \subset P \times \rr$ is the
Legendrian link composed of $L$ and $\phi_s(L)$, where $s\gg 0$.  The
Reeb chords of this link are degenerate: at every point of $L$ which
is not a Reeb chord endpoint, there starts a Reeb chord ending on
$\phi_s(L)$, while at every Reeb chord endpoint of $L$, there start
two such chords. This degenerate situation is similar to a Morse-Bott
degeneration in Morse theory. We use the following perturbation of
$\phi_s(L)$ in order to get back into a generic situation. Let $f$ be
a $C^1$-small admissible Morse-Smale function on $L$.
Perturb $\phi_s(L)$ to $df$ inside
$T^*\phi_{s}{L}$, and then use the symplectic neighborhood map to
transfer the result to $P$.  Finally, lift the result back to $P
\times \rr$. We call this perturbed
Legendrian submanifold $\tilde{L}$.  Note that the $z$-coordinate of a
point $p$ in the new $\tilde{L}$ differs by $f(p)$ from the
corresponding point in $\phi_s(L)$.

The generators of $\alg(\twocopy)$ come from two sources. First, for
every double point $q$ of $\Pi_P (L)$, there are four double points of
$\Pi_P(\twocopy)$: two copies of the original double point of
$\Pi_P(L)$, one denoted $q^0$ in $\Pi_P(L)$ and one denoted
$\tilde{q}^{0}$ in $\Pi_P(\tilde{L})$, and two intersections between
$\Pi_P(L)$ and $\Pi_P (\tilde{L})$.  Second, each critical point of
$f$ gives a double point of $\Pi_P(2L)$. Thus, we can split the
generators of the two-copy DGA into four types:
\begin{description}
\item[Pure generators] Double points $q^0$ of $\Pi_P (L)$ and their
  nearby counterparts $\tilde{q}^{0}$ for $\Pi_P(\tilde{L})$,
\item[Mixed $q$ generators] Intersections $q^1$ between $\Pi_P(L)$
  and $\Pi_P (\tilde{L})$ so that $(q^1)^-$ lies near $(q^0)^-$ and
  $(q^1)^+$ lies near $(\tilde{q}^0)^+$,
\item[Mixed $p$ generators] Intersections $p^1$ between $\Pi_P(L)$
  and $\Pi_P (\tilde{L})$ so that $(p^1)^-$ lies near $(q^0)^+$ and
  $(p^1)^+$ lies near $(\tilde{q}^0)^-$, and
\item[Mixed Morse generators] Double points $c^1$ corresponding to
  critical points of $f$.
\end{description}

\begin{figure}[ht]
  \relabelbox \small {\centerline{\epsfbox{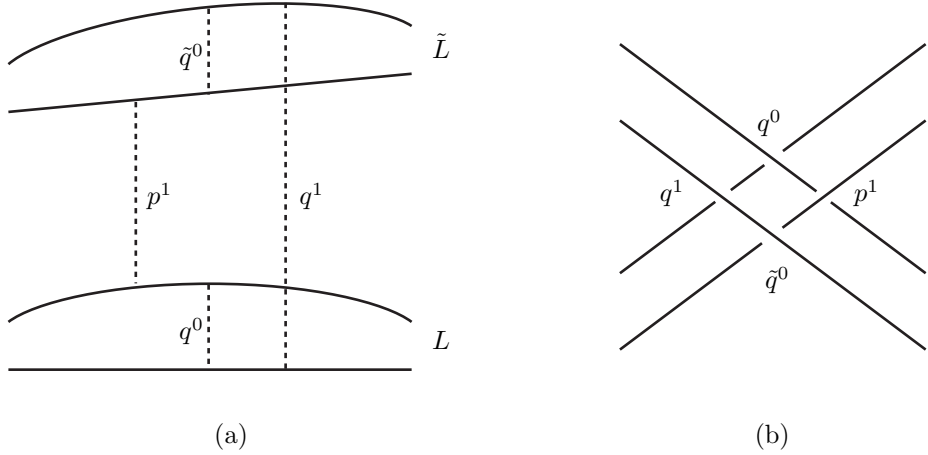}}}
  \relabel{1}{$\tilde{q}^0$}
  \relabel{2}{$p^1$}
  \relabel{3}{$q^1$}
  \relabel{4}{${q}^0$}
  \relabel{5}{$L$}
  \relabel{6}{$\tilde{L}$}
  \relabel{7}{(a)}
  \relabel{8}{(b)}
  \relabel{9}{$\tilde{q}^0$}
  \relabel{10}{$p^1$}
  \relabel{11}{$q^1$}
  \relabel{12}{${q}^0$}
  \endrelabelbox
  \caption{Three of the four types of generators of the two-copy
    algebra, as seen from (a) the front projection of \twocopy\ and
    (b) the Lagrangian projection of \twocopy\ for a knot in the
    standard contact $\rr^3$.}
  \label{fig:generators}
\end{figure}

The first three types of generators are pictured schematically in
Figure~\ref{fig:generators}. For a small perturbation $f$, we have the
following relationships between the lengths $\ell$ of each type of
chord:
\begin{equation} \label{eqn:length}
  \begin{split}
    \ell(q^1) &\approx s + \ell(q^0), \\
    \ell(c^1) &\approx s, \\
    \ell(p^1) &\approx s - \ell(q^0).
  \end{split}
\end{equation}


To define the grading on $\alg(\twocopy)$, we follow the relative
grading of \cite{chv, lenny:computable}. First, choose a basepoint $p
\in L$ away from the Reeb chords of \twocopy\ and let $\tilde{p}$ be
its $z$-translate in $\tilde{L}$. Fix an identification of $T_pL$ and
$T_{\tilde{p}}\tilde{L}$.  Next, choose capping paths $\gamma_i$ for
the Reeb chords of $L$ that run through the marked point $p$ and
denote by $\gamma_i^+$ and $\gamma_i^-$ the paths $p$ splits
$\gamma_i$ into, where $\gamma^{+}_{i}$ ($\gamma^{-}_{i}$) contains
the starting point (ending point) of $\gamma_{i}$. Finally, choose
capping paths $\tilde{\gamma}_i$ for the double points on $\tilde{L}$
to be $z$-translates of these (with small deformations near the
endpoints when necessary).  Given suitable trivializations of $TL$
($T\tilde{L}$) along $\gamma_i$ ($\tilde{\gamma}_i$), we can then
define loops $\hat{\Gamma}^\pm_i$ of Lagrangian planes in $\rr^{2n}$
by concatenating $\Gamma^+_i$ and $\Gamma^-_i$ and closing up as
before.

It is obvious that the $q^0$ ($\tilde{q}^{0}$) generators inherit
their gradings from $\alg(L)$.  Using the concatenation property of
the Conley-Zehnder index \cite{robbin-salamon:maslov}, it is easy to
see that
\begin{equation}\label{eqn:q1-grading}
  |q^1_i| = |q^0_i|.
\end{equation}
The loops $\hat{\Gamma}^\pm_i$ for the $p^1$ generators are simply the
reverses of the loops for the $q^1$ generators, so $\nu(p^1_i) =
n-\nu(q^1_i)$.  It follows that:
\begin{equation} \label{eqn:p-grading}
  |p^1_i| = -|q^1_i|+n-2.
\end{equation}
For the gradings of the $c^{1}$ generators, we work in local
coordinates around the critical point in $L$, where we can consider
$L$ to be the $1$-jet of the constant function and $\tilde{L}$ to be
the $1$-jet of $f$. Choosing $\gamma^+$ to be the reverse of a small
perturbation of the translate of $\gamma^-$, we can contract $\Gamma$
into this neighborhood and use the proof of
Lemma~\ref{lem:front-grading} in \cite{ees:high-d-geometry} to obtain:
\begin{equation}\label{eqn:c-grading}
  |c^1_i| = \ix_{c_i}(f) - 1.
\end{equation}

In order to distinguish the differential $\pa$ on $\alg(L)$ from the
differential on $\alg(2L)$, we will denote the latter by
\fdf. Furthermore, provided the Morse function $f$ is sufficiently
small, the algebras $\alg(L)$ and $\alg(\tilde L)$ are canonically
isomorphic DGAs. We will use this identification without further
comment below.

\subsection{The Linearized Complex}
\label{ssec:lin-2-copy}

Assume that the contact homology of $L\subset P \times \rr$ is
linearizable and let $\aug\colon\alg(L)\to\Lambda$ be an
augmentation. Define the algebra map $\faug: \alg(\twocopy) \to
\Lambda$ on the generators by:
\begin{equation}
  \faug(x) = \begin{cases}
    \aug(x) & x \text{ is a } q^0 \text{ or } \tilde{q}^0 \text{ generator,} \\
    0 & \text{otherwise}.
  \end{cases}
\end{equation}
Then \faug\ is an augmentation of $(\alg(\twocopy), \fdf)$. To see
this, first note that Lemma~\ref{lem:stokes} and equation
(\ref{eqn:length}) imply that the differential of any $q^0$ or
$\tilde{q}^{0}$ generator agrees with the original differential
\df. In particular,
\[
\faug \circ \fdf(q^0)=\faug \circ \fdf(\tilde{q}^0)=0.
\]
Second, note that every term in the differential of a mixed Reeb chord
must contain at least one more mixed chord. This implies:
\[
\faug \circ \fdf(q^1) = \faug \circ \fdf(p^1)=\faug \circ \fdf(c^1)=0,
\]
and the augmentation property follows.

The techniques just used generalize to the following lemma.

\begin{lem} \label{lem:one-punc} If a holomorphic disk with boundary
  on $2L$ has one positive puncture and its positive puncture maps to
  a mixed Reeb chord then it has exactly one negative puncture mapping
  to a mixed Reeb chord.  If a holomorphic disk with boundary on $2L$
  has one positive puncture and its positive puncture maps to a pure
  Reeb chord then all its negative punctures map to pure Reeb chords as
  well.
\end{lem}

\begin{proof}
  Immediate from Lemma~\ref{lem:stokes} and Equation
  (\ref{eqn:length}).
\end{proof}

Let us continue to explore the consequences of Lemma~\ref{lem:stokes}.
Consider the chain complex $Q(\twocopy)$ of the linearized contact
homology of $\twocopy$ with the differential $\fdfe_1$.  We decompose
$Q(\twocopy)$ into four summands
\[
Q(\twocopy)= Q^1 \oplus C^1 \oplus P^1 \oplus Q^0;
\]
where $Q^{1}$ is generated by mixed $q$ generators, $C^{1}$ by mixed
Morse generators, $P^{1}$ by mixed $p$-generators, and $Q^{0}$ by pure
generators. Some components of the matrix of $\fdfe_1$ corresponding
to this decomposition will turn the summands into chain complexes in
their own right, while other components will become maps between these
complexes. In the next section, we will reinterpret the complexes
corresponding to the summands as the original linearized complex, its
cochain complex, and the Morse-Witten complex of $L$ with respect to
the Morse function $f$.  The maps between these complexes will give
the exact sequence of Theorem \ref{thm:duality}.

\begin{rem}
  For simpler notation below, we will suppress the augmentations from
  the notation for the differential writing simply $\fdf_1$ and
  $\df_1$ instead of $\fdfe_1$ and $\dfe_1$, respectively.
\end{rem}

We next consider the decomposition of the differential corresponding
to the direct sum decomposition above.  Lemma~\ref{lem:one-punc}
implies that $Q^0\subset Q(2L)$ is a subcomplex and that the component
of $\fdf_1$ which maps $Q^1 \oplus C^1 \oplus P^1$ to $Q^0$
vanishes. Thus $Q^1 \oplus C^1 \oplus P^1$ is a subcomplex as well. It
is the subcomplex $Q^1 \oplus C^1 \oplus P^1$ which carries the
important information for the the exact sequence. We therefore
restrict attention to this subcomplex and leave $Q^{0}$ aside.

Lemma~\ref{lem:stokes} and Equation \eqref{eqn:length} imply that the
linearized differential takes on the following lower triangular form
with respect to the splitting $Q^1 \oplus C^1 \oplus P^1$:
\begin{equation} \label{eqn:small-d}
  \fdf_1 = \begin{bmatrix}
    \fdf_q & 0 & 0  \\
    \rho & -\fdf_c & 0  \\
    \eta & \sigma & \fdf_p
  \end{bmatrix}.
\end{equation}

Since $\fdf_1^2 = 0$, the diagonal entries in the matrix of
$\fdf_1^{2}$ give
\begin{equation} \label{eqn:diag-d}
  \fdf_q^2 = 0,  \qquad \fdf_c^2 = 0,\quad \text{and}\quad  \fdf_p^2 = 0.
\end{equation}
The subdiagonal entries give
\begin{equation} \label{eqn:rho-chain} \rho \fdf_q - \fdf_c \rho = 0
  \quad \text{and} \quad\fdf_p \sigma - \sigma \fdf_c = 0.
\end{equation}
The last interesting entry in the matrix for $\fdf_1^2$ is
\begin{equation} \label{eqn:rho-eta-0} \eta \fdf_q + \fdf_p \eta +
  \sigma \rho = 0.
\end{equation}

\begin{prop} \label{prop:small-subcomplexes} With notation as above,
  we have:
  \begin{enumerate}
  \item $(Q^1, \fdf_q)$, $(C^1,
    \fdf_c)$, and $(P^1, \fdf_p)$ are all chain complexes.
  \item The maps $\rho$ and $\sigma$ are chain maps of degree
    $-1$.
  \item $\sigma_* \rho_* = 0$, where $\eta$ acts as a chain
    homotopy between $\sigma \rho$ and the zero map.
  \end{enumerate}
\end{prop}
\begin{proof}
  This follows from \eqref{eqn:diag-d}, \eqref{eqn:rho-chain}, and
  \eqref{eqn:rho-eta-0}.
\end{proof}

Combining $Q^1$ and $C^1$ into $QC^1 = Q^1 \oplus C^1$ gives another
useful view of $\fdf_1$.  Let $\fdf_{qc} = -\fdf_q - \rho + \fdf_c$
and let $H = \eta + \sigma$.  With respect to the splitting $QC^1
\oplus P^1$, then, $\fdf_1$ takes the following form:
\begin{equation}
  \fdf_1 = \begin{bmatrix}
    -\fdf_{qc} & 0 \\
    H & \fdf_p
  \end{bmatrix}.
\end{equation}
In parallel to Proposition~\ref{prop:small-subcomplexes}, we obtain
the following proposition.

\begin{prop} \label{prop:large-subcomplexes}
  With notation as above we have:
  \begin{enumerate}
  \item $(QC^1, \fdf_{qc})$ is a chain complex.
  \item $H: (QC^1, \fdf_{qc}) \to (P^1, \fdf_p)$ is a chain map of
    degree $-1$.
  \end{enumerate}
\end{prop}

Another way to look at the Proposition \ref{prop:large-subcomplexes}
is as follows:

\begin{cor} \label{cor:mapping-cone} The complex $(QC^1 \oplus P^1,
  -\fdf_{qc} + H + \fdf_p)$ is the mapping cone of $H$. Further,
  $QC^1$ itself is the mapping cone of $\rho$.
\end{cor}

As we shall see, the mapping cone of $H$ is acyclic and hence $H$ is
an isomorphism. In fact, we will use this isomorphism to place the
linearized contact cohomology inside the exact sequence.

\subsection{Identifying the Complexes in $Q(\twocopy)$}
\label{ssec:subcomplexes}

In this section, we identify the subcomplexes of $Q(2L)$ discussed in
the previous section with the original linearized chain complex of
$L$, its cochain complex, and the Morse-Witten complex of $L$. The
proofs of these identifications rest on an analytic theorem which
describes all rigid holomorphic disks for a sufficiently small
admissible perturbation function $f$ satisfying certain technical
conditions. We will describe that theorem here, but defer the detailed
analytic treatment necessary for its proof to Section~\ref{sec:disks}.

\subsubsection{Generalized Disks}
\label{ssec:gen-disks}

In order to state the analytic
theorem for disks of the two-copy, we need to introduce terminology
for objects built from holomorphic disks with boundary on $L$ and flow
lines of the Morse function $f$ on $L$ used to shift $\phi_s(L)$ to $\tilde L$.

First, we make some assumptions.  Let $J$ be an almost complex
structure adjusted to $L$ and let $L$ be normalized at double
points. Assume
that $J$ is standard in a neighborhood of $\Pi_P(L)$ with respect to a
Riemannian metric $g$ on $L$. Let $f$ be an admissible Morse function
which is round at critical points with respect to $g$; see $({\bf
  a5})$ in Subsection \ref{s:morseflow}. We call triples $(f,g,J)$
with these properties \dfn{adjusted to} $L$; see Subsection
\ref{s:morseflow}. (The condition that the Morse function be round at
critical points is inessential and made only in order to simplify some
technicalities in proofs.)
 
Assume that the conclusions of Lemma \ref{lem:tvmspc} hold for
$J$-holomorphic disks with boundary on $L$. Then as explained in
Subsection \ref{s:cmdli}, the compactified moduli space
$\overline{\ms}$ of $J$-holomorphic disks with one positive puncture
is a $C^{1}$-manifold with boundary with corners. Furthermore, as
explained in Subsection \ref{s:evmaps}, the corresponding compactified
moduli space $\overline{\ms^{\ast}}$ of $J$-holomorphic disks with an
additional marked point on the boundary comes equipped with a $C^1$
evaluation map $\ev\colon \overline{\ms^{\ast}}\to L$.  Further,
assume that the admissible perturbing Morse function $f$ and the
Riemannian metric $g$ on $L$ are Morse-Smale and that, together with
the almost complex structure $J$, they have the property that the
stable and unstable manifolds of the critical points of $f$, defined
using the $g$-gradient of $f$, are stratumwise transverse to the
evaluation map $\ev\colon\overline{\ms^{\ast}}\to L$. For existence of
adjusted triples $(f,g,J)$ which satisfy these transversality
conditions, see Lemma \ref{lem:morsetv}.

We define a \dfn{generalized disk} to be a pair $(u, \gamma)$
consisting of a holomorphic disk $u \in \ms=\ms_A(a;b_1,\dots,b_k)$
with boundary on $L$ and a negative gradient flow line $\gamma$ of the
Morse function $f$ beginning or ending at the boundary of $u$; the
other end of $\gamma$ must lie at a critical point of $f$. We call the
point on the boundary of $u$ where the flow line $\gamma$ begins or
ends the \dfn{junction point} of $(u,\gamma)$. If the flow line starts
at the junction point then $(u,\gamma)$ has a \dfn{negative Morse
  puncture}; otherwise it has a \dfn{positive Morse puncture}.  The
\dfn{formal dimension} $\dim((u,\gamma))$ of a generalized disk
$(u,\gamma)$ as follows.
\begin{equation} \label{eqn:gen-formal-dim}
  \dim((u,\gamma))=
  \begin{cases} \dim \ms +1 + (I_f(p)-n), &\text{$p$ a
      positive Morse puncture,}\\ \dim \ms +1 - I_f(p), &\text{$p$
      a negative Morse puncture,}
  \end{cases}
\end{equation}
where $I_f(p)$ is the index of the critical point $p$ at the end of
$\gamma$ which is not the junction point.

Generalized disks correspond to intersections of a stable/unstable
manifold of $f$ with the evaluation map
$\ev\colon\overline{\ms^\ast}\to L$.  The transversality conditions of
the triple $(f,g,J)$ of the Morse function, the metric, and the almost
complex structure discussed above imply that such intersections are
transverse. In particular, it follows from
Lemma~\ref{lem:morsetv} that generalized disks determined by
$(f,g,J)$ have the following three properties: there are no
generalized disks of formal dimension $<0$, every generalized disk of
dimension $0$ is transversely cut out by its defining equation or
\dfn{rigid}, and the set of rigid generalized disks is finite. (The third property is a consequence of the compactness of the moduli space of holmorphic disks and the transveraslity of the evaluation map). We say
that triples $(f,g,J)$ with these properties and for which rigid
generalized disks satisfy certain pairwise transversality conditions
as in Condition $({\bf g3})$ of
Lemma~\ref{lem:morsetv}, are {\em generic with respect to
  rigid generalized disks}.

We define a \dfn{lifted disk} to be a holomorphic disk $u$ in the
space $\ms_A(a; b_1, \ldots, b_k)$ or $\ms_A(a_1,a_2;b_1,\dots,b_k,
c_1,\dots, c_l)$ together with the following data: If $u\in\ms_A(a;b_1,\dots,b_k)$ then choose a puncture $x_j$, $j>0$, of the marked disk $D_{m+1}$.  This puncture and $x_0$ (which maps to the positive puncture $a$)
split $\partial D_{m+1}$ into two components; assign the label $L$ to
one component and $\tilde{L}$ to the other. If $u\in
\ms_A(a_1,a_2;b_1,\dots,b_k,c_1\dots, c_l)$, note first that the projection of $\tilde L$ into $P$ agrees with the projection of $L_1(f)$, then we assign the label $L$ to $\rr\times\{0\}$ and $\tilde L$ to $\rr\times\{1\}$.

Finally, we define a \dfn{lifted generalized disk}. Let $(u,\gamma)$
be a generalized disk. The Morse flow line $\gamma$ in $L$ and its
orientation reversed $z$-translate $\tilde\gamma$ in $\tilde L$ are
two oriented curves, one which is oriented away from the junction
point and one which is oriented toward it. The positive puncture and
the junction point subdivide the boundary of the domain of $u$ into
two parts. One of these parts is oriented toward the junction point
and the other one away from it. If the curve $\gamma$ is oriented away from the junction point then we assign the component $L$ to the part of the boundary of $u$ which is oriented toward the junction point and $\tilde L$ to the other part. If the curve $\gamma$ is oriented toward the junction point then we assign the component $L$ to the part of the boundary of $u$ which is oriented away from the junction point and $\tilde L$ to the other part. 

We note that lifted disks, lifted generalized disks, and negative flow
lines of $f$ give rise to continuous maps from the boundary of a
punctured disk to $2L$ in a natural way. We call such maps {\em lifted
  boundary maps}.

\begin{figure}[ht]
  \relabelbox \small {\centerline{\epsfbox{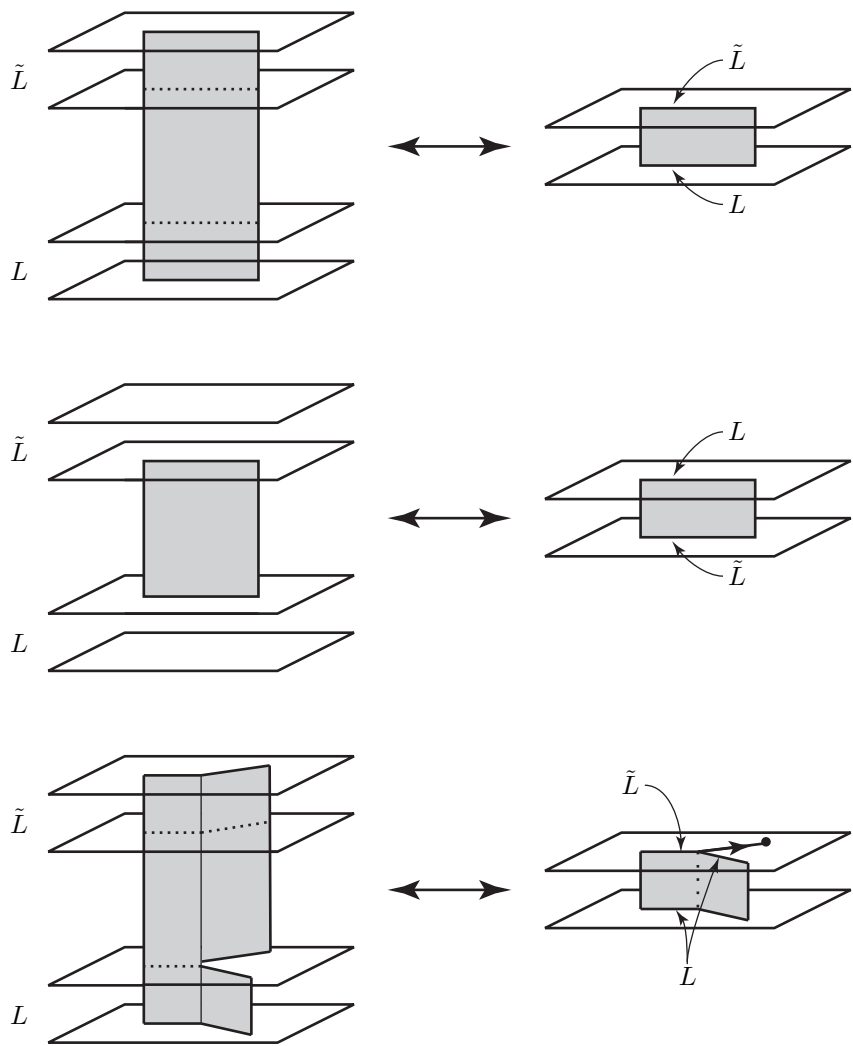}}}
  \relabel{1}{$\tilde{L}$}
  \relabel{2}{$L$}
  \relabel{3}{$\tilde{L}$}
  \relabel{4}{$L$}
  \relabel{5}{$\tilde{L}$}
  \relabel{6}{$L$}
  \relabel{7}{$\tilde{L}$}
  \relabel{8}{$L$}
  \relabel{9}{$\tilde{L}$}
  \relabel{A}{$L$}
  \relabel{B}{$L$}
  \relabel{C}{$\tilde{L}$}
  \endrelabelbox
  \caption{Schematic pictures of disks in the front diagrams of
    \twocopy\ (on the left side) and $L$ (on the right side) for parts (2) and (3) of the correspondence of
    Theorem~\ref{thm:disk+Morse}. The top two diagrams illustrate part
    (2), while the bottom illustrates part (3).  The labels on the
    right side are the labels of the lifted (generalized) disk.}
  \label{fig:gen-disks}
\end{figure}

We now have all of the objects we need to state the main analytic theorem: 

\begin{thm}\label{thm:disk+Morse}
  There exist an admissible Morse function $f\colon L\to\rr$, a
  Riemannian metric $g$ on $L$, and an almost complex structure $J$ on
  $P$ with the following properties.
  \begin{itemize}
  \item[$({\rm i})$] The pair $(f,g)$ is Morse-Smale.
  \item[$({\rm ii})$] The triple $(f,g,J)$ is generic with respect to rigid
    generalized disks.
  \item[$({\rm iii})$] Lemma \ref{lem:2pos} holds for $J$-holomorphic disks with
    boundary on $L_0\cup L_1(f)$.
  \item[$({\rm iv})$] All moduli spaces of $J$-holomorphic disks with one positive
    puncture and with boundary on $2L=L\cup\tilde{L}$ of dimension
    $\le 0$ are transversely cut out. (Here, as above $\tilde L$ is the push off of the Reeb flow image $\phi_s(L)$ of $L$ along $df$.)
  \end{itemize}
  Furthermore, there is a bijective correspondence between the set $X$
  of rigid holomorphic disks with boundary on $2L$ and the union $Y$
  of two (disjoint) copies of the set of rigid disks with boundary on
  $L$, rigid lifted disks determined by $(f,g,J)$, and rigid lifted
  generalized disks determined by $(f,g,J)$. This correspondence is
  such that the continuous lift of the boundary of a holomorphic disk
  with boundary on $2L$ is contained in a small neighborhood of the
  lifted boundary map of the object to which it corresponds, and vice
  versa.  In particular: 
  \begin{enumerate}
  \item The set of disks in $X$ without mixed punctures corresponds to
    the subset of $Y$ consisting of the two copies of the set of disks
    with boundary on $L$: disks with boundary on $L$ corresponds to
    one copy and disks with boundary on $\tilde{L}$ to the other.
  \item The set of disks in $X$ with mixed punctures but without mixed
    Morse punctures corresponds to the subset of $Y$ of lifted
    disks. See Figure~\ref{fig:gen-disks}.
  \item The set of disks in $X$ with exactly one
    Morse puncture corresponds to the subset of $Y$ of lifted generalized
    disks. See Figure~\ref{fig:gen-disks}.
  \item The set of disks in $X$ with two mixed Morse punctures
    corresponds to the subset of $Y$ of rigid negative gradient flow
    lines of $f$.
  \end{enumerate}
\end{thm}

It follows that the contact homology differential on $\alg(2L)$ can be
computed in terms of rigid disks, rigid lifted generalized disks, and
rigid lifted generalized disks. This theorem will be proven in Section
\ref{sec:disks}.

\subsubsection{Chain Complexes}
\label{ssec:chain-complex}

Theorem \ref{thm:disk+Morse} allows us to prove the following correspondences
between chain complexes.

\begin{prop}\label{prop:correspondence}
With notation as established in Section~\ref{ssec:lin-2-copy}, we have:
  \begin{enumerate}
  \item $(Q^1, \fdf_q)$ is isomorphic to the original linearized chain
    complex $(Q(L), \df_1)$ and
  \item $(C^1, \fdf_c)$ is isomorphic to the Morse-Witten complex
    $CM_*(L; f)$  with respect to the perturbing Morse function
    $f$, with degrees shifted so that $C^1_k \cong CM_{k+1}(L;f)$.
  \end{enumerate}
\end{prop}

\begin{proof}
  For both parts of the proposition, there is an obvious bijective
  correspondence between the generators of the underlying graded
  groups (up to the indicated shift in the second part). We will now
  show that the disks contributing to the differentials also
  correspond.

  For $(Q^1, \fdf_q)$, a disk $u$ that contributes $q^1_k$ to $\fdf_q
  q^1_j$ has one positive mixed $q$ puncture $q^1_j$, one negative
  mixed $q$ puncture $q^1_k$, and possibly other negative augmented
  pure $q$ punctures. By Theorem~\ref{thm:disk+Morse}, this
  corresponds to a lifted disk with a positive puncture at $q_k$, a
  negative puncture at $q_j$, and possibly other augmented negative
  punctures.  The rigid disk underlying the lifted disk contributes
  $q_k$ to $\df_1 q_j$ in $Q(L)$, as required.  Conversely, exactly
  one of the two lifts of such a rigid disk yields a disk with mixed
  $q$ punctures at $q_j$ and $q_k$ and appropriate signs at the
  punctures.

  For $(C^1, \fdf_c)$, a disk $u$ that contributes $c^1_k$ to $\fdf_c
  c^1_j$ has a positive mixed Morse puncture at $c^1_j$ and a negative
  mixed Morse puncture at $c^1_k$.  Lemma~\ref{lem:one-punc} and
  Theorem~\ref{thm:disk+Morse} imply that this disk has only these two
  mixed Morse punctures and corresponds to a flow line from the
  critical point $c_j$ to $c_k$, as in the Morse-Witten differential.
  The converse is clear.
\end{proof}

We next relate $(P^1, \fdf_p)$ to the cochain complex of $(Q^1,
\fdf_q)$. Define a pairing $\langle, \rangle_q$ on the generators of
$P^1_{n-k-2} \otimes Q^1_k$ by
\begin{equation}
  \langle p_i^1, q_j^1 \rangle_q = \delta_{ij}.
\end{equation}

\begin{prop} \label{prop:cochain} The complex $(P^1, \fdf_p)$ is
  isomorphic to the cochain complex of $(Q^1, \fdf_q)$ with respect to
  the pairing $\langle, \rangle_q$.
\end{prop}

\begin{proof}
  We need to show that
  \begin{equation} \label{eqn:cochain} \langle \fdf_p p_j, q_k
    \rangle_q = \langle p_j, \fdf_q q_k \rangle_q.
  \end{equation}
  A disk contributing to the left-hand side of \eqref{eqn:cochain} has
  a mixed positive puncture at $p_j$, a mixed negative puncture at
  $p_k$, and possibly other augmented pure punctures.  By
  Theorem~\ref{thm:disk+Morse}, such a disk corresponds to a unique
  rigid lifted disk.  Reversing assignment of components of \twocopy\
  in the lifted rigid disk, we obtain a lifted disk which contributes
  to the right-hand side; see the top two drawings in
  Figure~\ref{fig:gen-disks}. Similarly, to each disk contributing to
  the right hand side, we find a unique disk contributing to the left
  hand side.  \end{proof}

\subsubsection{The Maps $\rho_*$ and $\sigma_*$}
\label{ssec:rho-sigma}

In order to interpret Poincar\'e duality in the Morse-Witten complex,
we first recall that if $f_0$ and $f_1$ are Morse functions on $L$
with Morse-Smale gradient flows, then there is a continuation
homomorphism $h\colon CM_\ast(L;f_0)\to CM_{\ast}(L;f_1)$.  To define
this map, choose a generic path of functions $f_s$, $s\in\rr$, which
agrees with $f_0$ for $s\le 0$ and with $f_1$ for $s\ge 1$.  The maps
is then given by counting solutions of the differential equation
$\dot\gamma(t)=-\nabla f_t(\gamma(t))$, $t\in\rr$, which are
asymptotic to critical points of $f_0$ as $t \to -\infty$ and to
critical points of $f_1$ as $t \to+\infty$.

Poincar\'e duality for $H_*(L)$ can then be realized on the chain
level in the Morse-Witten complex as follows (see \cite{milinkovic} or
\cite{schwarz}, for instance): there is an obvious correspondence
\begin{align*}
  \Delta: CM_k(L;-f) &\to CM^{n-k}(L;f) \\
  x &\mapsto \langle x, \cdot \rangle.
\end{align*}
It induces an isomorphism $\Delta_*: HM_k(L;-f) \to HM^{n-k}(L;f)$ on
homology which, when combined with the continuation map $h: CM_k(L;f) \to
CM_k(L;-f)$, yields the Poincar\'e duality isomorphism $\Delta_*
h_*$.  We can thus define a Poincar\'e pairing on $HM_k(L;f) \otimes
HM_{n-k}(L;f)$ by the following pairing at the chain level:
\begin{equation} \label{eqn:poincare-pairing} x \bullet y = (\Delta
  \circ h (x))(y).
\end{equation}

Using this interpretation of Poincar\'e duality, we get the following
relationship between the maps $\rho$ and $\sigma$.

\begin{prop} \label{prop:rho-sigma-adj} The maps $\rho_*$ and
  $\sigma_*$ are adjoints in the following sense:
  \begin{equation*}
    x \bullet \rho_* q  = \langle \sigma_*x, q \rangle_q.
  \end{equation*}
\end{prop}

Proposition \ref{prop:rho-sigma-adj} is proved at the end of this
section. The proof uses the following lemma.

\begin{lem} \label{lem:rho-sigma-ind} The maps $\rho_*$ and $\sigma_*$
  are independent of the perturbation function $f$.  That is, if $f_0$
  and $f_1$ are two generic perturbations and $h_*: HM_*(L;f_0) \to
  HM(L;f_1)$ is the continuation map, then:
  \begin{equation*}
    \rho^1_* = h_*  \rho^0_* \quad \text{and} \quad \sigma^1_*
    h_* = \sigma^0_*,
  \end{equation*}
  where $\rho^i$ and $\sigma^i$ are the maps corresponding to the
  perturbation function $f_i$.
\end{lem}

Lemma \ref{lem:rho-sigma-ind} uses some results from Morse theory
which we present before giving its proof. Similar Morse
theoretic questions were considered in \cite{fooo,
  schwarz:equivalence}. Our approach here is a modification of that in
\cite{schwarz:equivalence}.

Denote the singular chain complex of $L$ by $C_\ast(L)$ and the
corresponding homology by $H_\ast(L)$. Define a map $\Phi^f: CM_*(L;f)
\to C_*(L)$ as follows. If $x$ is a critical point of $f$, then
$\Phi^f(x)$ is the singular chain carried by the closure of the
unstable manifold $\overline{W^{\rm u}(x)}$.

\begin{lem} \label{lem:phi-chain-map} The map $\Phi^f$ is a chain map
  that is compatible with the continuation map $h_*: HM_*(L;f_0) \to
  HM_*(L;f_1)$ in the sense that $\Phi^{f_1}_* h_* = \Phi^{f_0}_*$.
\end{lem}

\begin{proof}
  By Lemma 4.2 of \cite{schwarz:equivalence}, the boundary of
  $\overline{W^{\rm u}(x)}$ consists of products of rigid flow lines
  from $x$ to $y$ with $\overline{W^{\rm u}(y)}$.  The first part of
  the lemma follows.  The second part is similar to the proof of Lemma~4.8 
  of \cite{schwarz:equivalence}, to which we refer the reader for the 
  necessary gluing and compactness results:
  let $x_j$ be critical points of
  $f_j$, $j=0,1$, such that the number of rigid solutions of
  $\dot\gamma(t)=-\nabla f_{t}(\gamma(t))$ asymptotic to $x_0$ as $t
  \to -\infty$ and $x_1$ as $t \to +\infty$ equals $m$.  To construct
  a chain whose boundary is the chain carried by the closure of the
  unstable manifold $\overline{W^{\rm u}(x_0)}$ (defined with respect
  to $f_0$) and $m$ times the chain carried by the closure of the
  unstable manifold $\overline{W^{\rm u}(x_1)}$ (defined with respect
  to $f_1$), consider the closure of the set of points which lie on
  some solution curve of $\dot\gamma(t)=-\nabla f_{t}(\gamma(t))$
  which is asymptotic to $x_0$ as $ t\to-\infty$. 
\end{proof}

\begin{rem}
  For technical reasons, we assume from here on that $C_*(L)$ is
  generated by a set of $C^{1}$ chains that is the union of both the
  chains in the image of $\Phi^f$ and the chains in the image of the
  evaluation maps of all moduli spaces of holomorphic disks with
  boundary on $L$, one marked point, and at most two positive
  punctures.\footnote{See Subsection~\ref{ssec:diffl}, above, and
    \cite[\S7.7]{ees:high-d-analysis}.} Further, we assume that the
  $C^{1}$ chains are transverse to the closures of the stable manifolds
  $\overline{W^{\rm s}(x)}$.  That this chain complex exists follows
  from arguments similar to those in \cite[\S19]{fooo}.
\end{rem}

The inverse map $\Psi^f: C_*(L) \to CM_\ast(L;f)$ is defined using
intersections between chains and the stable manifolds of the critical
points as follows. Let:
\begin{equation}
  \Psi^f(\chi) = \sum_x (\chi \bullet W^{\rm s}(x))x,
\end{equation}
where $\chi$ is a chain, $\bullet$ denotes the intersection product on
chains, and where the sum ranges over all critical points $x$ of $f$.
Note that the intersection product on chains is well-defined by our
assumptions on the chains generating $C_*(L)$.  Using the same proof
as for Lemma~\ref{lem:phi-chain-map}, though this time based on ideas
from Section 4.2 of \cite{schwarz:equivalence}, we obtain the
following.

\begin{lem} \label{lem:psi-chain-map} The map $\Psi^f$ is a chain map
  compatible with the continuation map.
\end{lem}

Since it is clear that $\Psi^f \circ \Phi^f$ is the identity on
$CM_*(L;f)$, and since the Morse and singular homologies are
isomorphic, it follows that $\Phi^f_*$ and $\Psi^f_*$ are mutual
inverses.

We can now prove that $\rho_*$ and $\sigma_*$ are independent of the
perturbation.

\begin{proof}[Proof of Lemma~\ref{lem:rho-sigma-ind}]
  Let $\ms(q)$ be the moduli space of holomorphic disks with positive
  puncture at $q$ and possibly negative punctures at augmented
  crossings.  Let $K_q$ be the image in $C_*(L)$ of
  $\overline{\ms}(q)$ under the evaluation map.
  Theorem~\ref{thm:disk+Morse} and the discussion before it imply that
  $\rho$ may be defined using lifted generalized disks, so we obtain:
  \begin{equation*}
    \rho^i(q) = \Psi^i(K_q).
  \end{equation*}
  If $\alpha \in H_*(Q^1, \fdf_q)$, then:
  \begin{align*}
    \rho^0_* \alpha &= \Psi^0_*(K_\alpha) \\
    &= h_* \Psi^1_*(K_\alpha) \\
    &= h_* \rho^1_* \alpha.
  \end{align*}

  The proof for $\sigma_*$ is similar.
\end{proof}

Finally, we are ready to prove the original proposition.

\begin{proof}[Proof of Proposition~\ref{prop:rho-sigma-adj}]
  Denote by $\rho^f$ and $\sigma^{-f}$ the maps induced by the
  perturbations $\pm f$.  By Lemma~\ref{lem:rho-sigma-ind}, it
  suffices to prove that, on the chain level,
  \begin{equation} \label{eqn:rho-sigma-adj} \tilde{h}(y) \bullet
    \rho^f q = \langle \sigma^{-f} y, q \rangle_q
  \end{equation}
  for any given $y \in CM_k(L;-f)$, where $\tilde{h}$ is the
  continuation map from $CM_k(L; -f) \to CM_k(L;f)$.

  Theorem~\ref{thm:disk+Morse} says that a contribution to the
  left-hand side is given by a lifted generalized disk with a mixed
  positive puncture at $q$ and a negative Morse puncture at $y$ (and
  possibly other negative augmented pure punctures), where the
  flow line $\gamma$ in the generalized disk is a negative gradient
  flow line for $f$.  By reversing the lifting, we obtain a lifted
  generalized disk with a mixed negative puncture at the $p$
  corresponding to $q$ and a positive Morse puncture at $y$.  This
  time, however, the flow line $\gamma$ points in the opposite
  direction, and hence is a negative gradient flow line for $-f$.
  Thus, this disk contributes to $\sigma^{-f} y$, and hence to the
  right hand side. Therefore, there is a bijective correspondence
  between the disks that determine the left- and right-hand sides of
  Equation~(\ref{eqn:rho-sigma-adj}).
\end{proof}

\section{The Duality Sequence}
\label{sec:duality}

In this section, we prove Theorem \ref{thm:duality} in two
steps. First, we establish an isomorphism between $QC^{1}$ and $P^{1}$
(using notation as in Section \ref{ssec:lin-2-copy}) for $L\subset P
\times \rr$.  Second, we collect the information from the
identifications of Section \ref{ssec:subcomplexes}.

\subsection{The Duality Isomorphism}
\label{ssec:duality-map}

Assume that $L\subset P \times \rr$ is a Legendrian submanifold with
satisfying the assumptions of Theorem~\ref{thm:duality}.  The first
step will be to show that $H_*$ is an isomorphism between $QC^1$ and
$P^1$.

\begin{prop} \label{prop:duality-isom} The map $H_*: H_*(QC^1) \to
  H_*(P^1)$ is a degree $-1$ isomorphism.
\end{prop}

\begin{proof}
  By Corollary~\ref{cor:mapping-cone}, it suffices to prove that the
  complex
  \[
  (QC^1 \oplus P^1, -\fdf_{qc} + H + \fdf_p)
  \]
  is acyclic.  Since $(Q(\twocopy), \fdf_1)$ splits as the direct sum
  of two subcomplexes $Q^0$ and $QC^1 \oplus P^1$, it is clear that
  \begin{equation}
    H_*(Q(\twocopy)) \cong H_*(Q^0) \oplus H_*(QC^1 \oplus P^1).
  \end{equation}
  In order to prove that $QC^1 \oplus P^1$ is acyclic,
  we show that the linearized contact homology of $\twocopy$
  comes from $Q^0$ only, as follows.

  Modify the two-copy by a Legendrian isotopy that leaves the bottom
  component fixed and horizontally displaces the top component
  by the lift of a Hamiltonian isotopy in $P$ of
  the projection of $L$ off of itself.
  For any
  augmentation, the linearized contact homology of this shifted
  two-copy is clearly isomorphic to the homology of $Q^0$ since all of
  the other components of $Q(2L)$ are trivial.  Since the set of
  linearized contact homologies is invariant under Legendrian isotopy,
  this implies that, for any augmentation, the linearized contact
  homology of $\twocopy$ comes from $Q^0$ only, as desired.
\end{proof}

\begin{rem}
  It should not be surprising that $H$ turns out to be an isomorphism
  essential to the proof of duality. In finite-dimensional Morse
  theory, rigid gradient flow trees with two positive ends and one
  vertex define the Poincar\'e duality isomorphism (see
  \cite{cohen-betz}, for example).  Using
  Theorem~\ref{thm:disk+Morse}, it is straightforward to see that the
  disks that define $H$ --- or at least $\eta$ --- correspond to disks
  with two positive punctures and no non-augmented negative punctures
  in $L$.  Thus, the disks defining $H$ are combinatorially similar to
  the gradient flow trees that give Poincar\'e duality.
\end{rem}

\subsection{The Proof of Theorem \ref{thm:duality}}
\label{ssec:duality-pf}

The complex $QC^1$ is the mapping cone of $\rho$.  The corresponding
long exact sequence is:
\begin{equation}
  \cdots \to H_k(C^1) \overset{i_*}{\to} H_k(QC^1) \to H_k(Q^1)
  \overset{\rho_*}{\to} \cdots,
\end{equation}
where $i: C^1 \to QC^1$ is the natural inclusion.  By
Proposition~\ref{prop:correspondence}(2), we have $H_k(C^1) \cong
H_{k+1}(L)$.  By Propositions~\ref{prop:cochain} and
\ref{prop:duality-isom}, the middle term $H_k(QC^1)$ is isomorphic to
$H^{n-k-1}(Q^1)$. Further, we have $H_* i_* = \sigma_*$.  Inserting
these facts into the exact sequence above yields:
\begin{equation}
  \cdots \to H_{k+1}(L) \overset{\sigma_*}{\to} H^{n-k-1}(Q^1)
  \to H_k(Q^1) \overset{\rho_*}{\to} \cdots.
\end{equation}
Proposition~\ref{prop:correspondence}(1) now finishes the proof of the
first part of the duality theorem.  The second part of the duality
theorem follows directly from Proposition~\ref{prop:rho-sigma-adj}.

\section{Applications and Examples}
\label{sec:ex}

\subsection{Proof of Theorem \ref{thm:arnold}}
\label{ssec:arnold-proof}

Let $L\subset P \times \rr$ be a Legendrian submanifold with
linearizable contact homology over a field $\Lambda$. To set notation,
let:
\begin{align*}
  b_k &= \dim H_k(L), \\
  r_k &= \dim \img \rho_* \subset H_k(L), \\
  s_k &= \dim \img \sigma_* \subset H^{n-k}(Q(L)). \\
\end{align*}
Note that $r_k$ is at most the dimension of $H_k(Q(L))$, which in turn
bounds below $c_k$, the number of Reeb chords of grading $k$.  By the
second part of Theorem \ref{thm:duality}, we obtain $s_k = r_{n-k}$,
so:
\begin{equation*} \label{eqn:split-betti} b_k = r_k + s_k = r_k +
  r_{n-k} \leq c_k + c_{n-k}.
\end{equation*}

\subsection{Basic Examples}
\label{ssec:explan-ex}

We study the relation between the linearized contact homology and the
Morse homology implied by Theorem \ref{thm:duality} in several simple
examples.

\begin{exam}[The Flying Saucer,
  revisited] \label{ex:flying-saucer-hom} Recall the flying saucer of
  Example~\ref{ex:flying-saucer}, whose linearized homology (and
  cohomology) is $\zz$ in degree $n$ and trivial otherwise.  The
  relevant part of the duality exact sequence in
  Theorem~\ref{thm:duality} is
  \begin{equation*}
    \cdots \to H^{-1}(Q(L)) \to H_n(Q(L)) \overset{\rho_*}{\to}
    H_n(L) \to H^0(Q(L)) \to \cdots.
  \end{equation*}
  This reduces to $0 \to \zz \to \zz \to 0$.  Thus, we see that
  $\rho_*$ gives an isomorphism between the degree $n$ linearized
  homology and the group generated by the fundamental class $[L]$ of
  $L \cong S^n$.

  Since $H_0(Q(L)) = 0$, and hence the image of $\rho_*$ in $H_0(L)$
  is trivial, it is clear that the Poincar\'e dual of $[L]$ evaluates
  to $0$ on the image of $\rho_*$ in $H_0(L)$, as guaranteed by the
  Theorem~\ref{thm:duality}.  In fact, as we can see from the
  following part of the exact sequence, $H_0(L)$ is isomorphic to the
  dual of $H_n(Q(L))$
  \begin{equation*}
    \cdots \to H_0(Q(L)) \to H_0(L) \overset{\sigma_*}{\to}
    H^n(Q(L)) \to H_{-1}(Q(L)) \to \cdots.
  \end{equation*}
  In other words, we see that the Morse homology of $L$ splits
  between the linearized homology and the linearized cohomology.
\end{exam}

\begin{exam}[Chekanov's Example in $\rr^3$, reinterpreted]
  \label{ex:3d}
  \begin{figure}[ht]
    \relabelbox
    \small{\centerline{\epsfbox{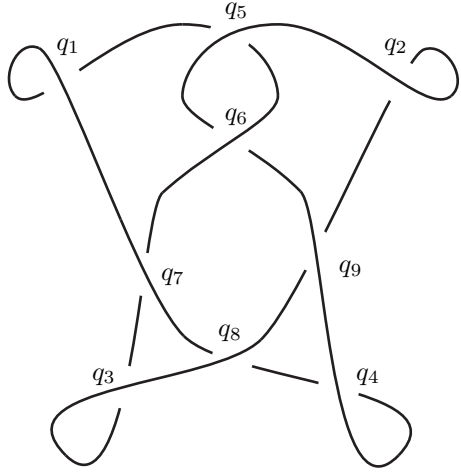}}}
    \relabel {1}{$q_1$}
    \relabel {2}{$q_2$}
    \relabel {3}{$q_3$}
    \relabel {4}{$q_4$}
    \relabel {5}{$q_5$}
    \relabel {6}{$q_6$}
    \relabel {7}{$q_7$}
    \relabel {8}{$q_8$}
    \relabel {9}{$q_9$}
    \endrelabelbox
    \caption{The knot $L$.}
    \label{fig:chex}
  \end{figure}
  Consider the Legendrian knot $L$ in $\rr^3$ whose Lagrangian
  projection is shown in Figure~\ref{fig:chex}.  Working over $\zz_2$,
  the algebra for $L$ is $\alg(L)=\zz_2\langle q_1, \ldots q_9\rangle$
  with $|q_i|=1$ for $i=1,\ldots, 4$, $|q_5|=2=-|q_6|$, and $|q_i| =
  0$ for $i \geq 7$.  We have
  \begin{equation*}
    \partial q_i= \begin{cases} 1+ q_7 + q_7q_6q_5& i=1,\\ 1+ q_9+
      q_5q_6q_9&i=2,\\ 1+ q_8q_7&i=3,\\ 1+ q_9q_8&i=4,\\ 0& i\geq 5.
    \end{cases}
  \end{equation*}

  This differential is not augmented, but there is a unique
  augmentation $\aug$ that sends $q_7, q_8$, and $q_9$ to $1$. The
  linearized differential is
  \begin{equation}
    \dfe_1 q_i= \begin{cases} q_7 & i=1,\\ q_9&i=2,\\
      q_8+q_7&i=3,\\ q_9+q_8&i=4,\\ 0& i\geq 5.
    \end{cases}
  \end{equation}
  Thus, we have the following ranks for the linearized homology

  \begin{center}
    \begin{tabular}{c|ccccc}
      $k$ & $-2$ & $-1$ & $\mathbf{0}$ & $1$ & $2$ \\ \hline
      $\dim H_k(Q(L))$ & $1$ & $0$ & $0$ & $1$ & $1$
    \end{tabular}
  \end{center}

  This computation agrees with the predictions of the duality theorem
  in \cite{duality}: off of a class in degree $1$, the
  linearized homology is symmetric about degree $0$.

  The first interesting part of the duality exact sequence is
  \begin{equation*} \cdots \to H_{-1}(L)\to H^2(Q(L)) \to
    H_{-2}(Q(L)) \to H_{-2}(L) \to \cdots.
  \end{equation*}
  This sequence reduces to $0 \to \zz_2 \to \zz_2 \to 0$ and shows
  that $H^2(Q(L))$ and $H_{-2}(Q(L))$ are isomorphic; this is the
  ``duality'' in Theorem~\ref{thm:duality}.

  The next interesting parts of the exact sequence are
  \begin{equation*} \cdots \to H^{-1}(Q(L))\to H_{1}(Q(L))
    \to H_1(L) \to H^0(Q(L)) {\to} \cdots.
  \end{equation*}
  and
  \begin{equation*}
    \cdots \to H_0(Q(L)) \to H_0(L) \to H^1(Q(L)) \to H_{-1}(Q(L)) \to \cdots.
  \end{equation*}
  As above, these sequences both reduce to $0 \to \zz_2 \to \zz_2\to
  0$.  As in the flying saucer example, we see that the homology of $L
  \cong S^1$ is split between the linearized contact homology (in this
  case, $H_1(L) \cong H_1(Q(L))$) and cohomology (in this case,
  $H_0(L) \cong H^1(Q(L))$) in Poincar\'e dual pairs.
\end{exam}

\begin{exam}[Front Spinning]
  Another source of examples in which the manifold classes follow an
  interesting pattern is the front spinning construction from
  \cite{ees:high-d-geometry}.  Given a Legendrian submanifold $L$ in
  $\rr^{2n+1}$, we construct the \emph{suspension} $\Sigma(L)$ of $L$
  by ``spinning the front of $L$''. More specifically, suppose $\phi
  \colon L\to \rr^{2n+1}$ is a parametrization of $L$, and for $p\in
  L$, we write $\phi(p)=(x_1(p),y_1(p), \ldots, x_n(p),y_n(p),z(p))$.
  The front projection $\Pi_F(L)$ of $L$ is parametrized by
  $\Pi_F\circ \phi(p)=(x_1(p),\ldots, x_n(p),z(p))$. We may assume
  that $L$ has been translated so that the $x_1$ coordinates of all
  points in $\Pi_F(L)$ are positive.  If we embed $\rr^{n+1}$ into
  $\rr^{n+2}$ via $(x_1, \ldots, x_n, z)\mapsto (x_0=0, x_1,\ldots
  x_n,z)$, then $\Pi_F(\Sigma L)$ is obtained by revolving
  $\Pi_F(L)\subset\rr^{n+1}$ around the subspace $\{x_0=x_1=0\}$ as in
  Figure~\ref{1fig:1spin}.
  \begin{figure}[ht]
    \relabelbox \small {\epsfxsize=3in\centerline{\epsfbox{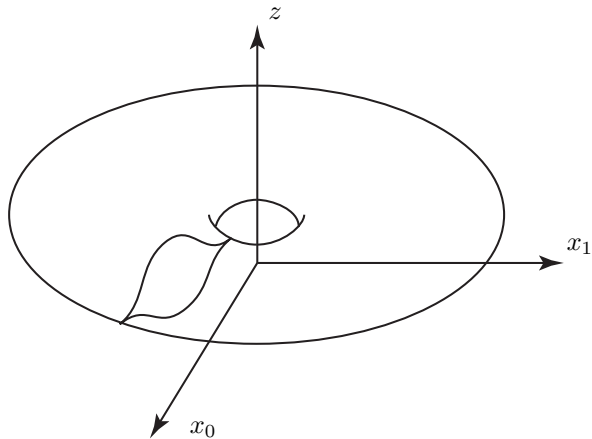}}}
    \relabel{0}{$x_0$}
    \relabel{1}{$x_1$}
    \relabel{z}{$z$}
    \endrelabelbox
    \caption{The front of $\Sigma L.$}
    \label{1fig:1spin}
  \end{figure}
  That is, we can parametrize $\Pi_F(\Sigma L)$ by $(x_1(p)
  \sin\theta, x_1(p) \cos\theta , x_2(p), \ldots, x_n(p),z(p))$, for
  $\theta\in S^1$. Thus, $\Pi_F(\Sigma L)$ is the front for a
  Legendrian embedding $L\times S^1\to\rr^{2n+3}$. We denote the
  corresponding Legendrian submanifold by $\Sigma L$.

  We can derive the following facts about the contact homology DGA
  over $\zz_2$ of $\Sigma L$ from Proposition 4.17 of
  \cite{ees:high-d-geometry}: if $\{q_1, \ldots, q_n\}$ are the
  generators of $\alg(L)$, then $\alg(\Sigma L)$ is stable tame
  isomorphic to an algebra generated by the set
  \begin{equation*}
    \bigl\{q_1[\alpha], \ldots, q_n[\alpha], \hat{q}_1[\beta], \ldots
    \hat{q}_n[\beta]\bigr\}_{\alpha = 0,2, \beta = 1,3}.
  \end{equation*}
  Let $\Delta_\alpha: \alg(L) \to \alg(\Sigma L)$ take $q_i$ to
  $q_i[\alpha]$, and similarly for $\Delta_\beta$.  The gradings of
  the new generators are given by $|q_i[\alpha]| = |q_i|$ and
  $|\hat{q}_i[\beta]| = |q_i| + 1$.  Using the form of the
  differential in \cite{ees:high-d-geometry}, we can see that if
  $\varepsilon$ is an augmentation of $(\alg(L), \df^L)$, then the
  following formula defines an augmentation of $(\alg(\Sigma L),
  \df^\Sigma)$
  \begin{align*}
    \varepsilon_\Sigma(q_i[\alpha]) &= \varepsilon(q_i), \\
    \varepsilon_\Sigma(\hat{q}_i[\beta]) &= 0.
  \end{align*}
  We will use this form of augmentation from here on, though there
  might be others.  Finally, the linearized differential takes the
  following form, where the augmentations have been suppressed from the
  notation
  \begin{align*}
    \df^\Sigma_1 q_i[\alpha] &= \Delta_\alpha (\df^L_1 q_i), \\
    \df^\Sigma_1 \hat{q}_i[\beta] &= q_i[0] + q_i[2] + \Delta_\beta
    (\df^L_1 q_i).
  \end{align*}

  \begin{claim} \label{clm:kunneth} The linearized homology of $\Sigma
    L$ with respect to the augmentation $\varepsilon_\Sigma$ may be
    computed using a K\"unneth-like formula
    \begin{equation*}
      H_*(Q(\Sigma L))
      \simeq H_*(Q(L)) \otimes H_*(S^1).
    \end{equation*}
  \end{claim}

  \begin{proof}
    To prove the claim, we change basis by replacing $\hat{q}_i[3]$ by
    $\hat{q}_i[1] + \hat{q}_i[3]$.  The subspace $Q[1+3]$ spanned by
    these new basis elements forms a subcomplex that is isomorphic to
    the original $(Q(L), \df^L_1)$, but with gradings shifted up by
    $1$.  Similarly, the subspaces $Q[0]$ spanned by the $q_i[0]$ and
    $Q[2]$ spanned by the $q_i[2]$ also form subcomplexes isomorphic
    to $(Q(L), \df^L_1)$. Finally, the restriction of the differential
    to the subspace $Q[1]$ yields no new cycles.  Further, if $z$ is a
    cycle in the original chain complex, then the differential of
    $\Delta_1(z)$ identifies the cycles $z[0]$ and $z[2]$ in homology.
    Thus, the linearized homology of $\Sigma L$ comes from $Q[1+3]$
    and $Q[0]$, each of which are copies of the original chain complex
    $Q(L)$, with degrees in $Q[1+3]$ shifted up by one.
  \end{proof}

  If we apply the spinning construction $k$ times to the
  $n$-dimensional ``flying saucer'' of
  Example~\ref{ex:flying-saucer-hom}, the claim above shows that the
  linearized homology corresponds to the homology of $T^k$ with
  degrees shifted up by $n$.  This is half of the homology of $S^n
  \times T^k$, and hence the linearized homology consists entirely of
  manifold classes.

  For a more interesting example, apply the spinning construction $k$
  times to the Legendrian knot $L$ in Example~\ref{ex:3d} to obtain a
  Legendrian $T^{k+1} \subset \rr^{2k+3}$.  Setting $k=2$ and using
  the calculation of the linearized homology in Example~\ref{ex:3d}
  and Claim~\ref{clm:kunneth}, we easily obtain the following
  dimensions

  \begin{center}
    \begin{tabular}{c|ccccccc}
      $k$ & $-2$ & $-1$ & $0$ & $\mathbf{1}$ & $2$ & $3$ & $4$ \\ \hline
      $\dim H_k(Q(\Sigma^2L))$ & $1$ & $2$ & $1$ & $1$ & $3$ & $3$ & $1$
    \end{tabular}
  \end{center}

  The duality theorem implies that the non-manifold classes are
  symmetric about $\frac{3-1}{2} = 1$, and we can see that every
  non-manifold class in $H_k(Q(L))$ has been replaced by $H_*(T^2)$,
  shifted by $k$.  The remaining classes are the manifold classes,
  corresponding to $H_*(T^2)$ shifted up by one degree.
\end{exam}

\subsection{Finding Manifold Classes}
\label{ssec:manifold-ex}

As indicated by the examples above, it is sometimes possible to use
symmetry arguments to predict the degree in which manifold classes
appear.  In particular, working over $\zz_2$, we can generalize the
fact from \cite{duality} that for a Legendrian \emph{knot}, there is
always a manifold class in degree $1$ as follows.

\begin{thm} \label{thm:mfld-class} Suppose $L \subset P \times \rr$
  satisfies the assumptions of Theorem~\ref{thm:duality}.  If the
  contact homology DGA of $L$ is good, then $\rho_*$ is trivial in
  degree $0$ and onto in degree $n$.  That is, the linearized homology
  of $L$ (with respect to any augmentation) always has a $\zz_2$
  factor in degree $n$ corresponding to the fundamental class of $L$.
\end{thm}

\begin{proof}
  By the second part of Theorem~\ref{thm:duality}, it suffices to
  prove that $\rho_*$ is trivial in degree $0$.  To compute $\rho$ in
  degree zero, we need to find lifted generalized disks $(u, \gamma)$
  in $L$ with one negative Morse puncture and a positive mixed
  puncture at $q$ with $|q| = 0$.  The fact that this generalized disk
  is rigid implies, via Equation (\ref{eqn:gen-formal-dim}), that $u$
  belongs to a moduli space of dimension $-1$.  That is, $u$ must be a
  constant map to the double point $q$ with one positive and one
  negative puncture. 
    Further, $q$ must be
  augmented, or else the generalized disk will not contribute to
  $\rho$.

  There are two possibilities for the flow line $\gamma$.  There are
  two sheets of $\Pi_P(L)$ incident to $q$, and the flow line
  $\gamma$ can start on either of these two sheets.  In a suitably
  generic setup, the self-intersection points of $\Pi_P(L)$ are
  disjoint from the stable manifolds of the critical points of $f$
  with index greater than $0$, so the two flow lines starting at $q$
  both descend to a minimum of $f$.  It follows that $\rho (q)$ is
  either zero or the sum of two minima of $f$.  Thus, given a homology
  class $\alpha \in H_0(Q^1)$, its image under $\rho_*$ is the sum of
  an even number of minima. As $L$ is connected, these minima are all
  homologous, so $\rho_* \alpha = 0$.
\end{proof}

\begin{rem}\label{rem:gencoeff}
  Provided certain orientation conventions are employed, Theorem
  \ref{thm:mfld-class} holds for more general coefficients in the case
  that $L$ is spin. More precisely, the two disks corresponding to the
  flow lines in the last argument of the proof should cancel with
  signs. The ordered punctures of these two disks are of the forms
  $(q^{1}, c^{1}, q^{0})$ and $(q^{1},\tilde{q}^{0},c^{1})$,
  respectively. Choosing capping operators as in \cite{ees:ori}, the
  capping operator of $\tilde q^{0}$ is identical to that of $
  {q}^{0}$ and both have index $1$. The capping operator of $c^{1}$
  has index $0$ if $n=\dim(L)$ is odd and index $1$ if $n$ is
  even. Thus, the disks cancel also with signs if $n$ is even. If $n$
  is odd, they cancel with signs provided we redefine the augmentation
  $\hat\epsilon$ on the $\tilde{q}^{0}$-chords by declaring that
  $\faug(\tilde{q}^{0})=-\aug(q)$ instead. Here $q$ is the Reeb chord
  of $L$ which corresponds to the Reeb chord $\tilde{q}^{0}$ of
  $\tilde L$, and $\aug$ is the augmentation on $\alg(L)$.
\end{rem}

\begin{cor} \label{cor:sphere-duality} Let $L$ be a Legendrian homology 
$n$-sphere satisfy the assumptions
  of Theorem~\ref{thm:duality}.  If the contact homology DGA of $L$ is
  good, then the linearized homology of $L$ (with respect to any
  augmentation) always has a $\zz_2$ factor in degree $n$ and, off of
  this factor, the remaining homology is symmetric about
  $\frac{n-1}{2}$. Said another way, if $P(t)$ is the
  Poincar\'e-Chekanov polynomial of the linearized homology, then:
  \begin{equation*}
    P(t)-t^{n-1}P(t^{-1})=(t^{n}-t^{-1}).
  \end{equation*}
\end{cor}

Theorem \ref{thm:mfld-class} and Corollary \ref{cor:sphere-duality}
can greatly ease computations, as can be seen in the next example; see
also \cite{melvin-shrestha}.

\begin{exam}[Example~\ref{ex:stab-ex}, revisited]
  Recall Example~\ref{ex:stab-ex}, in which two flying saucers are
  attached by a tube.  For $n>2$, the algebra is good for degree
  reasons, with three generators in degree $n$, three in degree $n-1$,
  and one is degree $0$.  The fact that the generator of degree $0$ is
  isolated shows that $\dim H_0(Q(L)) = 1$, and
  Corollary~\ref{cor:sphere-duality} immediately implies that
  \begin{equation*}
    P(t) = 1+t^{n-1} + t^n.
  \end{equation*}
  In particular, the duality theorem forces nonzero differentials
  between degrees $n$ and $n-1$.
\end{exam}

\begin{exam}[A Non-Spun Torus]
  The previous example may be combined with Claim~\ref{clm:kunneth} to
  produce an example of a Legendrian $S^1 \times S^n$ that is not spun
  from a Legendrian $S^n$ for $n > 1$.  Let $L$ be the Legendrian
  $n+1$-sphere constructed in the previous example, and let $L'$ be
  standard $n$-dimensional flying saucer.  Taking the connect sum $L
  \# \Sigma L'$ as in \cite{ees:high-d-geometry} yields a Legendrian
  submanifold with four generators in degree $n+1$, five generators in
  degree $n$, and one in degree $0$. As before, the fact that the
  generator of degree $0$ is isolated shows that $\dim H_0(Q(L)) = 1$,
  and Theorem~\ref{thm:duality} and~\ref{thm:mfld-class} 
  immediately imply that
  \begin{equation*}
    P(t) = 1+2t^{n} + t^{n+1}.
  \end{equation*}
  This homology, however, is not of the form $H_*(Q(L'')) \otimes
  H_*(S^1)$ near degree $0$, so $L \# \Sigma L'$ cannot be a spun
  submanifold.
\end{exam}

The final example shows that Theorem~\ref{thm:mfld-class} is the best
that we can hope for in terms of pinning down the degrees of the
manifold classes.

\begin{exam}[Super-spun Products of Spheres] \label{exam:superspun}
  Construct a generic front diagram in $\rr^{p+k+1}$ as follows: let
  $S^p \times S^k \hookrightarrow \rr^{p+k+1}$ be the standard
  embedding, with $S^p \times S^k = \partial (S^p \times B^{k+1})$.
  Deform the embedding so that the $S^k$ cross-sections are fronts for
  the flying saucer (see Example~\ref{ex:flying-saucer}).  Let $F: S^p
  \to \rr$ be a Morse function with one maximum and one minimum, and
  scale each cross section $\{x\} \times S^k$ by $1+\epsilon F(x)$ for
  some small $\epsilon>0$.

  The Legendrian embedding $L$ coming from this front has exactly two
  Reeb chords, one of degree $k+p$ and the other of degree $k$.
  Suppose, for convenience, that $p > k > 1$, so that the algebra
  $\alg(L)$ is good for degree reasons.  Direct calculation then shows
  that the linearized homology is $\zz$ in degrees $k+p$ and $k$, and
  zero otherwise.  Theorem~\ref{thm:duality} shows that the homology
  class of degree $k$ is a manifold class coming from $H_k(S^p \times
  S^k)$.

  Reversing the roles of $k$ and $p$ yields another Legendrian
  embedding of $S^p \times S^k$, but this time with a linearized
  homology class in degree $p$ is a manifold class that comes from
  $H_p(S^p \times S^k)$.  Thus, it is impossible to determine a priori
  which classes in $H_*(L)$ in degrees other than $0$ or $n$ will be
  manifold classes.
\end{exam}

\section{Proof of the Main Analytic Theorem}
\label{sec:disks}

This section is devoted to the proof of Theorem~\ref{thm:disk+Morse},
which describes the rigid holomorphic disks with boundary on $2L$ in
terms of holomorphic disks with boundary on $L$ and flow lines of a
Morse function on $L$. We present the main steps of the proof in
several subsections as follows: in Subsection \ref{s:cmdli}, we endow
compactified moduli spaces of holomorphic disks with boundary on $L$
with a structure of a manifold with boundary with corners, and in
Subsection \ref{s:evmaps}, we define evaluation maps on such
compactified moduli spaces of disks with an additional marked
point. In Subsection \ref{s:morseflow}, we study notions of genericity
for generalized disks and prove existence results for Morse functions
and for almost complex structures with desired transversality
properties. 


In Subsections \ref{s:disktotree} and \ref{ssec:gen-to-hol}, we turn
to the correspondence between rigid disks with boundary on \twocopy\
and rigid lifted and lifted generalized disks on $L$. In
particular, we show how holomorphic disks with boundary on $2L$
converge to generalized disks as the perturbation of the second copy
approaches zero and how generalized disks can be glued to holomorphic
disks with boundary on $2L$ for small enough separation,
respectively. Finally, in Subsection \ref{s:disk+Morse}, we 
prove Theorem~\ref{thm:disk+Morse}.

\subsection{Compactified moduli spaces as manifolds with boundary with
  corners}
\label{s:cmdli}

Moduli spaces of $J$-holomorphic disks with boundary on $L$ admit
natural compactifications consisting of broken holomorphic disks. In
order to describe these compactified spaces, we first describe the
smooth pieces out of which they are built. This description is
provided by Lemma \ref{lem:tvmspc}. We prove this lemma and Lemma
\ref{lem:2pos} in Subsection \ref{sssec:tvlemmas} after recalling some properties
of almost complex structures on cotangent bundles induced by
Riemannian metrics. With these lemmas established, we show how to glue
the smooth pieces together to form a manifold with boundary with
corners. More precisely, in Subsection \ref{sssec:brokendisk}, we
describe the broken disks that compactify the moduli space. In
Subsections \ref{sssec:floerpicard} and \ref{sssec:fansetup}, we
describe a basic tool for gluing and how to express the problem of
gluing broken holomorphic disks in a language suitable for using this
basic tool. In Subsection \ref{ssec:apphol}, we start our construction
of coordinate charts by constructing approximately holomorphic disks
near broken disks and in Subsection \ref{sssec:cornermap}, we complete
it. Finally, in Subsection \ref{sssec:cornerstr}, we show that the
coordinate charts fit together to give the desired structure of a
manifold with boundary with corners.

\subsubsection{Proofs of Lemmas \ref{lem:tvmspc} and \ref{lem:2pos}}
\label{sssec:tvlemmas}
\begin{rem}\label{rem:acsfrommetr}
  Since we will use some properties of the almost complex structure on
  a cotangent bundle induced by a Riemannian metric, we include a
  short discussion of its definition. Let $q=(q_1,\dots,q_n)$ be local
  coordinates on $L$ and let $p=(p_1,\dots,p_n)$ be dual coordinates
  on the fibers of $T^{\ast}L$. Write $\pa_j=\frac{\pa}{\pa q_j}$ and
  $\pa_{j^{\ast}}=\frac{\pa}{\pa p_j}$ for the corresponding tangent
  vectors. In order to shorten notation, we employ the summation
  convention in what follows. Let $g$ be a metric on $L$,
  $g(q)=g_{ij}(q)dq_i\otimes dq_j$ and let $g^{ij}(q)$ be such that
  $g^{ij}(q)g_{jk}(q)=\delta^{i}_k$, where $\delta^{i}_k$ is the
  Kronecker delta. Let $\Gamma^{i}_{jk}$ be the Christoffel symbols of
  $g$ given by
\[
\Gamma^{i}_{jk}(q)=\tfrac12\sum_sg^{is}\bigl(\pa_k(g_{js})+\pa_j(g_{ks})-\pa_s(g_{jk})\bigr).
\]
Then, the induced covariant derivative on the cotangent bundle is
\[
\nabla(a_idq_i)=\left(\pa_j(a_k) -\Gamma^{i}_{jk}a_{i}\right)\,dq_j\otimes dq_k.
\]
This covariant derivative gives a decomposition of the tangent bundle
$T(T^{\ast}L)$ into a vertical sub-bundle $V$ given by the kernel of
the differential of the projection $\pi\colon TL\to L$ and the
horizontal sub-bundle $H$ with fiber at $\alpha\in T^{\ast}L$ spanned
by tangent vectors at the starting point of covariantly constant
curves with initial value $\alpha$. In local coordinates
\[
V_{(q,p)}=\Spa\left\{\pa_{j^{\ast}}\right\}_{j=1}^{n},\quad
H_{(q,p)}=\Spa\left\{\pa_j+p_r\Gamma^{r}_{js}\pa_{s^{\ast}}\right\}_{j=1}^{n}.
\]  
The almost complex structure $J$ induced by $g$ satisfies $J(V)=H$ and
is defined as follows on vertical vectors $v\in V_\alpha$: translate
$v$ to the origin of $T_{\pi(\alpha)}^{\ast}L$, identify the
translate with a tangent vector of $L$ using the metric, and let $Jv$
be the negative of the horizontal lift of this tangent vector. In
local coordinates:
\begin{equation}\label{e:Jloccoord}
  J\pa_{j^{\ast}}=-g^{kj}\bigl(\pa_k+p_{r}\Gamma^{r}_{ks}\pa_{s^{\ast}}\bigr),\quad
  J\pa_j=g_{jk}\pa_{k^{\ast}} + g^{ls}\bigl(p_r\Gamma^{r}_{js}\pa_l+p_{r}p_{v}\Gamma^{r}_{js}\Gamma^{v}_{lt}\pa_{t^{\ast}}\bigr).
\end{equation} 
\end{rem}


\begin{proof}[Proof of Lemma \ref{lem:tvmspc}]
  Using Darboux balls around the double points of $\Pi_P(L)$, we may
  identify these neighborhoods with small balls around $0$ in
  $\cc^{n}$. Let $J_0$ in these balls correspond to the standard
  almost complex structure. It is then easy to find a small Legendrian
  isotopy which makes both sheets at each intersection point agree
  with their tangent spaces at $0$. Note that in a neighborhood of the
  intersection point, the pull-back of $J_0$ under the neighborhood
  map which identifies cotangent fibers with Lagrangian subspaces
  perpendicular the image linear subspace correspond to the almost
  complex structure induced by a flat metric, see Remark
  \ref{rem:acsfrommetr}. Letting $g$ be a metric which is flat in
  these neighborhoods we may push forward the corresponding almost
  complex structure to a neighborhood of $\Pi_P(L)$ using a symplectic
  neighborhood map which agrees with the map discussed above near each
  Reeb chord endpoint. This gives an almost complex structure in a
  neighborhood of $\Pi_P(L)$ compatible with $d\theta$. Extending it
  to all of $P$ establishes the first statement of the theorem.

  The proof of the second statement is a word by word repetition of
  the proof of Proposition 2.3 in \cite{ees:pxr} once we establish
  transversality within the smaller class of almost complex structures
  used here (the additional condition here is that they are required
  to be standard in some neighborhood of $\Pi_P(L)$) and an upper
  bound on dimension.

  We start with the dimension bound. Since the area of any non-trivial
  $J$-holomorphic disk is bounded from below and since there are only
  finitely many Reeb chords, it follows that there are only finitely
  many Reeb chord collections $(a;b_1,\dots,b_k)$ such that
  $\ms_A(a;b_1,\dots,b_k)$ may be non-empty. Gromov
  compactness\footnote{See \cite{ees:high-d-analysis,ees:pxr} for
    proofs of Gromov compactness in the present setting.} then implies
  that for each fixed Reeb chord collection, there are at most
  finitely many homology classes $A$ and homotopy classes $\alpha$
  such that $\ms_A^{\alpha}(a;b_1,\dots,b_k)$ may be non-empty. This
  gives the uniform dimension bound.

  As for transversality, in the argument in the proof of surjectivity
  in \cite[Lemma 4.5 (1)]{ees:pxr}, the perturbation $K_S$ may be
  taken to have support disjoint from $\Pi_P(L)$. Hence, it follows
  from Gromov compactness of the moduli space of $J$-holomorphic disks
  and from the uniform area bound of the disks discussed above that we
  can achieve transversality among $J$ which agree with $J_0$ in some
  neighborhood $N$ of $\Pi_P(L)$.
\end{proof}

\begin{proof}[Proof of Lemma \ref{lem:2pos}]
  The proof follows from an argument entirely analogous to the proof
  of \cite[Theorem 7.12]{ees:high-d-analysis}: consider the linearized
  $\bar\pa_J$-operator at a solution $u\colon D\to P$ with boundary on
  $L_0\cup L_1(\hat f)$. The two positive punctures of $u$ map to
  different points in $\Pi_P(L_0)\cap \Pi_P(L_1(\hat f))$ since at a
  double point corresponding to a positive puncture the disk comes in
  along the lower branch and out along the upper. Thus, by adding
  perturbations of $\hat f$ supported near the double point at one of
  the positive corners of $u$, we show that the operator is surjective
  on the complement of exceptional holomorphic disks. Exactly as in
  the proof of \cite[Theorem 7.12]{ees:high-d-analysis}, it follows
  that the existence of exceptional disks of small dimension
  contradicts this surjectivity result for disks of dimensions $\le
  0$. The lemma follows.
\end{proof}

\subsubsection{Broken disks}\label{sssec:brokendisk}

By Gromov compactness, a sequence of elements in
$\ms_A(a;b_1,\dots,b_k)$ has a subsequence which converges to a broken
$J$-holomorphic disk with one positive puncture. As we shall see, the
smooth moduli spaces corresponding to the pieces of the broken disks
fit together to a compact manifold with boundary with corners
$\overline{\ms}(a;b_1,\dots,b_k)$, the interior of which is
$\ms_A(a;b_1,\dots,b_k)$. 


A \dfn{broken $J$-holomorphic disk} with positive puncture at a Reeb
chord $a$ and negative punctures at the Reeb chords $b_1, \ldots, b_k$
is a collection of $J$-holomorphic disks $u^0, \ldots, u^r$ together
with a connected rooted (and hence directed) planar tree $T$ that has
$r+k+2$ vertices, of which $1$ is the root and another $k$ are leaves.
We will frequently drop $T$ from the notation in the future.  To each
interior vertex of $T$ there corresponds one disk $u^j$, with $u^0$
adjacent to the root, and to each leaf there corresponds a Reeb chord
$b_j$ in counterclockwise order from the left of the root.  These
correspondences satisfy:
\begin{itemize}
\item Each disk $u^j$ has one positive puncture and as many negative
  punctures as there are incoming edges into the vertex $j$.  
\item If there is an interior edge from vertex $j$ to vertex $l$, and
  that edge is the $m^{\text{th}}$ incoming edge counterclockwise from
  the outgoing edge of the vertex $l$, then the positive puncture of
  $u^j$ and the $m^{\text{th}}$ negative puncture of $u^l$ both map to
  the same Reeb chord $q$. We say that $u^{j}$ {\em is attached} to
  $u^{l}$ at $q$.
\item If there is a leaf labeled $b_j$ whose outgoing edge goes to
  vertex $l$, and that edge is the $m^{\text{th}}$ incoming edge
  counterclockwise from the outgoing edge of the vertex $l$, then the
  $m^{\text{th}}$ negative puncture of $u^l$ maps to the Reeb chord
  $b_j$.
\item The positive puncture of $u^0$ maps to the Reeb chord $a$.
\end{itemize}
See Figure~\ref{fig:disk-tree}. Let $T^j$ be the be the subtree of $T$
obtained by detaching the outgoing edge from $u^j$ from its endpoint
and adding a new root.  We also remark that if $u^{0},\dots,u^{r}$ are
the pieces of a broken $J$-holomorphic disk such that the boundary of
$u^{j}$ correspond to the homology class $A^{j}\in H_1(L)$ then the
boundary of the disk obtained by attaching all disks of the broken
disk corresponds to $A=\sum_j A^{j}$.

\begin{figure}[ht]
  \relabelbox
  \small{\epsfxsize=2in\centerline{\epsfbox{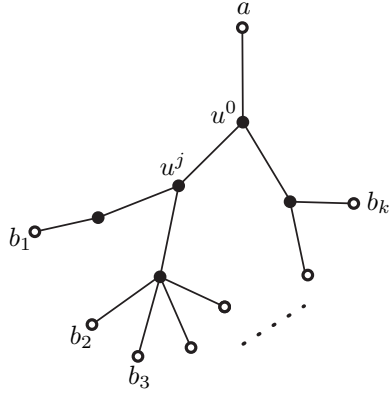}}}
  \relabel {1}{$a$}
  \relabel {2}{$u^0$}
  \relabel {3}{$u^j$}
  \relabel {7}{$b_1$}
  \relabel {8}{$b_2$}
  \relabel {9}{$b_3$}
  \relabel {10}{$b_k$}
  \endrelabelbox
  \caption{The tree structure of a broken disk with positive
    puncture at $a$ and negative punctures at $b_1, \ldots, b_k$.}
  \label{fig:disk-tree}
\end{figure}

\subsubsection{A gluing lemma}\label{sssec:floerpicard}

Let $u^{0},\dots,u^{r}$ be a broken $J$-holomorphic disk with negative
punctures at Reeb chords $b_1,\dots,b_k$, let $u^{j}$ represent an
element in $\ms^{j}=\ms_{A^{j}}(a^{j};b_1^{j},\dots,b_{k_j}^{j})$, and
write $\ms=\ms_{A}(a;b_1,\dots,b_k)$, where $A=\sum_j A^{j}$. In order
to equip $\overline{\ms}_{A}(a,b_1,\dots,b_k)$ with the structure of a
manifold with boundary with corners, we first construct a map
\begin{equation}\label{e:cornermap}
\Phi\colon \ms^{0}\times\dots\times\ms^{r}\times[0,\infty)^{r}\to\ms.
\end{equation}
The map $\Phi$ is defined by gluing the pieces of the broken disk into
a nearby holomorphic disk.  To construct the map, we will pass to a
more general functional analytic setup so that we can apply the
following result (known as Floer's Picard lemma):

\begin{lem}\label{lem:FloerPicard}
  Let $f\colon B_1\to B_2$ be a smooth Fredholm map of Banach spaces,
  \begin{equation}\label{e:TaylorFP}
    f(v)=f(0)+df(v)+N(v).
  \end{equation}
  Assume that $df$ is surjective and has a bounded right inverse $Q$
  and that the non-linear term $N$ satisfies a quadratic estimate of
  the form
  \begin{equation}\label{e:QuadraticFP}
    \|N(u)-N(v)\|_{B_2}\le C\|u-v\|_{B_1}(\|u\|_{B_1}+\|v\|_{B_1}).
  \end{equation}
  If $\|Qf(0)\|\le\frac{1}{8C}$, then for $\epsilon<\frac{1}{4C}$,
  $f^{-1}(0)\cap B(0;\epsilon)$, where $B(0;\epsilon)$ is an
  $\epsilon$-ball around $0\in B_1$, is a smooth submanifold
  diffeomorphic to the $\epsilon$-ball in $\krn(df)$.
\end{lem}

\begin{proof}
  The proof appears in Floer \cite{floer-mem}. Here we give a short
  sketch pointing out some features that we will use below. Let
  $K=\krn(df)$ and choose a splitting $B_1=B_1'\oplus K$ with
  projection $\pi\colon B_1\to K$. For $k\in K$, define $\hat f\colon
  B_1\to B_2\oplus K$, $\hat
  f(b_1)=\bigl(f(b_1),\pi(b_1)-k\bigr)$. Then solutions to the
  equation $f(b_1)=0$ with $\pi(b_1)=k$ are in $1-1$ correspondence
  with solutions to the equation $\hat f(b_1)=0$. Moreover, the
  differential $d\hat f$ is an isomorphism with inverse $\hat Q$.

  On the other hand, solutions of the equation $\hat f(v)=0$ are in
  $1-1$ correspondence with fixed points of the map $F\colon B_1\to
  B_1$ given by
  \[
  F(v)=v-\hat Q\hat f(v).
  \]
  To produce fixed points the Newton iteration scheme is applied: if
  \[
  v_0=k,\quad v_{j+1}=v_j-\hat Q\hat f(v_j),
  \]
  then $v_{j}$ converges to $v_\infty$ as $j\to\infty$ and
  $F(v_\infty)=v_{\infty}$. Furthermore, if $\|f(0)\|$ is sufficiently
  small then there is $0<\delta<1$ such that:
  \[
  \|v_{j+1}-v_j\|\le \delta^{j}\|f(0)\|
  \]
  and consequently
  \begin{equation}\label{e:iterest}
    \|v_\infty-v_0\|\le M\|f(0)\|,
  \end{equation}
  where $M$ is a constant.
\end{proof}

\subsubsection{Functional analytic setup}\label{sssec:fansetup}

The functional analytic setup we use is that of a bundle of Banach
manifolds of maps from the disk with boundary punctures into $P$ with
boundary on $\Pi_P(L)$ and with punctures mapping to double points of
$\Pi_P(L)$.  The base of the bundle is the space of conformal
structures on that disk; see \cite[Section 3.1]{ees:pxr}. A
neighborhood of a given map $w$ in this bundle is parametrized by an
exponential map acting on the product of a weighted Sobolev space
$\sblv_{2,\delta}$ of vector fields along $w$ and a finite dimensional
space of conformal variations $V_{\rm con}$ of the source. In this
setting, the $\bar\pa_J$-operator gives a Fredholm section of the
bundle of complex anti-linear maps from the tangent space of the
source disk of $w$ to the pull-back $w^{\ast}(TP)$. If
$\bar\pa_J(w)=0$ then the zero-set of this section gives a
neighborhood of $w$ in the moduli space.

To better understand the setup, let us discuss the domains of the
disks in greater detail. In order to gain better control of the
map $\Phi$ in \eqref{e:cornermap}, we add extra marked points to disks
as follows. Let $u\colon D_{m+1}\to P$ be a $J$-holomorphic disk with
boundary on $L$ and with positive puncture at the Reeb chord $c$ with
endpoints $c^{\pm}\in L$. Let $S_{c^{\pm}}^{\rm in}\subset L$ and
$S_{c^{\pm}}^{\rm out}\subset L$ be small concentric spheres in $L$
centered at $c^{\pm}$ where $S_{c^{\pm}}^{\rm out}$ lies outside
$S_{c^{\pm}}^{\rm in}$. Since $L$ is normalized at double points,
elementary Fourier analysis can be used in a straightforward manner to
derive asymptotics of $u$ near double points of $\Pi_P(L)$. The
asymptotics show that if $S_{c^{\pm}}^{\rm in/out}$ are chosen small
enough then for an open dense set of radii there are points
$p_{\pm}^{\rm in/out}\in \pa D_{m+1}$ which lie close to the positive
puncture of $u$ such that $u(p_{\pm}^{\rm in/out})\in S_{c^{\pm}}^{\rm
  in/out}$ and such that the intersections are
transverse. Consequently, there are such points for each
$J$-holomorphic $v$ in some neighborhood of $u$. Consider
$p_{\pm}^{\rm in/out}$ as punctures in the domain which are required
to map to $S_{c^{\pm}}^{\rm in/out}$, i.e. we view $u$ as a map
$u\colon D_{m+5}\to P$. In \cite[Section 4.2.3]{ees:ori} and
\cite[Section 8.6]{ees:high-d-analysis} it is explained how to
describe a neighborhood of $u$ in the moduli space in terms of disks
with such extra marked points corresponding to intersections. We will
use such disks with extra marked points below. Note that we add four
marked points to each disk, which makes the domain of any disk stable.

Consider a broken disk with pieces $u^{0},\dots,u^{k}$ as above.  We
represent the domain of a disk $u^{j}$ as the upper half plane
$H=\{x+iy\in\cc\colon y\ge 0\}$ with marked points on the boundary in
the following way: the puncture of $H$ at $\infty$ corresponds to the
positive puncture of $u^{j}$ and we take the punctures closest to the
positive puncture (i.e. the punctures corresponding to $p^{\rm
  in}_\pm$) to lie at $x=\pm 1$. Then the location of remaining
punctures in the interval $(-1,1)\subset\rr=\pa H$ determines the
conformal structure on the domain of $u^{j}$ uniquely. We will equip
such punctured half planes with a special metric for which the
neighborhood of any puncture looks like a half infinite strip. More
precisely, at the positive puncture we use the map
$(-\infty,0]\times[0,1]\to H$, $\tau+it\mapsto 2e^{-\pi(\tau+it)}$ to
identify the neighborhood with a half strip, at other punctures $q$ we
fix $r_0>0$ sufficiently small so that the interval on the real axis
of length $2r_0$ centered at $q$ does not contain other punctures and
use the map $[0,\infty)\times[0,1]\to H$, $\tau+it\mapsto
q+r_0e^{-\pi(\tau+it)}$ to identify the neighborhood with a half
strip. We use a Riemannian metric on this domain which agrees with the
standard metric on the half strips $(-\infty,-1]\times[0,1]$
(respectively $[1,\infty)\times[0,1]$) and which interpolates smoothly to
the flat metric on the compact part of $H$ which is the complement of
the open half strip neighborhoods. We denote the domain of $u^{j}$ by
$\Delta^{j}$.

We now give a more precise definition of the space of conformal
variations $V_{\rm con}^{j}$ of the source of $u^{j}$: conformal
variations of the domain of $u^{j}$ are linearizations of
diffeomorphisms which move the punctures along the boundary. We take
$V_{\rm con}^{j}$ to be the space spanned by conformal variations
supported near all punctures except the positive puncture and the two
punctures closest to it (these two correspond to $p^{\rm in}_\pm$). In
terms of half strip coordinates, a conformal variation of the source
has the form $\bar\pa(\beta v)$ where $\beta$ is a smooth cut-off
function equal to $1$ near the puncture and equal to $0$ on remaining
parts of the disk, and where $v$ is the holomorphic vector field given
by $e^{\pi z}$ in the half strip $[0,\infty)\times[0,1]$.

\subsubsection{Approximately holomorphic maps}\label{ssec:apphol}

In order to define $\Phi(u^{0},\dots,u^{r},\rho)$, where $u^{j}\in
\ms^{j}$ and where $\rho=(\rho_1,\dots,\rho_r)\in[0,\infty)^{r}$ has
all components $\rho_j$ sufficiently large, we first define a domain
$\Delta(\rho)$ which is conformally equivalent to the upper half plane
with punctures on the boundary but which has a metric somewhat
different from the metrics on $\Delta^{j}$.  Once we have that domain,
we will construct an approximately holomorphic map on
$\Delta(\rho)$. 

The punctures on a domain $\Delta^{j}$ are of two types: punctures
corresponding to interior edges of the tree $T$ and those
corresponding to edges coming from leaves. We call the first type of
punctures {\em bound punctures} and the second {\em free
  punctures}. Let $\hat\Delta^{0}(\rho)$ denote the subset of
$\Delta^{0}$ obtained as follows: for each bound puncture $q$ in
$\Delta^{0}$, cut off $(\rho_j,\infty)\times[0,1]$ from its half strip
neighborhood, where the index $j$ corresponds to the index of the disk
$u^{j}$ attached at $q$. For $j>0$, let $\hat\Delta^{j}(\rho)$ denote
the subset of $\Delta^{j}$ obtained as follows: at the positive
puncture cut off $(-\infty,-\rho_j+1)\times[0,1]$ and for each bound
puncture $q$ in $\Delta^{j}$, cut off $(\rho_k,\infty)\times[0,1]$
from its half strip neighborhood, where the index $k$ corresponds to
the index of the disk $u^{k}$ attached at $q$.

Join the domains $\Delta^{j}(\rho)$ to the domain $\Delta(\rho)$ by
identifying the vertical part of boundary of $\hat\Delta^{j}(\rho_j)$,
$j\ge 1$, which lies in the half strip corresponding to the positive
puncture with the vertical part of the boundary of
$\hat\Delta^{l}(\rho)$ lying in the half strip of the negative
puncture at which the positive puncture of $\Delta^{j}$ is
attached. Since metrics agree in gluing regions, this gives a domain
with a metric. In order to determine the conformal structure we note
that if we view all domains as upper half planes with punctures on the
boundary then the above construction corresponds to cutting out
half-disks of radii $e^{-\pi\rho_j}$ near bound punctures and
inserting scaled cut off half planes. This construction is repeated so
that further scaled half planes are attached in scaled half planes
already attached; see Figure \ref{fig:uhalfpl}.

\begin{figure}[ht]
  \relabelbox \small
  {\epsfxsize=3in\centerline{\epsfbox{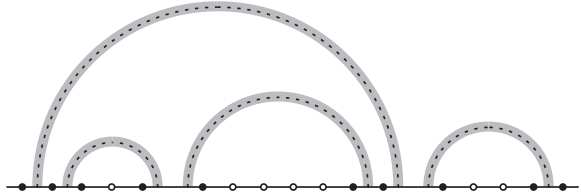}}}
  \endrelabelbox
  \caption{The glued domain for the broken disk in
    Figure~\ref{fig:disk-tree}. Shaded regions correspond to the half
    strips near the punctures.  Stationary punctures are filled,
    moving punctures are not.}
  \label{fig:uhalfpl}
\end{figure}

Consider next the space of conformal variations $V_{\rm con}(\rho)$ of
$\Delta(\rho)$. We will use a decomposition of this space of the form
\[
V_{\rm con}(\rho)= V_{\rm con}'(\rho)\oplus V_{\rm con}''(\rho).
\]
To define the decomposition we first note that $V_{\rm con}(\rho)$ is
spanned by conformal variations as described above supported near each
puncture except at the positive puncture and at the two punctures
closest to it. These punctures correspond to the positive puncture of
$u^{0}$ and the intersection points $p^{\rm in}_{\pm}$ of
$u^{0}$. Other punctures of the disk $\Delta(\rho)$ are of two types:
\begin{itemize}
\item {\em Moving punctures} which correspond to (free) punctures in
  $\Delta^{j}$ where some conformal variation in $V_{\rm con}^{j}$ was
  supported; see Figure \ref{fig:uhalfpl}.
\item {\em Stationary punctures} which correspond to (free) punctures
  in $\Delta^{j}$ where no conformal variation in $V_{\rm con}^{j}$
  was supported; see Figure \ref{fig:uhalfpl}.
\end{itemize}
Second, we use the tree structure on the broken disk. Let $T$ be the
directed tree corresponding to the broken disk. We take $V'_{\rm
  con}(\rho)$ generated by the conformal variations at all moving
punctures and the following additional variations
$\gamma^{j}$. Consider the conformal representative of $\Delta(\rho)$,
which is an upper half plane $H(\rho)$ with boundary punctures. Let
$u^{j}$ be a disk. For large $\rho_j$ the punctures of $\Delta(\rho)$
corresponding to the negative punctures of $u^{k}$ for all $u^{k}\in
T^{j}$ lie very close together and we let $\gamma^{j}$ denote the
conformal variation which correspond to moving all these punctures
uniformly along the boundary $\pa H$. In terms of the generators of
$V_{\rm con}(\rho)$ this is a linear combination of the variations at
all stationary and moving punctures coming from $\Delta^{k}$ for
$u^{k}\in T^{j}$. Finally, we take $V''_{\rm con}(\rho)$ to be
generated by the conformal variations $\theta_j$ which correspond to
moving all punctures corresponding to negative punctures of $u^{k}\in
T^{j}$ toward the mid-point of the subinterval $I^{j}\subset \rr=\pa
H$ bounded by the punctures corresponding to $p_{\pm}^{\rm in}$ in
$\Delta^{j}$, by scaling and which keeps all other punctures fixed.

Consider $H(\rho)$ with boundary punctures as just described and
define
\[
d(\rho,j)=d^{\rm in}_+(\rho,j)+d^{\rm in}_-(\rho,j),
\]
where $d^{\rm in}_{\pm}(\rho,j)$ is the distance between the puncture
corresponding to $p^{\rm in}_{\pm}(\rho,j)$ and the puncture at
$\pm1\in\rr=\pa H$. Then all conformal variations in $V'_{\rm
  con}(\rho)$ preserve $d(\rho,j)$ for $j=1,\dots,r$. Furthermore,
note that if $\tau=(0,\dots,x,\dots,0)$ where $x>0$ sits in the
$j^{\rm th}$ component and if $\rho'=\rho+x$ then
$d(\rho',k)>d(\rho,k)$ for all $k$ such that $u^{k}\in T^{j}$.

Define the approximately holomorphic map
\[
w_{u^{0},\dots,u^{r};\rho}\colon\Delta(\rho)\to P,
\]
by interpolating between the joined maps in the region
$[\rho+1,\rho]\times[0,1]$ as in \cite[Proof of Proposition
4.6]{ees:pxr}. 


\subsubsection{Definition and properties of the corner map}
\label{sssec:cornermap}

We apply Lemma \ref{lem:FloerPicard} to produce solutions near
$w_{u^{0},\dots,u^{r};\rho}$ for $\rho$ sufficiently large. We take
$f$ to be a restriction of the $\bar\pa_J$-operator in local
coordinates centered at $w_{u^{0},\dots,u^{r};\rho}$. More precisely,
trivializing the bundle of complex anti-linear maps in a neighborhood
of $w_{u^{0},\dots,u^{r};\rho}$ as described in \cite[Section
3.2]{ees:pxr} and using a trivialization of the tangent bundle of the
disk, we view the $\bar\pa_J$-operator as a map
\[
\bar\pa_J\colon \sblv_{2,\delta}\oplus V_{\rm con}(\rho)\to\sblv_{1,\delta},
\]
where $\sblv_{1,\delta}$ is the Sobolev space of vector fields along
$w_{u^{0},\dots,u^{r};\rho}$ with one derivative in $L^{2}$ weighted
by the same weight function used to define $\sblv_{2,\delta}$. As in \cite[Proof of Proposition
4.6]{ees:pxr}, we find:
\begin{equation}\label{e:f(0)}
\|\bar\pa_J w_{u^{0},\dots,u^{r};\rho}\|_{1,\delta}= \Ordo(e^{-\theta_0\min_{j}{\rho_j}}),
\end{equation}
where $\|\cdot\|_{1,\delta}$ is the norm on $\sblv_{1,\delta}$,
for some $\theta_0>0$ determined by the complex angles (see
\cite[Section 4.1]{ees:pxr}) at the double points of $\Pi_P(L)$.

We take $f$ of Lemma \ref{lem:FloerPicard} as the restriction of $\bar\pa_J$
to the subspace $\sblv_{2,\delta}\oplus V'_{\rm con}(\rho)$ so that
\[
f=\bar\pa_J\colon \sblv_{2,\delta}\oplus V_{\rm con}'(\rho)\to\sblv_{1,\delta},
\]
and by \eqref{e:f(0)}, $\|f(0)\|_{1,\delta}=\Ordo(e^{-\theta_0\min_{j}{\rho_j}})$

\begin{lem}\label{lem:partialglu}
  There exists $\rho_0>0$ such that for
  $\rho=(\rho_1,\dots,\rho_r)\in[\rho_0,\infty)^{r}$, the differential
  $df$ is surjective and admits a bounded right inverse. Moreover, the
  quadratic estimate \eqref{e:QuadraticFP} for the non-linear term in
  the Taylor expansion of $f$ holds.
\end{lem}

\begin{proof}
  The quadratic estimate follows from \cite[Proof of Proposition
  4.6]{ees:pxr}. In order to see that the differential is surjective
  and admits a bounded right inverse, we first note that it is a
  Fredholm operator of index $d=\sum_j\dim(u^j)$. Second, consider the
  subspace $K$ of cut-off kernel functions corresponding to all pieces
  of the disk. It follows that $\dim(K)=d$. Note that the component of
  the cut off kernel element in the disk $\Delta^{j}$ along the
  conformal variation supported at a negative puncture where the disk
  $\Delta^{l}$ is attached corresponds to the conformal variation
  $\gamma^{j}$.  Now a standard argument (see, for example,
  \cite[Lemma 8.9]{ees:high-d-analysis}) shows that on the subspace
  which is the $L^{2}$-complement of $K$, there is an estimate
  \[
  \|df(v)\|\ge C\|v\|,
  \]
  provided the components of $\rho$ are sufficiently large.
\end{proof}

With Lemma \ref{lem:partialglu} established, define
$\Phi_\rho(u^{0},\dots,u^{r};\rho)$ as the result of applying Newton
iteration to $w_{u^{0},\dots,u^{r};\rho}$ as in the proof of Lemma
\ref{lem:FloerPicard}. Combining the above maps for all
$\rho\in[\rho_0,\infty)^r$ with sufficiently large $\rho_0$ and
reparametrizing $[0,\infty)\approx[\rho_0,\infty)$, we get the map
\[
\Phi\colon\ms^{0}\times\dots\dots\ms^{r}\times[0,\infty)^{r}\to\ms.
\]
\begin{lem}\label{lem:corneremb}
The Newton iteration map $\Phi$ is an embedding.
\end{lem}
\begin{proof}
  Let $u=(u^{0},\dots,u^{k})$. We first look at the conformal
  structure of the domain of $\Phi(u;\rho)$: by construction
  $d^{j}(\rho)$ is independent of $u$. If $\rho\ne \rho'$ then let $j$
  be an index such that the $\rho_j\ne \rho'_j$ and such that
  $\rho_l=\rho_l'$ for all $l\ne j$ with $u^{j}\in T^{l}$ and it
  follows that the conformal structures of $\Phi(u;\rho)$ and
  $\Phi(u';\rho')$ are different if $\rho\ne \rho'$.

  It is then a consequence of Lemma \ref{lem:FloerPicard} that $\Phi$
  is a local embedding and, if $\Phi(u;\rho)=\Phi(u';\rho)$, then the
  estimate \eqref{e:iterest} implies that $w_{u';\rho}$ lies in a
  small neighborhood of $w_{u;\rho}$.  Thus, $\Phi$ is an actual
  embedding.
\end{proof}

\subsubsection{Corner structure}
\label{sssec:cornerstr}

With Lemma \ref{lem:corneremb} established, we can define the
compactification of the moduli space $\ms$ by adding broken
configurations at corners with the Newton iteration map as a local chart. To discuss this, assume that $\ms$ and $\ms^{0},\dots,\ms^{r}$ are as in Lemma \ref{lem:corneremb} and let
\[
  \widehat{\Phi} \colon
  \ms^{0}\times\dots\times\ms^{r}\times[0,\infty)^{r}\to\ms,
\]
denote the Newton iteration embedding. Assume further that for each $j=1,\dots,r$ there are broken disks in $\ms^{j;1}\times\dots\times\ms^{j;k_j}$ which can be joined to a disk in $\ms^{j}$. Let
\[
  \Phi^{j} \colon
  \ms^{j;0}\times\dots\times\ms^{j;k_j}\times[0,\infty)^{k_j}\to\ms^{j},\quad j=0,\dots,r,
\]
denote the corresponding Newton iteration maps.
Finally, Newton iteration can be applied directly to broken disks in
\[
(\ms^{0;0}\times\dots\times\ms^{0;k_0})\times\dots\times(\ms^{r;0}\times\dots\ms^{r;k_r})
\]
and produces disks in $\ms$. Let
\begin{align*}
  \Phi \colon
  &\ms^{0;0}\times \cdots \times
  \ms^{0;k_0}\times\dots\times\ms^{r;0}\times\dots\times\ms^{r;k_r} \\
  &\quad \quad \times[0,\infty)^{k_0}\times\dots\times[0,\infty)^{k_r}\times[0,\infty)^{r} \to \ms,
\end{align*}
denote the corresponding embedding. 

Combining the maps $\Phi^{j}$ and $\widehat{\Phi}$, we get an
embedding:
\[
\Psi\colon
\ms^{0;0}\times\dots\times\ms^{0;k_0}\times\dots\times\ms^{r;0}\times\dots\times\ms^{r;k_r}
\times[0,\infty)^{k_0}\times\dots\times[0,\infty)^{k_r}\times[0,\infty)^{r} \to \ms,
\]
with
$\tilde\rho=(\rho_0,\dots,\rho_{r})\in[0,\infty)^{k_0}\times\dots\times[0,\infty)^{k_r}$ and $\rho\in[0,\infty)^{r}$
\begin{align*}
&\Psi\bigl(u^{0;0},\dots,u^{0;k_0},\dots,
u^{r;0},\dots,u^{r;k_r};\tilde\rho,\rho\bigr)=\\
&\quad \quad \widehat\Phi\bigl(\Phi^{0}(u^{0;0},\dots,u^{0;k_0};\rho^{0}),\dots,
\Phi^{r}(u^{r;0},\dots,u^{r;k_r};\rho^{r});\rho\bigr).
\end{align*}

We claim that the maps $\Psi\circ\Phi^{-1}$ and $\Phi\circ\Psi^{-1}$
are $C^{1}$ diffeomorphisms. To see this consider the definitions of
the gluing maps involved and note that $L^{2}$-projections define
linear isomorphisms between the spaces spanned by cut off kernel
functions of the differentials on pieces of a broken disk and the
kernel of the differential of the approximately holomorphic disk which
is the starting point for the Newton iteration of Lemma
\ref{lem:FloerPicard}; see the proof of Lemma \ref{lem:partialglu}.
Since the kernel of the differential at the approximately holomorphic
disk is the domain of the chart, the composition statements follow
since the corresponding statements on the level of cut off kernel
functions clearly hold.

In order to define the structure of a manifold with boundary with
corners identify $[0,\infty)$ with $[0,1)$. Consider the inclusion
$[0,1)\subset[0,1]$, and complete the map $\Phi$ to a chart
\[
\overline{\Phi}\colon\ms^{0}\times\dots\times\ms^{r}\times[0,1]^{r}\to\ms
\]
by identifying the points $(u^{0},\dots,u^{r};\rho)$ with
$\rho=(\rho_1,\dots,\rho_r)$, where some of the $\rho_{j}$ equal $1$,
as follows. Consider the tree $T$ corresponding to the broken
disk $u^0, \ldots, u^r$.  For each $s$ such that $\rho_s = 1$, replace
each edge from the vertex $s$ to the vertex $j$ by two edges, one of
which connects $s$ to a new root and the other of which connects $j$
to a new leaf that has the same Reeb chord label as the new leaf.  For
edges leaving $s$, the role of roots and leaves are reversed. See
Figure~\ref{fig:glue-tree}.  Each component of the new graph
corresponds to a broken disk $u^{l;0},\dots,u^{l;l_k}$ and
$\rho^{l}=(\rho_{l;0},\dots,\rho_{l:l_s})$, where $\rho^{l:j}$ is the
component of $\rho$ corresponding to $u^{l;j}$.  There is a
corresponding chart
\[
\Phi^l\colon \ms^{l;0}\times\ms^{l;l_s}\times[0,1)^{l_s}\to\ms^l.
\]
Identify the point $(u^{0},\dots,u^{r};\rho)$ with the broken disk
with components $\Phi^l(u^{l;0},\dots,u^{l;l_s};\rho^{l})$,
$l=1,\dots,k$ and use the coordinate chart $\Phi$ which glues this
configuration to define a coordinate neighborhood.  The compatibility
between $\Phi$ and $\Psi$ discussed above then ensures that this
construction gives a well-defined $C^{1}$-structure.

\begin{figure}[ht]
  \relabelbox
  \small{\epsfxsize=4in\centerline{\epsfbox{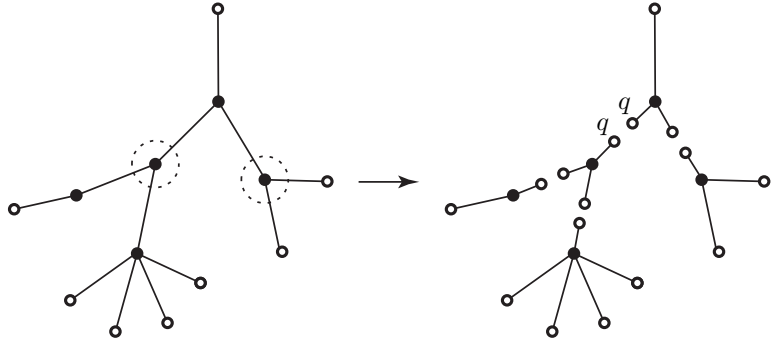}}}
  \relabel {1}{$q$}
  \relabel {2}{$q$}
  \endrelabelbox
  \caption{Splitting the broken disk into sub-disks when $\rho_s =
    1$.}
  \label{fig:glue-tree}
\end{figure}


In order to build the compactified moduli space, we use the area
filtration.  The moduli space of disks of smallest area is a compact
manifold since no disk in this space can break into smaller pieces. In
fact, the moduli space of disks of second smallest area is a compact
manifold as well since a disk with only one positive puncture cannot
be glued to itself. In order to define the boundaries of moduli spaces
of disks of a given area $A$ we use the gluing maps discussed above
applied to moduli spaces of disks of area smaller than $A$. The
boundaries of the smaller area moduli spaces were constructed in
previous steps and by the compatibility discussed above this
construction yields a manifold with boundary with corners of class
$C^{1}$. Note though that the $C^{1}$-structure is non-canonical since
it depends on our specific choices of conformal models of the domains, on the interpolations between maps, and on the choice of identification of $[0,\infty)$ with $[0,1)$.

\subsection{Evaluation maps}\label{s:evmaps}

The generalized disks that appear in Theorem~\ref{thm:disk+Morse}
require the use of evaluation maps on compactified moduli spaces of
$J$-holomorphic disks which in turn requires a slight extension of the
definition of compactified moduli space above. As in Subsection
\ref{s:cmdli}, we add marked points near positive punctures in order
to get stable domains of all $J$-holomorphic disks. We then add one
additional marked point on the boundary. If $\ms$ is the original
moduli space, we let $\ms^{\ast}$ denote the corresponding moduli space
of disks with one extra marked point on the boundary. Note that we
have a fibration $\ms^{\ast}\to\ms$. Furthermore, there is a natural
evaluation map
\[
\ev\colon\ms^{\ast}\to L,
\]
which corresponds to evaluation of the boundary lift of $u$ at the
marked point and this map has a natural extension to configurations
where the extra marked point agrees with one of the extra punctures by
evaluating at the puncture. It is straightforward to extend the notion
of a disk with marked point on the boundary to broken configurations:
here the marked point is in one of pieces of the disk. Furthermore,
the gluing map $\Phi$ discussed in Subsection \ref{s:cmdli} can be
applied also to broken disks with marked points to give a disk with
marked point as long as the marked point lies outside the piece which
is cut off from the domain. However, for a fixed broken configuration
with a marked point on the boundary, there is a $\rho_0$ such that if
the components of $\rho$ in Lemma \ref{lem:corneremb} are all bigger
than $\rho_0$ then the map $\Phi$ gives an embedding into the moduli
space of disks with one extra marked point. Thus, there is a
compactification $\overline{\ms^{\ast}}$ of the space $\ms^{\ast}$,
where the boundary has two types of strata: one corresponds to broken
disks with one marked point, and the other to the marked point lying
at Reeb chord endpoints. As in Subsection \ref{s:cmdli}, this
compactified space has the structure of a $C^{1}$ manifold with
boundary with corners and the evaluation maps fit together to give a
$C^{1}$-map
\[
\ev\colon\overline{\ms^{\ast}}\to \Pi_P(L).
\]
See \cite{fooo} for a discussion of evaluation maps in a similar situation.

\subsection{Morse flows and generalized disks}\label{s:morseflow}

\subsubsection{Transversality}

The next step in proving Theorem~\ref{thm:disk+Morse} is to find
triples $(f,g,J)$ where $f\colon L\to\rr$ is a Morse function, $g$ is
a Riemannian metric on $L$, and $J$ is an almost complex structure on
$P$ which have the following properties:
\begin{itemize}
\item[$({\bf a1})$] The almost complex structure $J$ is adjusted to
  $L$ and $L$ is normalized at double points; see Subsection
  \ref{ssec:diffl}.
\item[$({\bf a2})$] Condition $({\bf a1})$ implies that there are Reeb
  chord coordinates at the Reeb chord endpoints in $L$ in which the
  Lagrangian projection is a linear embedding as a map into double
  point coordinates; see Subsection \ref{ssec:diffl} for notation. The
  metric $g$ agrees with the standard flat $\rr^{n}$ metric in Reeb
  chord coordinates.
\item[$({\bf a3})$] There is a symplectic neighborhood map $\Phi\colon
  U\to P$, where $U\subset T^{\ast}L$ is a neighborhood of the
  $0$-section, which is linear in the canonical coordinates
  corresponding to Reeb chord coordinates as a map into double point
  coordinates, and which is such that the almost complex structure on
  $U$ induced by $g$ is equal to $\Phi^{\ast}J$. (In other words, $J$
  is standard in a neighborhood of $\Pi_P(L)$ with respect to $g$.)
\item[$({\bf a4})$] The Morse function $f\colon L\to\rr$ is real
  analytic and regular in the Reeb chord coordinate neighborhoods; see
  Lemma \ref{lem:2pos}.
\item[$({\bf a5})$] If $p$ is a critical point of $f$, then the
  eigenvalues of the Hessian of $f$ at $p$ (i.e. the critical values
  of the Hessian quadratic form restricted to the unit sphere in
  $T_pL$ determined by $g$) all have the same absolute value. If a
  Morse function has this property we say that it is {\em round at
    critical points with respect to $g$}.
\end{itemize}
We call triples $(f,g,J)$ with properties $({\bf a1})-({\bf a5})$ {\em
  adjusted to $L$}.

Let $L\subset P\times\rr$ be chord generic, let $J$ be an almost
complex structure, and let $g$ be a Riemannian metric on $L$ such that
$({\bf a1})-({\bf a3})$ above hold and such that Lemma
\ref{lem:tvmspc} holds. Let $\overline{\ms^{\ast}}$ denote the moduli
space of $J$-holomorphic disks with one positive puncture and with a
marked point on the boundary and let
$\ev\colon\overline{\ms^{\ast}}\to L$ denote the evaluation map. If
$f\colon L\to\rr$ is a Morse function such that the $g$-gradient of
$f$ defines a Morse-Smale flow and if $p\in L$ is a critical point of
$f$, then the stable and unstable manifolds $W^{\rm s}(p)$ and $W^{\rm
  u}(p)$ are submanifolds of $L$ with natural compactifications
consisting of broken (and constant) flow lines.

\begin{lem}\label{lem:morsetv} 
  There exists a Morse function $f\colon L\to\rr$ with $g$-gradient
  which is Morse-Smale such that the triple $(f,g,J)$ is adjusted to
  $L$ and such that, for each critical point $p$ of $f$, $W^{\rm
    s}(p)$ and $W^{\rm u}(p)$ are stratumwise transverse to $\ev\colon
  {\overline{\ms^{\ast}}}\to L$. In particular, generalized disks defined by
  $(f,g,J)$ have the following properties:
  \begin{itemize}
  \item[$({\bf g1})$] There are no generalized disks of formal
    dimension $<0$.
  \item[$({\bf g2})$] All generalized disks of dimension $0$ are
    transversely cut out.
  \end{itemize}
Furthermore, $(f,g,J)$ can be chosen so that the following condition holds.
\begin{itemize}
\item[$({\bf g3})$] Rigid generalized disks defined by $(f,g,J)$
  intersect in general position in the following sense.
  \begin{itemize}
  \item If $\dim(L)>2$ then the interior of the $J$-holomorphic disk
    part of any rigid generalized disk is disjoint from $\Pi_P(L)$,
    boundary lifts of two distinct $J$-holomorphic disks which are
    parts of rigid generalized disks are disjoint and if the boundary
    lifts of two rigid generalized disks intersect then these two
    generalized disks have either have the same disk part of flow line
    parts which are subsets of the same flow line.
  \item If $\dim(L)=2$ then the interior of the $J$-holomorphic disk
    part of any rigid generalized disk is disjoint from the
    $\Pi_P$-image of the boundary lift of any rigid generalized disk,
    boundary lifts of two distinct $J$-holomorphic disks which are
    parts of two rigid generalized disks intersect transversely in
    finitely many points and if the if the boundary lifts of two rigid
    generalized disks intersect non-transversely or in infinitely many
    points, then these two generalized disks have either have the same
    disk part of flow line parts which are subsets of the same flow
    line.
  \end{itemize}
\end{itemize}
\end{lem}

\begin{proof}
  Fix a Morse function with Morse-Smale $g$-gradient for which $({\bf
    a4})$ and $({\bf a5})$ hold, and for which the critical points of
  $f$ are in general position with respect to
  $\ev(\overline{\ms^{\ast}})$. As in \cite{smale}, linearized
  perturbations of $f$ supported in small neighborhoods of spheres
  around the critical points span the normal bundles of all stable and
  unstable manifolds. Consequently, we may achieve the desired
  transversality by perturbation of $f$ near the critical points. The
  technical details of this argument are analogous to
  \cite{smale}, so we leave them out.
\end{proof}

We say that the \dfn{special points} of a rigid generalized disk
consist of the junction points and critical points; in case
$\dim(L)=2$, we also add in the intersection points of lifts of rigid
generalized disks.  We call triples $(f,g,J)$ which are adjusted to
$L$ and for which $({\bf g1})-({\bf g3})$ hold {\em generic with
  respect to rigid generalized disks}.


\subsubsection{Normal Forms}

We next show that we can deform a triple $(\hat f,\hat g,\hat J)$
which is
generic with respect to rigid generalized disks to a triple $(f,g,J)$
which maintains those properties and has a certain normal form near
special points. This will allow us to apply the techniques of
\cite{ekholm:morse-flow} during our analysis of the correspondence
between rigid generalized disks and $J$-holomorphic disks with
boundary on the perturbed $2L$.

First, we add more special points: consider a flow line segment
$\gamma$ which connects a critical point and a junction point. Note
that any such segment has finite length and add a finite number of
points along $\gamma$ so that they are of distance at most $\eta>0$
apart. Call these new points {\em special points} as well.

We say that a triple $(f,g,J)$ which is generic with respect to
rigid generalized disks is {\em semi-normalized} if the following
conditions hold.
\begin{itemize}
\item[$({\bf n1})$] If $p$ is a critical point of $f$ then there are
  coordinates $x=(x_1,\dots,x_n)$ on a neighborhood $U(p)$ around $p$
  in which $f$ and $g$ are given by
    \begin{equation*}
    f(x)=c+ \sigma_1 x_1^{2}+\dots+\sigma_n x_n^{2},\quad
    g(x)=\sum_j dx_j\otimes dx_j,
    \end{equation*}
    where $|\sigma_j|=\sigma> 0$, $j=1,\dots,n$.
  \item[$({\bf n2})$] If $q$ is a special point which is not a critical point (or an intersection
    point if $\dim(L)=2$) of a generalized disk, then there are
    coordinates $x=(x_1,\dots,x_n)$ on a neighborhood $U(q)$ around
    $q$ in which $f$ and $g$ are given by
    \begin{equation*}
    f(x)= c+\mu x_1,\quad 
    g(x)=\sum_j dx_j\otimes dx_j,
    \end{equation*}
    for constants $c,\mu$, where $\mu\ne 0$.
\end{itemize}
Note that $({\bf g1})-({\bf g3})$ of Lemma \ref{lem:morsetv} hold for semi-normalized triples.

\begin{lem}\label{lem:seminormal} 
  Let $\hat f\colon L\to\rr$ be a Morse function, let $\hat g$ be a
  metric, and let $\hat J$ be an almost complex structure such that
  the triple $(\hat f,\hat g,\hat J)$ is generic with respect to rigid
  generalized disks. Then there exist a Morse function $f\colon L\to\rr$, a
  metric $g$, and an almost complex structure $J$ such that the triple
  $(f,g,J)$ is semi-normalized and has the following property. There
  is a $\delta>0$ and a bijective correspondence between the rigid
  generalized disks determined by $(\hat f,\hat g,\hat J)$ and rigid
  generalized disks determined by $(f,g,J)$ such that there is exactly
  one rigid generalized disk determined by one of the triples with a
  lift in a $\delta$-neighborhood of a given rigid generalized disk
  determined by the other triple.
\end{lem}

\begin{proof} 
  We start by making a general remark about almost complex structures
  induced by Riemannian metrics. Let $\hat g$ and $g$ be Riemannian
  metrics on $L$ and let $\hat J$ respectively $J$ denote the
  corresponding almost complex structures induced on $T^{\ast}L$.  In
  order to discuss distances, we fix a reference metric $h$ on $L$ and
  note that it induces a metric on all tensor bundles of $L$ as well
  as on $T^{\ast}L$ and on its tensor bundles. Let
  $d_{C^{j}}(T_1,T_2)$ denote the $C^{j}$-distance between tensor
  fields $T_1$ and $T_2$ as measured with respect to $h$.
  Consider a
  small $\epsilon>0$.  It follows from the local coordinate formula
  \eqref{e:Jloccoord} that if $d_{C^{1}}(\hat g,g)=\Ordo(\epsilon)$
  and if $d_{C^{2}}(\hat g,g)=\Ordo(1)$ then, in an
  $\epsilon$-neighborhood of the $0$-section,
  \begin{align}\label{e:JC0}
    d_{C^{0}}(\hat J,J)&=\Ordo(\epsilon^{2})\\\label{e:JC1}
    d_{C^{1}}(\hat J,J)&=\Ordo(\epsilon).
  \end{align}

  Consider a critical point $p$ of $\hat f$. Choose normal coordinates
  in an $\epsilon$-ball around $p$. Define $f$ to equal the second
  degree Taylor polynomial of $\hat f$ in these coordinates in a
  neighborhood of $0$ and define $g$ by letting the metric be constant
  in a small neighborhood of $0$. Then $d_{C^{2}}(f,\hat
  f)=\Ordo(\epsilon^{3})$, $d_{C^{1}}(\hat g,g)=\Ordo(\epsilon)$, and
  $d_{C^{2}}(\hat g,g)=\Ordo(1)$. We conclude that stable and unstable
  manifolds determined by $f$ and $g$ are at $C^{1}$-distance
  $\Ordo(\epsilon)$ from those determined by $\hat f$ and $\hat g$.

  Let $J$ be an almost complex structure which is standard with
  respect to $g$ in an $\epsilon$-neighborhood of the $0$-section and
  which is such that $d_{C^{1}}(\hat J,J)=\Ordo(\epsilon)$; see
  \eqref{e:JC1}. If $J'$ is an almost complex structure, write
  $\ms_{J'}$ for the moduli space of $J'$-holomorphic disks.

  As in Subsection \ref{s:cmdli}, we consider $\ms_{J'}$ as the
  $0$-sets of a $\bar\pa_{J'}$-operator acting on a bundle of Sobolev
  spaces. In order to compare the evaluation maps
  $\ev\colon\overline{\ms^{\ast}_{\hat J}}\to L$ and
  $\ev\colon\overline{\ms^{\ast}_J}\to L$, we use weighted Sobolev
  spaces $W^{2,p}$ with two derivatives in $L^{p}$, $p>2$. Since
  $d_{C^{1}}(\hat J,J)=\Ordo(\epsilon)$, we find that for $\epsilon>0$
  small enough, the zero sets of $\bar\pa_J$ and $\bar\pa_{\hat J}$
  are $C^{1}$-diffeomorphic and zeros of the $\bar\pa_J$- and
  $\bar\pa_{\hat J}$-operator are at distance $\Ordo(\epsilon)$ in the
  corresponding Sobolev norm. Since the $W^{2,p}$-norm controls the
  $C^{1}$-norm it follows, using this argument at each step in the
  inductive construction of the compactification of moduli spaces as
  manifolds with boundary with corners that the moduli spaces are
  $C^{1}$-diffeomorphic and that $\ev(\overline{\ms_{\hat J}^{\ast}})$
  and $\ev(\overline{\ms_{J}^{\ast}})$ are
  $C^{1}$-close. Combining the fact that both Morse flows and evaluation maps determined by the triples $(\hat f,\hat g,\hat J)$ and $(f,g,J)$, respectively, are arbitrarily $C^{1}$-close with the genericity of $(\hat f,\hat g,\hat J)$ with respect to rigid generalized disks, we draw the following two conclusions. First, the triple $(f,g,J)$ is generic with respect to rigid generalized disks. Second, there is a $1-1$ correspondence between
  rigid generalized disks determined by $(\hat f,\hat g,\hat J)$ and
  those determined by $(f,g,J)$ with properties as claimed in the statement. 

  Relabel the triple $(f,g,J)$ just constructed and which has desired
  properties near critical points $(\hat f,\hat g,\hat J)$ and
  consider a special point $q$ which is not a critical point. Choose normal coordinates in an
  $\epsilon$-ball around $q$, replace $\hat f$ by its first degree
  Taylor polynomial in this ball, and let $g$ be the flat metric in
  the normal coordinates. Since the flow time that a flow line spends
  in the region where the function is deformed is $\Ordo(\epsilon)$,
  since the deformation of the vector field is $\Ordo(\epsilon)$ in
  $C^{0}$-norm and $\Ordo(1)$ in $C^{1}$-norm, and since the
  deformation of the metric is $\Ordo(\epsilon)$ in $C^{1}$-norm, we
  conclude that the resulting deformation of the Morse flow is
  $\Ordo(\epsilon^{2})$ in $C^{0}$-norm and $\Ordo(\epsilon)$ in
  $C^{1}$-norm.

  We next derive a corresponding result for the induced deformation of
  evaluation maps.  First consider the moduli spaces as $0$-sets in
  bundles of weighted Sobolev spaces $W^{1,p}$ with one derivatives in
  $L^{p}$, $p>2$. Using \eqref{e:JC0}, we find that the zero sets of
  $\bar\pa_J$ and  $\bar\pa_{\hat J}$ are at distance
  $\Ordo(\epsilon^{2})$ in the corresponding Sobolev norm. Since the
  $W^{1,p}$-norm controls the $C^{0}$-norm, the image of a smooth
  component of an evaluation map lies in an
  $\Ordo(\epsilon^{2})$-neighborhood of the other. This shows in
  particular that if $\epsilon>0$ is sufficiently small then
  $\ev(\overline{\ms^{\ast}_J})$ intersects the $\epsilon$-ball of
  normal coordinates since $\ev(\overline{\ms^{\ast}_{\hat J}})$
  passes through its center. Second, we argue as above, using weighted
  Sobolev spaces $W^{2,p}$ to show that the $C^{1}$-distance between
  the evaluation maps is $\Ordo(\epsilon)$. Together with the above
  estimates on the Morse flow, it follows as in the argument at critical points above that $(f,g,J)$ is generic with respect to rigid generalized disks and there is a unique
  generalized disk of $(f,g,J)$ near every generalized rigid disk of
  $(\hat f,\hat g,\hat J)$.

  If $\dim(L)=2$ then an argument similar to that used for special
  points shows that it is possible to normalize at intersection points
  introducing only a small perturbation of rigid generalized disks.
\end{proof}

Consider a semi-normalized triple $(\hat f,\hat g,\hat J)$ as above
and the corresponding $1$-parameter family of Morse functions $\hat
f_\lambda=\lambda \hat f$, $0<\lambda\le 1$. We next deform the metric
$\hat g$ and the $1$-parameter family of Morse functions $\hat
f_\lambda$, $0<\lambda\le 1$ in a neighborhood of the rigid
generalized disks defined by $(\hat f,\hat g,\hat J)$ in order to
facilitate the construction of holomorphic disks from generalized
disks, which will be carried out in Subsection
\ref{ssec:gen-to-hol}. We will do this maintaining genericity with
respect to rigid generalized disks and without changing rigid
generalized disks.

We say that $(f_\lambda,g,J)$, where $f_\lambda$, $0<\lambda\le 1$ is
a $1$-parameter family of Morse function, is {\em normalized} if
$(f_\lambda,g,J)$ is semi-normalized for each $\lambda$ and if the
following conditions hold:
\begin{itemize}
\item[$({\bf n3})$] Outside a neighborhood of the flow lines of rigid generalized disks $f_\lambda=\lambda f$ for some function $f$.
\item[$({\bf n4})$] Let $\gamma$ be a flow line of a rigid generalized
  disk ending or beginning at a critical point $p$. Choose coordinates
  $(x_1,\dots,x_n)$ around $p$ as in $({\bf n2})$  such that $\gamma$
  corresponds to the $x_1$-axis. Then there is a {\em plateau point}
  $q=(a,0,\dots,0)$ on $\gamma$ (which lies between the critical point
  and any special point on $\gamma$) such that
  \begin{equation*}
    f_\lambda(q+y)=
    \begin{cases}
      \lambda(c+ \sigma_1 a^{2}) + 2\sigma_1 ay_1 &\text{ for }y_1\ge 0,\\
      \lambda(c + \sigma_1(a+y_1)^{2}+\sum_{j=2}^{n}\sigma_j y^{2}_j)
      &\text{ for }y_1<-K\lambda,
    \end{cases}
  \end{equation*}
  where $K>0$ is a large constant. 
\end{itemize}

Assume that there is a plateau point near every critical point and let
$\gamma$ be a flow line of a rigid generalized disk. Then parts of
$\gamma$ between special points, between special points and plateau
points, or between plateau points are finite flow segments
$\hat\gamma$ with endpoints $p_1$ and $p_2$ around which there are
coordinates $(x_1,\dots,x_n)\in\rr^{n}$ in which $g$ is the standard
Euclidean metric and in which function looks likes $f(x)=c_j+k_jx_1$,
$j=1,2$. Pick two inflection points $q_1$ and $q_2$ between any two
special points $p_1$ and $p_2$. Assume that the order of these points
along $\gamma$ is $p_1,q_1,q_2,p_2$.
\begin{itemize}
\item[$({\bf n5})$] In a neighborhood of the inflection points $q_1$
  and $q_2$ there are coordinates
  $u=(u_1,\dots,u_n)$ and $v=(v_1,\dots,v_n)$ such
  that the metric $g$ is the flat metric $g(u)=\sum_{j}
  du_j\otimes du_j$ and $g(v)=\sum_{j}
  dv_j\otimes dv_j$, and such that
  \begin{align*}
    f_\lambda(u)&=
    \begin{cases}
      \lambda(c_{1}+k_{1}u_1), &\text{ for }0\le u_1\le K\lambda\\
      \lambda(c_{1}+k_{1}u_1-\alpha u_1^{2}), &\text{ for } -2K\lambda\le u_1\le -K\lambda, 
    \end{cases}\\
    f_\lambda(v)&=
    \begin{cases}
      \lambda(c_{2}+k_{2}v_1-\alpha v_1^{2}), &\text{ for }0\le u_1\le K\lambda\\
      \lambda(c_{2}+k_{2}v_1), &\text{ for } -2K\lambda\le u_1\le -K\lambda, 
    \end{cases}\\
  \end{align*}
  where $k_1$ and $k_2$ are determined by the nearby special points as
  explained above and where
  $\alpha=\frac{k_1^{2}-k_2^{2}}{2(c_2-c_1)}$. We call the region
  $-K\lambda\le y_1\le 0$ the {\em inflection region}.
\item[$({\bf n6})$] There are flow coordinates
  $x=(x_1,\dots,x_n)=(x_1,x')$ in a neighborhood $U(\hat \gamma)$ of
  $\hat \gamma$ such that the following hold. The metric is the
  standard flat metric $g(x)=\sum_j dx_j\otimes dx_j$.  The flow line
  $\gamma$ corresponds to $\{x\colon 0\le x_1\le T,x'=0\}$ and outside
  a finite number of {\em inflection regions} of the form $\{x\colon
  a\le x_1\le a+K\lambda\}$ for some large constant $K>0$ the function
  $f$ is given by
    \[
    f_\lambda(x)=\lambda(c+\mu x_1+\nu x_1^{2}),
    \]
    for constants $c,\mu,\nu$ where $|\mu|+|\nu|\ne 0$. 
\end{itemize}

\begin{rem}
  By $({\bf n3})$, outside the interpolation neighborhood, the
  functions $f_\lambda$ are simply scalings of a function and we
  define a gradient line of $f_\lambda$ in this region as a gradient
  line of $f$. Inside the interpolation region the functions converge
  after rescaling to a piecewise linear Lagrangian. We define the
  Morse flow lines of $f_\lambda$ as the flow lines of the rescaled
  limit function. We also point out that a normalized triple
  $(f_\lambda,g,J)$ gives rise to a non $C^{1}$ gradient field in the
  re-scaled limit obtained by scaling the fiber only. For local
  questions about holomorphic curves it is more useful to scale also
  the base and under such scaling the limit is smooth.
\end{rem}

\begin{lem}\label{lem:normal} 
  Let $(\hat f,\hat g,\hat J)$ be a semi-normalized triple and let
  $\hat f_\lambda=\lambda \hat f$, $0<\lambda\le 1$. Then there exists
  a normalized triple $(f_\lambda,g,J)$, $0<\lambda\le 1$ such that
  $({\bf n3})$ hold with $f=\hat f$, with no generalized disks of formal dimension $<0$, and with rigid generalized disks of dimesion $0$ which are
  identical to the rigid generalized disks of $(\hat f,\hat g,\hat
  J)$.
\end{lem}

\begin{proof}
  To prove this lemma we first note that a $C^{2}$-small deformation
  $\tilde f_\lambda$ of $\hat f_\lambda$ near critical points allows
  us to introduce plateau points and furthermore that this can be done
  without altering any flow line which is part of a rigid generalized
  disk. It then follows that the rigid generalized disks of $(\tilde
  f_\lambda,\hat g,\hat J)$ agree with those of $(\hat f,\hat g,\hat
  J)$.

  With plateau points introduced, we apply \cite[Lemmas 4.5 and
  4.6]{ekholm:morse-flow} to achieve $({\bf n5})$. More precisely,
  these lemmas allow us to find coordinates $(u_1,\dots,u_n)$ around
  $\hat\gamma$ which agree with the coordinates already defined at the
  endpoints of $\hat\gamma$ and in which $\tilde f_\lambda(u)=Q(u_1)$
  where $Q$ is a quadratic polynomial. The coordinates
  $(u_1,\dots,u_n)$ are constructed from Fermi coordinates on level
  surfaces of $\tilde f_\lambda$ perpendicular to $\hat\gamma$ and
  hence the metric
  \[
  g(u)=\sum_j du_j\otimes du_j
  \]
  in a neighborhood of $\hat\gamma$ is at $C^1$-distance $\Ordo(\eta)$
  from $\hat g$, where $\eta$ is the separation between special
  points, and of bounded $C^{2}$-distance from it provided the
  neighborhood is taken sufficiently small. Letting $J$ be the almost
  complex structure induced by $g$ we then find that we can take $\hat
  J$ and $J$ to be $C^{0}$-close as well; see Remark
  \ref{rem:acsfrommetr}.

  It remains to show that the rigid generalized disks agree with those of $(\hat f,\hat g,\hat J)$ and that there are no generalized disks of formal dimension $<0$ of $(f_\lambda,g,J)$. By
  construction, rigid generalized disks of $(\hat f,\hat g,\hat J)$
  are rigid generalized disks of $(f_\lambda,g,J)$ as well. As in the
  proof of Lemma \ref{lem:seminormal}, a $C^{0}$-small deformation of
  the almost complex structure leads to a $C^{0}$-small deformation of
  the evaluation map $\ev(\overline{\ms^{\ast}})$. In particular, for
  $\hat J$ and $J$ sufficiently close in $C^{0}$, any holomorphic disk
  part of a generalized rigid disk of $(f_\lambda,g,J)$ lies in a
  small neighborhood of some holomorphic disk part of a rigid
  generalized disk of $(\hat f,\hat g,\hat J)$. However, in such a
  neighborhood $J=\hat J$ and it follows that the disks, and
  consequently the rigid generalized disks agree. 

The same argument shows that there would exist generalized disks $(\hat f,\hat g,\hat J)$ of dimension $<0$ near any such generalized disk of $(f_\lambda,g,J)$. Since $(\hat f,\hat g,\hat J)$ is semi-normalized it is in particular generic with respect to rigid generalized disks and we conclude there are no generalized disks of dimension $<0$. 
\end{proof}

\begin{rem} In the language of \cite{ekholm:morse-flow}, the midpoints
  of the intervals where the functions $f_\lambda$ in $({\bf n4})$ and
  $({\bf n5})$ are not necessarily given by a polynomial of degree at
  most two, are called edge points and their neighborhoods edge point
  regions. Although not mentioned explicitly in
  \cite{ekholm:morse-flow} also the $\Ordo(\lambda)$-neighborhoods
  there should be taken to have diameter $K\lambda$ where $K$ is some
  large constant so that the Lagrangian is sufficiently close to the
  $0$-section on the $\lambda$-scale. See Remark \ref{rem:lambdascale}
  for details.
\end{rem}

\begin{rem}
  The introduction of special points and coordinates along flow lines
  as described above essentially correspond to a piecewise linear
  approximation of the graph of $df_\lambda$ with corners smoothed in
  regions of size $K\lambda$.
\end{rem}

\subsection{From holomorphic to generalized disks}\label{s:disktotree}
 
Fix a Morse function $f$, a metric $g$, and an almost complex
structure $J$ such that the triple $(f,g,J)$ is normalized as in Lemma
\ref{lem:normal}. Let $L_\lambda$ denote the Legendrian submanifold
obtained by shifting $L$ a large distance $s$ upwards in the
$z$-direction and then along $f_\lambda$ (which is a function in a
$C^{1}$ $\Ordo(\lambda)$-neighborhood of the $0$-function). With
notation as in Lemma \ref{lem:2pos}, we have
$\Pi_P(L_\lambda)=\Pi_P(L_1(f_\lambda))$). As in
Section~\ref{ssec:2-copy}, Reeb chords of $L\cup L_\lambda$ are
separated into pure chords, mixed chords, and Morse chords and we
study $J$-holomorphic disks with boundary on $L\cup
L_\lambda$. (Recall that a $J$-holomorphic disk with boundary on
$L\cup L_{\lambda}$ is a map $u\colon D_{m+1}\to P$ such that $u(\pa
D_{m+1})\subset \Pi_{P}(L\cup L_{\lambda})$ and such that $u|_{\pa
  D_{m+1}}$ has a continuous lift to $L\cup L_{\lambda}$; see
Subsection \ref{ssec:diffl}.) By Lemma~\ref{lem:one-punc}, such a disk
with one positive puncture has either zero or two mixed Reeb
chords. Disks with all their punctures at pure Reeb chords correspond
naturally to holomorphic disks with boundary on $L$. We therefore
concentrate on disks with mixed punctures. In this subsection, we show
that holomorphic disks with boundary on $L\cup L_\lambda$ and with at
least one Morse chord converge to generalized disks as $\lambda\to
0$. (Once the analysis of such disks is complete, disks without Morse
chords are controlled by Lemmas \ref{lem:tvmspc} and \ref{lem:2pos};
see the end of Section~\ref{ssec:gen-to-hol}.)
  
In outline, the proof of generalized disk convergence runs as follows:
For disks with two Morse chords, convergence follows from
\cite[Theorem 1.2]{ekholm:morse-flow} in combination with a
monotonicity argument. For disks with one Morse chord, we represent
the domains of the holomorphic maps $u_\lambda$ with boundary on
$L\cup L_\lambda$ as strips $\rr\times[0,m]$, $m\in\zz$, $m\ge 1$ with
slits around rays $[a_j,\infty)\times\{j\}$, where $j\in\zz$,
$0<j<m$. After adding a uniformly finite number of punctures to the
domains, we get a uniform derivative bound $|du_\lambda|=\Ordo(1)$. It
is then a consequence of Gromov compactness that this sequence
converges to a broken disk $u_0$ with boundary on $L$, uniformly on
compact subsets. In the present situation, that does not give the full
picture because areas of holomorphic disks with boundary on $L\cup
L_\lambda$ are not uniformly bounded from below as $\lambda\to 0$. For
example, the holomorphic disks corresponding to a Morse flow line of
$f_\lambda$ constructed in \cite[Theorem 1.3]{ekholm:morse-flow} have
areas of size $\Ordo(\lambda)$ and on every compact neighborhood of a
point which maps to a point in the flow line to which the disks
converge, the holomorphic maps converge to a constant map. In order to
capture the full picture, we must understand holomorphic disks with
areas of size $\Ordo(\lambda)$ as well. To this end, we establish the
existence of vertical line segments in the domain $\Delta$ of the
$u_\lambda$ which subdivide $\Delta$ into two pieces
$\Delta=\Delta_1\cup\Delta_2$ such that on $\Delta_2$, the stronger
derivative bound $|du_\lambda|=\Ordo(\lambda)$ holds. With this
derivative bound established, the arguments from \cite[Section
5]{ekholm:morse-flow} show that $u_\lambda|\Delta_2$ converges to a
Morse flow line of $f_\lambda$. Finally, choosing ``maximal''
$\Delta_2$, we show that the limiting flow line of
$u_\lambda|\Delta_2$ starts at a point on the boundary of a component
of the (possibly broken) disk with boundary on $L$ which is the limit
of $u_\lambda|\Delta_1$.

\begin{lem} \label{lem:2-punc}
  Let $u_\lambda\colon D\to P$ be holomorphic disks with boundary on
  $L\cup L_\lambda$ with the positive puncture and one negative
  puncture at Morse chords. Then $u_\lambda$ has no other punctures
  and $u_\lambda$ converges to a (possibly broken) Morse flow line of
  $f_\lambda$ as $\lambda\to 0$. Furthermore, if the disks $u_\lambda$
  are rigid, then so is the limiting flow line.
\end{lem}

\begin{proof} By Lemma~\ref{lem:stokes}, the area of the disk is of
  size $\Ordo(\lambda)$. By monotonicity, the disk cannot leave an
  $\Ordo(\lambda^{\frac12})$-neighborhood of $\Pi_P(L)$. Since the
  complex structure is standard in a such neighborhood of $\Pi_P(L)$,
  it follows from \cite[Lemmas 5.13]{ekholm:morse-flow} that the
  holomorphic disk converges to a flow line. This flow line must
  furthermore be rigid by comparison of dimension formulas for
  holomorphic disks and flow lines, see \cite[Proposition
  3.18]{ekholm:morse-flow}.
\end{proof}

\begin{rem}\label{rem:lambdascale} In the proof of \cite[Lemma
  5.13]{ekholm:morse-flow}, the edge point regions (see \cite[Section
  4.3.8]{ekholm:morse-flow}, in the notation of this paper these
  regions are the regions around plateau points and inflection points
  described in $({\bf n4})$ and $({\bf n5})$ of Subsection
  \ref{s:morseflow}, respectively) were not explicitly mentioned. For
  completeness, we give the explicit argument here. An edge point
  region is a region of size $\Ordo(\lambda)$ around a point on a flow
  line in a rigid generalized disk where the Lagrangian interpolates
  between its nearby affine pieces. In fact, these edge point regions
  should be chosen to have size $K\lambda$, where $K$ is a
  sufficiently large constant. With such a choice, the derivative of
  the interpolation function can be taken as small as
  $\Ordo(\frac{1}{K})$ after rescaling of base and fiber by
  $\lambda^{-1}$ as compared to all other gradient differences
  nearby. In the proof of the convergence result \cite[Lemma
  5.12]{ekholm:morse-flow}, one uses a split coordinate system near
  edge points. In the directions perpendicular to the flow line the
  argument is the one given in the lemma. In the directions along the
  flow line the re-scaled correction function $f$ in \cite[Equation
  (5-8)]{ekholm:morse-flow} is now $\Ordo(\frac{1}{K})$ rather than
  $\Ordo(\lambda)$. However, changes in this direction just correspond
  to reparametrization of the holomorphic disk (i.e. variations along
  the flow direction). The fact that the error term is small after
  rescaling shows that the time spent by a holomorphic map $u_\lambda$
  in the $K\lambda$-neighborhood (the length of the part of the domain
  mapping to the $K\lambda$-neighborhood) is $\Ordo(1)$. Consequently,
  the disk lies at most $\Ordo(\lambda)$ from a flow line after having
  passed through the edge point region.
\end{rem}

We next work with disks that have one mixed Morse puncture and another
mixed puncture. It will be convenient to think of the domains and
spaces of conformal structures of our holomorphic disks as ``standard
domains'' in the language of \cite{ekholm:morse-flow}. For details we
refer to \cite[Subsection 2.2.1]{ekholm:morse-flow}; here we give a
brief description. Consider $\rr^{m-2}$ with coordinates
$a=(a_1,\dots, a_{m-2})$. Let $t\in\rr$ act on $\rr^{m-2}$ by $t\cdot
a=(a_1+t,\dots,a_{m-2}+t)$. The orbit space of this action is
$\rr^{m-3}$. Define a {\em standard domain} $\Delta_{m}(a)$ as the
subset of $\rr\times[0,m]$ obtained by removing $m-2$ horizontal slits
of width $\epsilon$, $0<\epsilon\ll1$, starting at $(a_j,j)$,
$j=1,\dots,m-2$ and going to $\infty$. All slits have the same shape,
ending in a half-circle, see Figure \ref{fig:stdom}.

\begin{figure}[ht]
  \relabelbox \small {\epsfxsize=3in\centerline{\epsfbox{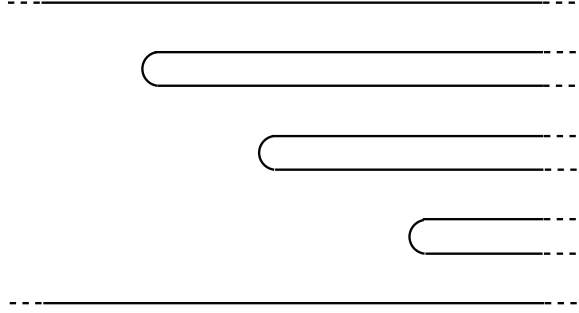}}}
  \endrelabelbox
  \caption{A standard domain}
  \label{fig:stdom}
\end{figure}

Endowing $\Delta_{m}(a)$ with the flat metric, we get a conformal
structure $\kappa(a)$ on the $m$-punctured disk. Moreover, using the
fact that translations are biholomorphic, we find that
$\kappa(a)=\kappa(t\cdot a)$ for all $a$. As shown in \cite[Lemma
2.2]{ekholm:morse-flow}, $\kappa$ is a diffeomorphism from $\rr^{m-2}$
to the space of conformal structures on the $m$-punctured disk.  Below
we will often drop $a$ from the notation and write $\Delta_m$ for a
standard domain.

If $I$ is a boundary component of $\Delta_m[a]$ which has both of its
ends at $\infty$, we will call the point with smallest real part along
$I$ a {\em boundary minimum} of $\Delta_m$. A {\em vertical line
  segment} of a standard domain $\Delta_m$ is a line segment of the
form $\{\tau\}\times[a,b]$ contained in $\Delta_m$ and with $(\tau,a)$
and $(\tau,b)$ in $\pa\Delta_m$.

The first step toward establishing generalized disk convergence is to
show that for any sequence of $J$-holomorphic disks $u_\lambda\colon
\Delta_m^{\lambda}\to P$ with boundary on $L\cup L_\lambda$, where
$\Delta_m^{\lambda}$ are standard domains, there are neighborhoods of
each Morse puncture where the disks converge to a (possibly constant)
flow line.

The key to establishing this convergence is a certain derivative
bound. More precisely, we have the following. As in Subsection \ref{ssec:projs} let $\theta$ denote the $1$-form on $P$ which is the primitive of its symplectic form. Fix a constant $M>0$ and
let $l_\lambda\approx [0,1]$ be a vertical segment in
$\Delta_m^\lambda$ with the following properties:
\begin{itemize}
\item[({\bf l1})] $u_\lambda(l_\lambda)$ is contained in an
  $\Ordo(\lambda)$-neighborhood of $L$.
\item[({\bf l2})] $\left|\int_{l_\lambda}u_\lambda^{\ast}(\theta)\right|\le
  M\lambda$.
\item[({\bf l3})] $\left||z(1)-z(0)|-s\right|\le M\lambda$, where we think
  of $0$ and $1$ as the endpoints of $l_\lambda$ and $z$ as the $\rr$
  coordinate in $P \times \rr$ (evaluated on the boundary lift of $u$).
\end{itemize}

Note that $l_\lambda$ subdivides $\Delta_m^\lambda$ into two
components. Let $\Delta^\lambda(l_\lambda)$ denote the component which
contains the Morse puncture. Also, for $d>0$, let
$\Delta^\lambda(l_\lambda,d)$ denote the subset of points in
$\Delta^\lambda(l_\lambda)$ which are at distance at least $d$ from
$l_\lambda$.

\begin{lem}\label{lem:outest} For all sufficiently small $\lambda>0$,
  the following derivative bound holds:
  \[ |du_\lambda(z)|=\Ordo(\lambda),\quad
  z\in\Delta^\lambda(l_\lambda,1).
  \]
\end{lem}

\begin{proof} The idea of the proof is to localize the situation and
then use technology from the analysis of flow lines (trees) in
\cite[Section 5]{ekholm:morse-flow}. 

Let $b$ denote the Morse chord.  We first show that the area of
$u_\lambda(\Delta^\lambda(l_\lambda))$ is of order
$\Ordo(\lambda)$. If the puncture of $u_\lambda$ at $b$ is negative,
then the area of $u_{\lambda}(\Delta^\lambda(l_\lambda))$ is smaller
than
  \[ \int_{l_\lambda} u_\lambda^{\ast}(\theta) + |z(1)-z(0)| -\ell(b) -
\sum\ell(c),
  \] where the sum ranges over all negative punctures in
$\Delta^\lambda(l_\lambda)$. Since the second and third terms are both
of size $s+\Ordo(\lambda)$, since the first term is of size
$\Ordo(\lambda)$, and since $\ell(c)>\ell_{\rm min}>0$ for all Reeb
chords $c$, we find that the sum must be empty, and consequently that
the area of $u_{\lambda}(\Delta^\lambda(l_\lambda))$ is of size
$\Ordo(\lambda)$.

If the puncture of $u_\lambda$ at $b$ is positive, the area of
$u_{\lambda}(\Delta^\lambda(l_\lambda))$ is smaller than
  \[ \ell(b)-\int_{l_\lambda} u_\lambda^{\ast}(\theta) - |z(1)-z(0)| -
\sum\ell(c),
\] and an analogous argument shows that there are no negative
punctures and that the area of
$u_{\lambda}(\Delta^\lambda(l_\lambda))$ is also of size
$\Ordo(\lambda)$.

We conclude by monotonicity that
$u_\lambda(\Delta^\lambda(l_\lambda))$ must lie in an
$\Ordo(\lambda^{\frac12})$-neighborhood of $L$. Since $J$ agrees with
the almost complex structure induced by the metric on $L$ in such a
neighborhood and since, in local the coordinates given by a symplectic neighborhood map $\Phi\colon T^{\ast}L\to X$, $\theta=p\,dq$, \cite[Lemma 5.4]{ekholm:morse-flow} shows that the
function $|p|^2$, where $p$ is the fiber coordinate in $T^\ast L$
composed with $u_\lambda$, is subharmonic on
$\Delta^\lambda(l_\lambda)$ and therefore attains its maximum on the
boundary. The lemma then follows from \cite[Lemma
5.6]{ekholm:morse-flow}.
\end{proof}

\begin{cor}\label{cor:outconv} The restrictions
  $u_\lambda|_{\Delta^\lambda(l_\lambda,\log(\lambda^{-1}))}$ converge
  to a flow line of $f_\lambda$. (Here we also allow constant flow
  lines.)
\end{cor}

\begin{proof} Note that any region of diameter $\log(\lambda^{-1})$
  maps inside a disk of radius
  $\Ordo(\lambda\log(\lambda^{-1}))$. Moreover, along any strip region
  in $\Delta^\lambda(l_\lambda,\log(\lambda^{-1}))$, the maps converge
  to a flow line at rate $\Ordo(\lambda)$ by the proof of
  \cite[Theorem 1.2]{ekholm:morse-flow}.
\end{proof}

\begin{rem} If the limiting flow line in Corollary \ref{cor:outconv}
  is constant, then it lies at $\Pi_P(c)$ for some Reeb chord $c$ of
  $L$. To see this, note that
  $\Delta^\lambda(l_\lambda,\log(\lambda^{-1}))$ always contains a
  half infinite strip and that if the starting point of this strip
  does not converge to the projection of a Reeb chord, then the flow
  line is non-constant.
\end{rem}

\subsubsection{Blow up analysis}
\label{ssec:blowup}

We next show that for any sequence of $J$-holomorphic disks
$u_\lambda$ with boundary on $L\cup L_\lambda$, we can choose
conformal representatives $\Delta^{\lambda}_m$ of their domains such
that the derivatives $|du_\lambda|$ are uniformly bounded. When a
bubble forms in a sequence of maps on such domains, some coordinate of
the domains $\Delta^\lambda_m$ in the space of conformal structures
goes to $\infty$ (rather than that the derivative of $u_\lambda$
blowing up).

\begin{lem}\label{lem:blowup} Let $u_\lambda\colon\Delta_m^\lambda\to
  P$ be a sequence of $J$-holomorphic disks with boundary on $L\cup
  L_\lambda$. After addition of a finite number of punctures in
  $\Delta_m^\lambda$, creating new domains $\Delta_{m+k}^{\lambda}$,
  the induced maps $u_\lambda\colon\Delta_{m+k}^\lambda\to P$ satisfy
  a uniform derivative bound.
\end{lem}

\begin{proof} The proof is a standard blow up argument, so we only
  sketch the details. Assume that
  $M_\lambda=\sup_{\Delta_m^{\lambda}}|du_\lambda|$ is not
  bounded. The asymptotics near the punctures of $u_\lambda$ show that
  there exist points $p_\lambda\in\Delta^\lambda_m$ at which
  $|du_\lambda|=M_\lambda$. Consider the sequence of maps
  $g_\lambda=u_\lambda\left(p_\lambda+\frac{z}{M_\lambda}\right)$
  defined on $\bigr\{z\in\cc: (p_\lambda+\frac{z}{M_\lambda})\in
  \Delta_m^\lambda\bigr\}$. Note that the derivatives of these maps
  are uniformly bounded. Therefore we can extract a convergent
  subsequence. This gives a non-constant holomorphic disk with
  boundary on $L$ which has one positive puncture and no other
  puncture. Denote this limit disk $v^{[1]}\colon D\to P$ and fix a
  local hypersurface $H$ transversely intersecting $v^{[1]}(\partial
  D)$ at a point far from all Reeb chords. It follows from the
  convergence $g_\lambda\to v^{[1]}$ that there exists a point in a
  neighborhood of $p_\lambda$ which $u_\lambda$ maps to
  $H$. Puncturing $\Delta_m^\lambda$ at this point induces a new
  sequence of maps $u_\lambda^{[1]}\colon\Delta_{m+1}^\lambda\to
  P$. If $|du^{[1]}_\lambda|$ is uniformly bounded then the lemma
  follows.

  Assume that $\sup_{\Delta_{m+1}^\lambda}|du_\lambda^{[1]}|$ is
  unbounded. Arguing as above, we find another bubble $v^{[2]}$ in the
  limit with one positive puncture and no other punctures. Adding
  another puncture in the domain which corresponds to some point in
  $v^{[2]}$ in the limit, we get new maps and domains
  $u^{[2]}_\lambda\colon \Delta^{\lambda}_{m+2}\to P$. Repeating this,
  we either have no derivative blow up in which case the lemma follows
  or we add punctures as above. To see that this is a finite process
  note that each bubble has area bounded from below by the length of
  the shortest Reeb chord and that the sum of the areas of all bubbles
  must be smaller than the length of the longest Reeb chord of $L$.
\end{proof}

\begin{rem}\label{r:extrapuncture} Let $p$ be the image point
  corresponding to an additional puncture added in the procedure
  described above. Note that there exists a disk of finite radius
  around $p$ which does not contain any Reeb chords. It follows from
  monotonicity that for $\lambda>0$ small enough, the area
  contribution of any subdisk obtained by cutting off a vertical
  segment which connects $L$ to $L_\lambda$ and which contains the
  puncture corresponding to the marked point is uniformly bounded from
  below. It follows that $\Delta^{\lambda}(l_\lambda)$ in
  Lemma \ref{lem:outest} cannot contain any additional punctures.
\end{rem}

\subsubsection{Generalized disk convergence} Consider a sequence
$u_\lambda$ of $J$-holomorphic disks with boundary on $L\cup
L_\lambda$ with one puncture at a Morse chord. Using Lemma
\ref{lem:blowup}, we assume that these are maps
$u_\lambda\colon\Delta^{\lambda}_m\to P$ with uniformly bounded
derivatives. It is a consequence of Gromov compactness that this
sequence converges to a broken disk $v$ on $L$, uniformly on compact
subsets. Let $\partial v$ denote the image in $L$ of the boundary of
the possibly broken non-constant limit disk $v$ or if there are only
constant limit disks then let $\pa v$ denote the double point
corresponding to the mixed puncture of the disk which is not a Morse
chord.  Lemma \ref{cor:outconv} and Remark \ref{r:extrapuncture} imply
that if a vertical line segment $l_\lambda\subset \Delta^{\lambda}_m$
satisfies $({\bf l1})-({\bf l3})$ then on
$\Delta^{\lambda}_m(l_\lambda)$ the disks converge to a flow line.

\begin{lem}\label{lem:nogap} There exist vertical segments $l_\lambda$
  which satisfies $({\bf l1})-({\bf l3})$ and such that
  $u_\lambda(l_\lambda)$ converges to a point in $\partial v$.
\end{lem}

\begin{proof} We prove this lemma by contradiction: if the statement
  of the lemma does not hold, then the area difference between a limit
  disk and the disks before the limit violates an $\Ordo(\lambda)$
  bound derived from Stokes' theorem.

  Thus, we assume that the lemma does not hold. Then there exists
  $\epsilon>0$ such that for any sequence of $l_\lambda$ which
  satisfies $({\bf l1})-({\bf l3})$, some point on $l_\lambda$ maps a distance at
  least $\epsilon>0$ from $\partial v$. Consider a strip region
  $[-d,d]\times[0,1]\subset \Delta^{\lambda}_m$ for which some point
  converges to a point a distance $\delta$ from $\partial v$, where
  $\frac{\epsilon}{4}<\delta<\frac{\epsilon}{2}$. Let
  $\sup_{[-d,d]\times[0,1]}|du_\lambda|=K$. Then $K$ is not bounded by
  $M\lambda$ for any $M>0$. Since the difference between the area of
  the limit disk $v$ and that of $u_\lambda$ is of order of magnitude
  $\Ordo(\lambda)$ it follows that
  $|du_\lambda|=\Ordo(\lambda^{\frac12})$ from the usual bootstrap
  estimate. Thus $K=\Ordo(\lambda^{\frac12})$. Consider next the
  scaling of the target by $K^{-1}$. We get a sequence of maps $\hat
  u_\lambda$ from $[-d,d]\times[0,1]$ with bounded derivative. Note
  moreover that the boundary condition is $\Ordo(\lambda^{\frac12})$
  from the $0$-section. Changing coordinates to the standard
  $(\cc^{n},\rr^{n})$ respecting the complex structure at the limit
  point, we find that there are maps
  $f_\lambda\colon[-d,d]\times[0,1]\to\cc^{n}$ with the following
  properties
  \begin{itemize}
  \item
    $\sup_{[-d,d]\times[0,1]}|D^{k}f_\lambda|=\Ordo(\lambda^{\frac12})$,
    $k=0,1$.
  \item $u_\lambda+f_\lambda$ satisfies $\rr^{n}$ boundary conditions
  \item $\bar\partial (u_\lambda+f_\lambda)=\Ordo(\lambda^{\frac12})$.
  \end{itemize} It follows that $u_\lambda + f_\lambda$ converges to a
  holomorphic map with boundary on $\rr^{n}$, which takes $0$ to $0$
  and which has derivative of magnitude $1$ at $0$. Using solubility
  of the $\bar\partial$-equation in combination with $L^{2}$-estimates
  in terms of area we find that the area of $\hat u_\lambda$ must be
  uniformly bounded from below by a constant $C$. The area
  contribution to the original disks near the limit is thus at least
  $K^{2} C$.  Since $[-d,d] \times [0,1]$ covers a length along $L$ of
  at most $2Kd$, we may repeat the argument with many disjoint finite
  strips which together cover a finite length and with maximal
  derivatives $K_j$. We find that the area contribution is bounded
  from below by $C \sum K_j^{2}$ and that, since the length
  contribution is finite, we get:
  \[ 2d \sum K_j\ge \frac{\epsilon}{100}.
  \] Now,
  \[ C \sum K_j^{2}\ge C\inf_j\{K_j\}\sum K_j\ge C'\inf_j\{K_j\}.
  \] For any $M>0$, $\inf_j\{K_j\}\ge M\lambda$. To see this assume
  that it does not hold true. Then there is a sequence of vertical
  segments $l_\lambda$ such that $|du_\lambda|\le 2M\lambda$ with the
  property that the distance between $u(l_\lambda)$ and $\partial v$
  is at most $\frac34\epsilon$. This however contradicts our
  hypothesis. Consequently, the area contribution from the remaining
  part of the disk is not $\Ordo(\lambda)$, which contradicts Stokes'
  theorem.
\end{proof}

As a consequence we get the following:

\begin{cor}\label{cor:conv1} 
  Any sequence of rigid holomorphic disks
  $u_{\lambda}\colon\Delta^{\lambda}_m\to P$ with boundary on $L\cup
  L_\lambda$ and with one Morse puncture has a subsequence which
  converges to a rigid generalized disk.
\end{cor}

\begin{proof} It follows from Lemmas \ref{cor:outconv} and
  \ref{lem:nogap} and from Gromov compactness that the limit gives a
  generalized disk. This generalized disk must furthermore have formal
  dimension $0$. Since the triple $(f_\lambda,g,J)$ is normalized, it is in particular generic with respect to rigid generalized disks, see Lemma \ref{lem:normal}, and it follows that the
  $J$-holomorphic disk part is not broken and that the generalized disk is
  transversely cut out.
\end{proof}

\subsection{From generalized to holomorphic disks}
\label{ssec:gen-to-hol}

Let $f_\lambda$, $0<\lambda\le 1$ be a $1$-parameter family of Morse
functions on $L$, $g$ be a metric on $L$, and $J$ an almost complex
structure on $P$ such that the triple $(f_\lambda,g,J)$ is normalized.
The goal of this section is to produce a unique rigid holomorphic disk
with boundary on $L\cup L_\lambda$ near each generalized holomorphic
disk determined by $(f_\lambda,g,J)$ for all sufficiently small
$\lambda>0$. We begin by associating to each generalized holomorphic
disk a family of domains and approximately holomorphic maps with
boundary on $L\cup L_\lambda$. We then define a functional analytic
space of variations of the approximately holomorphic map, show that
the linearized $\bar\partial_J$-operator is uniformly invertible on
this space, and derive a second derivative estimate. By Lemma
\ref{lem:FloerPicard}, the invertibility and second derivative
estimate gives a unique holomorphic disk in the functional analytic
neighborhood of the approximate solution, and we show that any
solution must lie in this neighborhood.

\subsubsection{Approximate solutions}

Consider a rigid generalized disk $(u,\gamma)$, where $u\colon
\Delta_{m-1}\to P$ is the holomorphic disk part with $m-1$ punctures
mapping to Reeb chords and $\gamma$ is the Morse flow line part. We
consider the following three cases separately:
\begin{itemize}
\item[$({\bf gd1})$] The map $u$ is constant,
\item[$({\bf gd2})$] The map $u$ is non-constant and $\gamma$ is
  constant, and
\item[$({\bf gd3})$] The map $u$ is non-constant and $\gamma$ is
  non-constant.
\end{itemize}
For $(u,\gamma)$ of type $({\bf gd1})$, the rigid generalized disk is
a flow line and existence of a unique holomorphic disk with boundary
on $L\cup L_\lambda$ near the rigid generalized disk follows from
\cite[Theorem 1.3]{ekholm:morse-flow}.

In cases $({\bf gd2})$ and $({\bf gd3})$, add a puncture to $u$ at the
junction point to obtain a map $u\colon\Delta_m\to P$. Fix a reference
point $0$ in $\Delta_m$ which $u$ maps to a point far from any
puncture and note that the derivative of $u$ is bounded. It follows
from the standard form of the Lagrangian projection and the complex
structure near junction points in $({\bf n2})$ that, in a half strip
neighborhood of the junction point $[0,\infty)\times\rr$, the map $u$
looks like:
\begin{equation}\label{e:locjunct}
  u(z)=\sum_{n<0} c_n e^{n\pi(z)},\quad z\in [0,\infty)
  \times\rr,\quad c_n\in\rr^{n},
\end{equation}
where $\rr^{n}$ corresponds to $L$. Similarly, near a Reeb chord
puncture as in $({\bf a2})$, the map $u=(u_1,\dots,u_n)$ looks like:
\begin{equation}\label{e:solcritp} 
  u_j(z)=\sum_{n\le 0} c_{j,n} e^{(-\theta_j+n\pi)z},\quad
  z\in[0,\infty) \times[0,1],\quad c_{j,n}\in\rr,
\end{equation}
where $0<\theta_j<\pi$. Consequently, there are vertical segments
$l_\lambda$ in a half strip neighborhood of any puncture at distance
$\Ordo(\log(\lambda^{-1}))$ from $0\in\Delta_m$ such that a finite
neighborhood of $l_\lambda$ maps into an $\Ordo(\lambda)$ neighborhood
of the corresponding special point.

In case $({\bf gd2})$, we construct approximately holomorphic maps
$w_\lambda\colon \Delta_m\to P$ with boundary on $L\cup L_\lambda$ as
follows. At each Reeb chord puncture of $u$ and at the junction point
consider vertical line segments $l_\lambda$ in $\Delta_m$. The
vertical line segment $l_\lambda$ subdivides $\Delta_m$ into an inner
component $\Delta_m^{0}(\lambda)$ containing $0$ and outer half strip
regions near punctures. We change the boundary condition of $u$ on
$\Delta_m^{0}(\lambda)$ according to the boundary lift of the rigid
disk. To this end we must move the boundary of
$u|_{\Delta_m^{0}(\lambda)}$ a distance $\Ordo(\lambda)$. Supporting
such a deformation in an $\Ordo(\lambda^{\frac12})$-neighborhood of
the boundary, one can achieve this while changing the derivative of
$u$ by at most $\Ordo(\lambda^{\frac12})$. Finally, we interpolate to
constant maps to double points in finite region in $\Delta_m$ near the
vertical segments $l_\lambda$. Then the derivative of the resulting
map in these finite regions is of size $\Ordo(\lambda)$. This
completes the definition of the approximately holomorphic maps
$w_\lambda\colon\Delta_m\to P$ in case $({\bf gd2})$.

In case $({\bf gd3})$, the construction of approximately holomorphic
maps $w_\lambda\colon\Delta_m\to P$ is a bit more involved. Here we
subdivide the domain into three pieces. First pick vertical segments
$l_\lambda$ in $\Delta_m$ near each Reeb chord puncture exactly as in
case $({\bf gd2})$. Near the junction point puncture we instead pick a
vertical segment $l$ which is mapped into a small but finite
$\epsilon$-neighborhood of the junction point. Such a segment lies at
distance $\Ordo(1)$ from $0\in\Delta_m$. Cut $\Delta_m$ off at $l$ and
$l_\lambda$ and denote the component containing $0$ by
$\Delta_m^{0}(\lambda)$. The two other pieces are half strips
$[0,d\lambda^{-1}]\times[0,1]$, where $d>0$ is a suitable constant and
a half infinite strip $[0,\infty)\times[0,1]$. To construct
$\Delta_m$, we glue $[0,d\lambda^{-1}]\times [0,1]$ to
$\Delta_m^{0}(\lambda)$ along $l$ and glue in $[0,\infty)\times[0,1]$
to complete the resulting domain.

On the piece $\Delta_{m}^{0}(\lambda)$, we define the almost
holomorphic map $w_\lambda^{0}\colon\Delta_m^{0}\to P$ by moving the
boundary condition of $u$ to $L\cup L_\lambda$ according to the
boundary lift of the rigid generalized disk and then interpolate to
constant maps at all Reeb chord punctures, exactly as in case $({\bf
  gd2})$.

Near the Morse flow line $\gamma$, cut off at a small distance from
the junction point. As in \cite[Section 6.1]{ekholm:morse-flow}, we
construct almost holomorphic maps
$w^{\gamma}_\lambda\colon[0,\infty)\times[0,1]\to P$ which agree with
the natural holomorphic strip over the gradient flow lines in the flat
metric, as in \cite[Sections 6.1.1 and 6.1.2]{ekholm:morse-flow}
outside the regions near plateau- and inflection points of $({\bf
  n4})$ and $({\bf n5})$, where it interpolates between these
solutions and satisfies $|\bar\partial_J
w^{\gamma}_\lambda|=\Ordo(\lambda)$.

Finally, consider the region near the junction point. We define a
holomorphic map $w^{\rm jun}_\lambda\colon
[0,d\lambda^{-1}]\times[0,1]\to P$ which satisfies the boundary
conditions. To define this map, we assume that the
$\cc^{n}$-coordinates $x+iy$ have been chosen so that $L$ corresponds
to $\rr^{n}=\{y_1=\dots=y_n=0\}$ and $L_\lambda$ corresponds to
$\{y_1=\lambda,y_2=\dots=y_n=0\}$. The holomorphic map is then
\begin{equation}\label{e:junctsol} w^{\rm jun}_\lambda(z)=(\lambda
z,0,\dots,0) + \sum_{n<0} c_n e^{n\pi z},\quad c_n\in\rr^{n}.
\end{equation} 

Near the junction point, the maps $w_\lambda^{0}$ and $w^{\rm
  jun}_\lambda$ (as well as $w_\lambda^{\gamma}$ and $w^{\rm
  jun}_\lambda$) are then of distance $\Ordo(\lambda)$ apart (see
\eqref{e:locjunct} and \eqref{e:junctsol}), and we interpolate between
them using a function of size $\Ordo(\lambda)$ on a finite
rectangle. The function resulting from this interpolation is
$w_\lambda\colon \Delta_m\to P$ in case $({\bf gd3})$.

As mentioned above, we will produce $J$-holomorphic disks near
$w_\lambda\colon\Delta_m\to P$ parametrizing a neighborhood of this
map by a weighted Sobolev space of vector fields. We next describe the
weight function in cases $({\bf gd2})$ and $({\bf gd3})$. In the
former case, this is straightforward: take $h\colon\Delta_m\to\rr$ to
be a function which equals $1$ on $\Delta_m(\lambda_0)$ for some fixed
$\lambda_0$. On the remaining strip regions of the form
$[0,\infty)\times[0,1]$ we take $h(\tau+it)=e^{\delta|\tau|}$ for
$\delta>0$, where $\delta\ll\theta_j$ for all $\theta_j$ as in
\eqref{e:solcritp} at any Reeb chord.  In the latter case, the
function $h\colon\Delta_m\to\rr$ equals $1$ on the piece
$\Delta_m^{0}(\lambda_0)$ for some fixed $\lambda_0$ and equals to
$e^{\delta|\tau|}$ for $\tau+it\in[0,\infty)\times[0,1]$ in the
neighborhood $[0,\infty)\times[0,1]$ of each Reeb chord puncture for
small $\delta>0$. In
$[0,\frac{d}{\lambda}]\times[0,1]\subset\Delta_m$, we let:
\begin{equation}\label{e:weightjunct}
h(\tau+it)=e^{\delta\left(\tfrac12d\lambda^{-1}-\left|\tau-\tfrac12d\lambda^{-1}\right|\right)}.
\end{equation}
In $[0,\infty)\times[0,1]$, the function $h$ has the same shape as in
\eqref{e:weightjunct} along edges where $w^{\gamma}_\lambda$ agrees
with a explicit solution and equals $1$ in a neighborhood of the
interpolation regions at plateau- and inflection points, as in
\cite[Section 6.3.1]{ekholm:morse-flow}.

Let $\|\cdot\|_{k,\delta}$ denote the Sobolev norm in the Sobolev space of functions with $k$ derivatives in $L^{2}$ with weight $h$.

\begin{lem} The approximately holomorphic function $w_\lambda$
satisfies
  \[ \|\bar\partial_J
w_\lambda\|_{1,\delta}=\Ordo(\lambda^{\frac34-\delta}(\log\lambda^{-1})^{\frac12})
  \]
\end{lem}

\begin{proof} The regions where the map is non-holomorphic are of two kinds: finite rectangles where the
size of the derivatives are $\Ordo(\lambda)$ and an
$\Ordo(\lambda^{\frac12})$-neighborhood of the boundary in
$\Delta_m(d)$ cut off at all $l_\lambda^{\rm Reeb}$. The former
regions give a contribution of size $\Ordo(\lambda)$ the latter gives
the contribution
  \[
\sqrt{\Ordo(\lambda)\cdot\Ordo(\lambda^{-2\delta})\cdot\Ordo(\lambda^{\frac12})\cdot\Ordo(\log\lambda^{-1})},
  \] where the first factor is the square of the size of the
deformation, the second factor is the maximum value of the weight function $h$ at a point where $w_\lambda$ is not holomorphic, and the product of the last two estimates the area of the region in which $w_\lambda$ is non-holomorphic.
\end{proof}

We associate a variation space $\hat\sblv_{2, \delta}$ to
$w_\lambda$. This is a direct sum
\[ \hat\sblv_{2,\delta}=\sblv_{2,\delta}\oplus V_{\rm con} \oplus
V_{\rm sol},
\] where the summands are the following:
\begin{itemize}
\item In both cases $({\bf gd2})$ and $({\bf gd3})$,
  $\sblv_{2,\delta}$ is a Sobolev space of vector fields $v$ along
  $w_\lambda$ with two derivatives in $L^{2}$ weighted by $h$ which
  satisfy the following additional conditions: $v$ is tangent to $L$
  along the boundary, $\bar\nabla_J v=0$ along the boundary (as in
  \cite[Lemma 3.2]{ees:pxr} and \cite[Section
  6.3.1]{ekholm:morse-flow}). In case $({\bf gd3})$, the vector fields
  $v$ satisfy the following additional conditions: $v$ vanishes at one
  boundary point midway between plateau-points and inflection points,
  midway between inflection points (as in \cite[Section
  6.3.1]{ekholm:morse-flow}) as well as at a boundary point in the
  middle of the strip at the junction point. In other words, $v$
  vanishes at one of the boundary points of every vertical segment
  along which the weight function $h$ has a local maximum.
\item In both cases $({\bf gd2})$ and $({\bf gd3})$, $V_{\rm con}$ is
  the space of conformal variations of $\Delta_m$ (see \cite[Section
  5.6]{ees:high-d-analysis} for a description).
\item In case $({\bf gd2})$, $V_{\rm sol}=0$. In case $({\bf gd3})$,
  $V_{\rm sol}$ is a finite dimensional space consisting of cut off
  constant solutions of the $\bar\pa_J$-equation supported in the
  regions where the weight function $h$ is large, as follows. Along
  the flow line part, there are $n$-dimensional cut-off constant
  solutions between any two plateau/inflection-points exactly as in
  \cite[Section 6.3.2]{ekholm:morse-flow}. There is also an
  $n$-dimensional space corresponding to the junction point. More
  precisely, in the coordinates of \eqref{e:junctsol}, this
  $n$-dimensional space is spanned by
  \[
  \beta(1,0,\dots,0),\,\,\beta(0,1,0,\dots,0),\,\,\dots\,\,,\beta(0,\dots,0,1),
  \] 
  where $\beta$ is a cut off function equal to $1$ in the region
  $[0,d\lambda^{-1}]\times[0,1]$ corresponding to the junction point
  and equal to $0$ outside a uniformly finite neighborhood of it. We
  equip $V_{\rm sol}$ with the supremum norm.
\end{itemize}

Choosing a $1$-parameter family of Riemannian metrics $G^{\sigma}$,
$0\le \sigma\le 1$, on $P$ as in \cite[Section 3.1.2]{ees:pxr}, see
also \cite[Section 5.2]{ees:high-d-analysis}, and an extension
$z\colon\Delta_m\to\rr$ of the $z$-coordinate of the boundary lift of
$w_\lambda$, we define an exponential map:
\[
\exp(v)=\exp^{G^{\sigma(z(\zeta))}}_{w_\lambda(\zeta)}(v(\zeta)),
\] 
where $\sigma\colon\rr\to[0,1]$, and where $\exp^{G^{\sigma}}$ is the
exponential map in the metric $G^{\sigma}$, which gives a local chart
in the configuration space of maps around $w_\lambda$. In particular,
as in Subsection \ref{sssec:cornermap}, we think of the (non-linear)
$\bar\pa_J$-operator on $\hat\sblv_{2,\delta}$ as the operator
\[
\bar\pa_J(v)=d\exp(v)+Jd\exp(v)j.
\]

\subsubsection{Uniform invertibility} Let $L\bar\partial_J$ denote the
linearization of the $\bar\partial_J$-operator acting on elements
$v\in \hat\sblv_{2,\delta}$. This map takes vector fields in
$\hat\sblv_{2,\delta}$ to complex anti-linear maps $T\Delta_m\to
w_\lambda^{\ast} TP$. We pick a trivialization of $T\Delta_m$ and
identify the complex anti-linear map with the image of the
trivializing vector field. In this way, we view the $L\bar\pa_J$ as a
map $\hat\sblv_{2,\delta}\to\sblv_{1,\delta}$ where $\sblv_{1,\delta}$
is the Sobolev space of vector fields along $w_\lambda$ with one
derivative in $L^{2}$ weighted by $h$.

\begin{lem}\label{lem:unifinv} The differential
  \[ L\bar\partial_J\colon\hat\sblv_{2,\delta}\to\sblv_{1,\delta}
  \] is uniformly invertible.
\end{lem}

\begin{proof} The proof in case $({\bf gd3})$ is similar to the proof
  of \cite[Proposition 6.20]{ekholm:morse-flow}, so we just give an
  outline. Recall that the domain of $w_\lambda$ was built out of
  three pieces $\Delta_m=\Delta_m^{0}\cup[0,d\lambda^{-1}]\times[0,1]
  \cup[0,\infty)\times[0,1]$. Consider a variation $v_\lambda$ of
  $w_\lambda$. Write $v_\lambda=v^{\rm 0}_\lambda+v^{\rm
    jun}_\lambda+v^{\rm mo}_\lambda$ where we use cut off functions to
  subdivide $v$ into a disk-piece $v^{0}$ supported in a neighborhood
  of $\Delta_m^{0}$, a junction-piece $v^{\rm jun}$ supported in a
  neighborhood of $[0,d\lambda^{-1}]\times[0,1]$, and a Morse-piece
  $v^{\rm mo}$ supported in a neighborhood of
  $[0,\infty)\times[0,1]$. If the operator is not uniformly
  invertible, then there is a sequence of variation maps $v_\lambda$
  such that
  \begin{equation}\label{e:contr} \|v_\lambda\|_{2,\delta}=1,\quad
    \|L\bar\partial_J v_\lambda\|_{1,\delta}\to 0.
  \end{equation}
  Parametrize a neighborhood of the holomorphic disk part using a
  Sobolev space with small positive exponential weight near the added
  marked point and cut off constant solutions transverse to the
  disk. We infer from the properties on the disk part that
  $v^{0}_\lambda$ has non-zero component along the cut-off and
  conformal solutions at the junction point. Similarly, $v^{\rm
    mo}_\lambda$ has non-zero component in the tangent directions of
  the (un)stable manifold in which it lies. The components of the
  cut-off solutions correspond to cut-off solutions at the junction
  point in the space of variations of $w_\lambda$. Since the
  evaluation map from the moduli space of holomorphic disks is
  transverse to the (un)stable manifold at the junction point, it
  follows that the component of $v_\lambda^{\rm jun}$ along the cut
  off solutions must go to $0$ with $\lambda$. We conclude from this
  that $v^{0}_\lambda\to 0$ and $v^{\rm mo}_\lambda\to 0$. Consider
  the affine inclusion
  $[0,d\lambda^{-1}]\times[0,1]\subset\rr\times[0,1]$ taking
  $\frac{d}{2\lambda}$ to $0$. Since the $\bar\partial$-operator on
  $\rr\times[0,1]$ with $\rr^{n}$ boundary conditions acting on a
  Sobolev space with weight $e^{-|\tau|}$, $\tau+it\in\rr\times[0,1]$
  is invertible on the complement of cut off solutions, we conclude
  that $v^{\rm jun}_\lambda\to 0$ as well. This contradicts
  \eqref{e:contr}.

  The proof in case $({\bf gd2})$ is similar but simpler.
\end{proof}

\subsubsection{Existence and uniqueness of solutions near generalized
  disks} 

As mentioned above, after trivializing the bundle of complex
anti-linear maps over a neighborhood of $w_\lambda$, see Subsection
\ref{sssec:cornermap}, we consider the $\bar\partial_J$-operator on
maps in a neighborhood of $w_{\lambda}$ as a map
\[
f\colon\hat\sblv_{2,\delta}\to\sblv_{1,\delta},
\] 
with $w_\delta$ corresponding to $0\in\hat\sblv_{2,\delta}$.

\begin{lem}\label{lem:quadest} There exists a constant $C>0$ such that
  \[ f(v)=f(0)+ df(v) + N(v),
  \] where
  \[ \|N(v_1)-N(v_2)\|_{1,\delta}\le
C(\|v_1\|_{2,\delta}+\|v_2\|_{2,\delta})\|v_1-v_2\|_{2,\delta}.
  \]
\end{lem}

\begin{proof} This is a standard argument in the case there are no
  weights and no cut-off solutions. However, as the weight function is
  $\ge 1$, the left hand side is linear in the weight and the right
  hand side is quadratic, so the weight does not interfere with the
  estimate.  The cut-off solutions are true solutions except in finite
  regions where the weight is bounded. We conclude that the estimate
  holds.
\end{proof}

\begin{cor}\label{cor:exist} There exists a unique disk in a finite
  $\|\cdot\|_{2,\delta}$-neighborhood of $w_\lambda$ for all
  sufficiently small $\lambda>0$.
\end{cor}

\begin{proof} This is a consequence of Lemma \ref{lem:quadest} in
  combination with Lemma \ref{lem:FloerPicard}.
\end{proof}

\begin{lem}\label{lem:unique} For sufficiently small $\lambda>0$, if a
  holomorphic disk lies in a sufficiently small $C^{0}$-
  neighborhood of $w_\lambda$, then it lies inside a small
  $\|\cdot\|_{2,\delta}$-neighborhood of $w_\lambda$.
\end{lem}

\begin{proof} The proof of this lemma is analogous to the last
  argument in the proof of \cite[Theorem 1.3]{ekholm:morse-flow}, so
  we only sketch it. Fix a generalized disk. It is a consequence of
  the convergence result Corollary \ref{cor:conv1} that for
  sufficiently small $\lambda$, any holomorphic disk in a finite
  neighborhood of the generalized disk converges to it as $\lambda\to
  0$. Further, the domains of such a sequence of disks can be
  subdivided into three pieces: a subset of a domain
  $\Delta_m(\lambda)$ which converges to the domain $\Delta_m$ of the
  limit disk on which the map converges to the limit map, a strip part
  of length $\Ordo(\lambda^{-1})$ mapping to a small neighborhood of
  the junction point, and a half infinite strip converging to a Morse
  flow line. We estimate the $\|\cdot\|_{2,\delta}$-distance by
  considering these three pieces separately. Over the big-disk part,
  closeness to $w_\lambda$ in the $\|\cdot\|_{2,\delta}$-norm follows
  from the convergence result just mentioned using the decomposition
  of holomorphic functions into cut-off constant solutions and
  exponentially decaying functions near the marked point
  puncture. Over the gradient line part, the
  $\|\cdot\|_{2,\delta}$-norm can be controlled exactly as in \cite[Proof of Theorem
  1.3]{ekholm:morse-flow} (cf. the calculation on page 1218 and the argument for the ``finte number of strip regions'' on page 1219). Finally, in the region over the junction
  point we note that if $y_\lambda$ is any solution and if $u^{\rm
    jun}_\lambda$ is the local solution that was used to build
  $w_\lambda$, then $y_\lambda-u^{\rm jun}_\lambda$ maps into
  $\cc^{n}$, is holomorphic, and satisfies $\rr^{n}$ boundary
  conditions. As in the proof of \cite[Theorem
  1.3]{ekholm:morse-flow}, we find that the $C^{0}$-distance near the
  endpoints of the strip region controls the
  $\|\cdot\|_{2,\delta}$-norm and the $C^{0}$-distance goes to $0$ by
  Corollary \ref{cor:conv1}.
\end{proof}

\begin{cor}\label{cor:disk+Morse} For all sufficiently small
  $\lambda>0$, there is a unique rigid holomorphic disk with at most
  two Morse chords corresponding to each rigid generalized disk.
\end{cor}

\begin{proof} This follows from Corollary \ref{cor:exist} and Lemma
  \ref{lem:unique}.
\end{proof}

\subsection{Proof of Theorem \ref{thm:disk+Morse}}
\label{s:disk+Morse}

Lemma~\ref{lem:2-punc} and Corollaries \ref{cor:conv1} and
\ref{cor:disk+Morse} imply that there is a family of normalized
triples $(f_\lambda,g,J)$, $0<\lambda<1$ such that any rigid
holomorphic disk with boundary on $L\cup L_\lambda$ which has at least
one Morse puncture is transversely cut out and corresponds either to a
flow line or to a generalized disk. In particular, it follows that
$({\rm i})$ and $({\rm ii})$ hold, that $({\rm iv})$ holds for any
moduli space of disks with at least one Morse puncture, and that $(3)$
and $(4)$ hold.

It thus remains to deal with disks with two
mixed punctures, neither of which are Morse punctures. Consider a
mixed puncture of such a disk. At such a puncture, the disk has an
incoming and an outgoing sheet. Identify these sheets of $L\cup
L_\lambda$ with the corresponding sheets of $L$. This gives asymptotic
data for a holomorphic disk with boundary on $L$ at the Reeb chord
corresponding to the mixed Reeb chord we started with. That asymptotic
data may correspond to a positive or a negative puncture. We call it
the {\em induced asymptotic data} at the mixed puncture. There are the
following two cases to consider.
\begin{itemize}
\item[(${\rm I}$)] The asymptotic data of one of the mixed punctures
  corresponds to a positive puncture and that of the other corresponds
  to a negative puncture.
\item[(${\rm II}$)] The asymptotic data of both mixed punctures correspond to
  positive punctures.
\end{itemize}

To get the correspondence between rigid disks of type $({\rm I})$ and liftings
of rigid disks in $\ms(a; b_1,\dots,b_k)$, we argue as follows. View
the boundary condition for a disk of type $({\rm I})$ as a perturbation of the
boundary condition for a disk in $\ms(a; b_1,\dots,b_k)$. As the
latter moduli space is transversely cut out, it follows by Gromov
compactness that for $\lambda>0$ small enough there is a bijective
correspondence between rigid disks of type $({\rm I})$ and liftings of rigid
disks in $\ms(a;b_1,\dots,b_k)$.

Fix a small $\lambda>0$ so that this bijective correspondence exists
and so that Corollary \ref{cor:disk+Morse} holds. Assume that the
Morse function $f_\lambda$ is sufficiently generic so that Lemma \ref{lem:2pos}
holds for $L_0=L$ and $L_1=L_1(f_\lambda)=L_\lambda$: as the proof of Lemma \ref{lem:2pos} shows, to achieve this genericity we need only perturb $f_\lambda$ an arbitrary small amount near the double points of $L$. For sufficiently small such perturbation $(f_\lambda,g,J)$ remains normalized. Then, by definition, rigid
holomorphic disks of type $({\rm II})$ correspond to liftings of rigid disks
in $\ms(a_1,a_2;b_1,\dots,b_k)$. This shows that $({\rm iii})$, the general version of $({\rm iv})$, as well as $(2)$ hold. As mentined in the beginning of Subsection \ref{s:disktotree}, $(1)$ holds provided $\lambda>0$ is small enough, and we conclude that the theorem holds.

\appendix

\section{Torsion Framings and the Grading}
\label{apdx:framings}

The goal of this section is to show how a choice of sections of $TP$
over the $3$-skeleton of some fixed triangulation of $P$ determines a
loop $\zz_g$-framing used for grading in
Section~\ref{ssec:grading}. We will prove these results in the
slightly more general setting of a rank $k$ complex vector bundle
$\eta \to M$, where $k \geq 2$. See Section~\ref{ssec:grading} for the
definitions of the greatest divisor $g(\eta)$ and a $\zz_g$-framing of
$\eta$ along a closed curve $\gamma \subset M$.

Fix a triangulation $T$ of $M$ and let $T^{(j)}$ denote the
$j$-skeleton of $T$. Fix a Hermitian metric on $\eta$.

\begin{lem}\label{l:triv-1} 
  The complex vector bundle $\eta$ admits $k-1$ everywhere orthonormal
  sections over $T^{(3)}$. Any two restrictions of such $(k-1)$-frames
  over $T^{(2)}$ are homotopic.
\end{lem}

\begin{proof}
  The first statement follows from the fact that the higher Chern
  classes of $\eta|_{T^{(3)}}$ vanish as $H^i(T^{(3)}) = 0$ for $i
  \geq 4$.


  We next consider the uniqueness of such frames.  Let $V(r,r-1)$
  denote the Stiefel manifold of orthonormal $(r-1)$-frames in
  $\cc^{r}$ and consider the natural fibration
  \[
  \begin{CD} V(r-1,r-2) @>{\iota}>> V(r,r-1) @>{\pi}>> S^{2r-1}
  \end{CD},
  \] where $\pi$ is the projection which maps a frame to its first
  vector and where $\iota$ is the inclusion of the fiber. The long
  exact homotopy sequences of these fibrations show that for $j=1,2$
  we have
  \[
  \pi_j(V(r,r-1))\cong\pi_j(V(r-1,r-2))\cong\dots\cong\pi_j(V(2,1))=0,\quad j=1,2.
  \]

  A $(k-1)$-frame in $\eta$ over $T^{(2)}$ gives a section in the
  bundle $V(\eta,k-1)|_{T^{(2)}}$, where $V(\eta,k-1)$ is the bundle
  with fiber $V(k,k-1)$ naturally associated to $\eta$. The
  obstructions to finding a homotopy between two such sections over
  $T^{(j)}$ lie in
  \[
  H^{j}\Bigl(T^{(2)};\pi_j\bigl(V(k,k-1)\bigr)\Bigr)=0,\quad j=1,2
  \]
  and the lemma follows.
\end{proof}

Pick $k-1$ orthonormal vector fields $(v_1,\dots,v_{k-1})$ of $\eta$
over $T^{(3)}$. This induces a decomposition
\[
\eta|_{T^{(3)}}=\epsilon_1\oplus\dots\oplus\epsilon_{k-1}\oplus L,
\]
where $\epsilon_j$, $j=1,\dots,k-1$ are trivial and trivialized line
bundles and where $L$ is a line bundle. The line bundle $L$ has Chern
class $c_1(L)=c_1(\eta)=g(\eta)a$ for some $g(\eta)\ge 0$ and some
$a\in H^{2}(M;\zz)=H^{2}(T^{(3)};\zz)$.
Let $K$ be a line bundle over $T^{(3)}$ with $c_1(K)=a$ and let $w$ be
a generic section of $K$. (Here we say that $w$ is generic if it does
not vanish along $T^{(1)}$ and if its $0$-set is transverse to
$T^{(2)}$.) Since two line bundles are isomorphic if and only if they
have the same Chern class, it follows that $L\cong K^{\otimes
  g(\eta)}$. In particular, $\left(v_1,\dots,v_{k-1},
  w^{g(\eta)}\right)$ gives a framing $\zeta^{(1)}$ of $\eta$ over
$T^{(1)}$.

\begin{lem}\label{l:divobst} If $\delta_j\colon S^{1}\to T^{(1)}$,
  $j=1,\dots,m$ are curves in $T^{(1)}$ and if $C\colon \Sigma\to
  T^{(2)}$ is any map of an orientable surface with $m$ boundary
  components $\pa\Sigma=\pa\Sigma_1\cup\dots\cup\pa\Sigma_m$ such that
  $C|\pa\Sigma_j=\delta_j$, then the obstruction to extending the
  trivialization $C^{\ast}(\zeta^{(1)})$ of $C^{\ast}(\eta)$ from
  $\pa\Sigma$ to $\Sigma$ is a class $b\in
  H^{2}\bigl(\Sigma,\pa\Sigma;\zz\bigr)$ which is divisible by
  $g(\eta)$.
\end{lem}

\begin{proof} The obstruction to finding such a trivialization is
  equal to the obstruction to extending the section $w^{g(\eta)}$ of
  $C^{\ast}(L)$ over $\Sigma$, which equals $g(\eta)$ times the
  obstruction of extending $w$.
\end{proof}

The first cohomology group $H^{1}(M;\zz_g)$ of $M$ acts naturally on
$\zz_g$-framings as follows. Let $g\ge 0$ and let $\gamma$ be a closed
curve in $M$. Assume that $\gamma$ is equipped with a $\zz_g$-framing
$\Xi_\gamma$ of the complex vector bundle $\eta$. If $b\in
H^{1}(M;\zz_g)$, then we define an action of $b$ on the
$\zz_g$-framing of $\gamma$, $\Xi_\gamma\mapsto\Xi'_\gamma$ as
follows. Pick a framing $Z_\gamma$  of $\eta|_\gamma$ representing
$\Xi_\gamma$. Then the $\zz_g$-framing $\Xi'_\gamma$ is represented by
any framing $Z_\gamma'$ with $d(Z_\gamma,Z'_\gamma)=\la
b,[\gamma]\ra \mod{g}$, where $\la a,\beta\ra$ denotes the homomorphism
$H_1(M;\zz)\to\zz_g$ corresponding to the cohomology class $a$
evaluated on the homology class $\beta$. Note that in the special case
$g=0$, a $\zz_g$-framing is simply a framing and $H^{1}(M;\zz)$ acts on
framings.

\begin{lem}\label{l:indframe} Let $M$ be a manifold with a complex
  vector bundle $\eta$. Let $g=g(\eta)$ be the greatest divisor of
  $\eta$, with $c_1(\eta)=g \cdot a$. Let $\gamma$ be any closed curve
  in $M$. Then there is an induced $\zz_g$-framing $\Xi_\gamma$ of
  $\gamma$ which is unique up to the action of $H^{1}(M;\zz_g)$. In
  particular, the $\zz_g$-framing of any $\gamma$ which represents a
  homology class in $H_1(M;\zz)$ which generates a subgroup isomorphic
  to $\zz_m$ where $m$ and $g$ are relatively prime is unique. In the
  special case $g=0$, the framing of any curve representing a torsion
  class is unique.
\end{lem}

\begin{proof} Fix a triangulation $T$ of $M$. Construct a framing
  $\zeta^{(1)}$ over $T^{(1)}$ as described above. 
  Pick a homotopy $C$ which connects $\gamma$ to a
  curve in $T^{(1)}$. Transporting the trivialization $\zeta^{(1)}$
  over $C$ gives a trivialization $Z_\gamma$ over $\gamma$. Let $C'$
  be another homotopy inducing another trivialization $Z'_\gamma$. Let
  $C''$ be a homotopy in $T^{(2)}$ connecting the end-curve of $C$ to
  the end-curve of $C'$. Then the cylinder $D$ constructed by joining
  $C$ to $C''$ and $C''$ to $C'$ connects $\gamma$ to itself. The
  obstruction to finding a framing over the torus corresponding to $D$
  is on the one hand equal to $\la c_1(\eta),[D]\ra=g\la a,[D]\ra$. On
  the other hand, this obstruction equals the sum of
  $d(Z_\gamma,Z'_\gamma)$ and the obstruction $o\in\zz$ to extending
  the framing $\zeta^{(1)}$ over $C''$. The obstruction $o$ is
  divisible by $g$ by Lemma \ref{l:divobst}, $o=go'$. Thus
  \[ d(Z_\gamma,Z'_\gamma)=g\bigl(\la a,[D]\ra-o'\bigr),
  \] and existence of $\Xi_\gamma$ follows.

  We next consider uniqueness. Consider applying the above
  construction in two ways for a fixed triangulation $T$. This gives
  two generic frames $(v_1,\dots,v_{k-1},w)$ and
  $(v_1',\dots,v_{k-1}',w')$ over $T^{(2)}$. Lemma \ref{l:triv-1}
  implies that there is a homotopy of the $(k-1)$-frames. Note that
  such a homotopy induces $1$-parameter family of bundle isomorphisms
  on orthonormal complements of the $(k-1)$-frame. Deforming
  $(v_1',\dots,v_{k-1}')$ to $(v_1,\dots,v_{k-1})$ thus gives a new
  generic frame $(v_1,\dots, v_{k-1},w'')$. The generic frames give
  the same $\zz_g$-framing of all loops in $T^{(1)}$ provided the
  corresponding difference classes of the sections $w$ and $w''$ of
  the line bundle $L$ are divisible by $g$. It follows that
  $H^{1}(T^{(2)};\zz_g)$ acts transitively on homotopy classes of
  $\zz_g$-framings of loops in $T^{(1)}$. Thus all loop
  $\zz_g$-framings obtained from a fixed triangulation form a
  principal homogeneous space over
  $H^{1}(T^{(2)};\zz_g)=H^{1}(M;\zz_g)$.

  Let $S$ be some other triangulation of $M$. After
  small perturbation of $S$ we may assume that the triangulations $S$
  and $T$ have a common refinement $U$. Using $U$ we find that there
  exist trivializations $\zeta^{(1)}_T$ and $\zeta^{(1)}_S$ over the
  $1$-skeleta of $T$ and $S$, respectively such that the corresponding
  loop $\zz_g$-framings agree. Since the action of
  $H^{1}(M;\zz_g)$ on $\zz_g$-framings is independent of triangulation it
  follows that the set of loop $\zz_g$-framings is independent of the
  chosen triangulation.
\end{proof}

\begin{rem}\label{r:cotangent1} Let $M$ be an orientable
  manifold. Consider the tangent bundle of the cotangent bundle
  $T(T^{\ast}M)$. After fixing an almost complex structure $J$ on
  $T(T^{\ast}M)$ which is compatible with the standard symplectic
  form, we consider $T(T^{\ast}M)$ as complex vector bundle. Since
  $T^{\ast}M\simeq M$ and since the restriction of $T(T^{\ast}M)$ to
  $M$ is the complexification of a real vector bundle, we see that
  $c_1(T(T^{\ast}M))=0$. Furthermore, an orientation of $M$ gives a
  non-zero section of $\Lambda^{\rm max}TM$, which gives a non-zero
  section of $\Lambda^{\rm max}T(T^{\ast}M)$. This, in turn, induces a
  trivialization over the $3$-skeleton of $M$ following the
  construction above by requiring that the section
  $v_1\wedge\dots\wedge v_{n-1}\wedge w$, where $v_j$ and $w$ give the
  trivialization $\zeta^{(1)}$ in Lemma \ref{l:divobst}, is homotopic
  to the one induced by the orientation.  The result is the canonical
  trivialization discussed after Definition~\ref{d:Z/gZ-framing}.
\end{rem}


\def\cprime{$'$} \def\polhk#1{\setbox0=\hbox{#1}{\ooalign{\hidewidth
  \lower1.5ex\hbox{`}\hidewidth\crcr\unhbox0}}}
\providecommand{\bysame}{\leavevmode\hbox to3em{\hrulefill}\thinspace}
\providecommand{\MR}{\relax\ifhmode\unskip\space\fi MR }
\providecommand{\MRhref}[2]{%
  \href{http://www.ams.org/mathscinet-getitem?mr=#1}{#2}
}
\providecommand{\href}[2]{#2}


\begin{thebibliography}{10}

\bibitem{cohen-betz}
M.~Betz and R.~L. Cohen, \emph{Graph moduli spaces and cohomology operations},
  Turkish J. Math. \textbf{18} (1994), no.~1, 23--41.

\bibitem{chv}
Yu. Chekanov, \emph{Differential algebra of {L}egendrian links}, Invent. Math.
  \textbf{150} (2002), 441--483.

\bibitem{ekholm:morse-flow}
T.~Ekholm, \emph{Morse flow trees and {L}egendrian contact homology in 1-jet
  spaces}, Geom. Topol. \textbf{11} (2007), 1083--1224.

\bibitem{ekholm:rat-sft-ex-lag}
T.~Ekholm, \emph{Rational symplectic field theory over ${\mathbb Z}_2$ for exact Lagrangian cobordisms}, J. Eur. Math. Soc. (JEMS) \textbf{10} (2008), no.~3, 641--704.


\bibitem{ees:high-d-analysis}
T.~Ekholm, J.~Etnyre, and M.~Sullivan, \emph{The contact homology of
  {L}egendrian submanifolds in {${\mathbb R}\sp {2n+1}$}}, J. Differential
  Geom. \textbf{71} (2005), no.~2, 177--305.

\bibitem{ees:high-d-geometry}
\bysame, \emph{Non-isotopic {L}egendrian submanifolds in {$\mathbb R\sp
  {2n+1}$}}, J. Differential Geom. \textbf{71} (2005), no.~1, 85--128.

\bibitem{ees:ori}
\bysame, \emph{Orientations in {L}egendrian contact homology and exact
  {L}agrangian immersions}, Internat. J. Math. \textbf{16} (2005), no.~5,
  453--532.

\bibitem{ees:pxr}
\bysame, \emph{Legendrian contact homology in {$P\times\mathbb{R}$}}, Trans.
  Amer. Math. Soc. \textbf{359} (2007), no.~7, 3301--3335 (electronic).

\bibitem{yasha:icm}
Ya. Eliashberg, \emph{Invariants in contact topology}, Proceedings of the
  International Congress of Mathematicians, Vol. II (Berlin, 1998), no. Extra
  Vol. II, 1998, pp.~327--338 (electronic).

\bibitem{egh}
Ya. Eliashberg, A.~Givental, and H.~Hofer, \emph{Introduction to symplectic
  field theory}, Geom. Funct. Anal. (2000), no.~Special Volume, Part II,
  560--673.

\bibitem{epstein-fuchs}
J.~Epstein and D.~Fuchs, \emph{On the invariants of {L}egendrian mirror torus
  links}, Symplectic and contact topology: interactions and perspectives
  (Toronto, ON/Montreal, QC, 2001), Fields Inst. Commun., vol.~35, Amer. Math.
  Soc., Providence, RI, 2003, pp.~103--115.

\bibitem{floer-mem}
A.~Floer, \emph{Monopoles on asymptotically flat manifolds (3--41)}, in The Floer memorial volume.
Edited by Helmut Hofer, Clifford H. Taubes, Alan Weinstein and Eduard Zehnder. Progress in Mathematics, 133. Birkh{\"a}user Verlag, Basel, 1995.

\bibitem{fuchs:augmentations}
D.~Fuchs, \emph{Chekanov-{E}liashberg invariant of {L}egendrian knots:
  existence of augmentations}, J. Geom. Phys. \textbf{47} (2003), no.~1,
  43--65.

\bibitem{fuchs-ishk}
D.~Fuchs and T.~Ishkhanov, \emph{Invariants of {L}egendrian knots and
  decompositions of front diagrams}, Mosc. Math. J. \textbf{4} (2004), no.~3,
  707--717.

\bibitem{fooo}
K.~Fukaya, Y-G. Oh, T.~Ohta, and K~Ono, \emph{{L}agrangian intersection {F}loer
  homology --- anomaly and obstruction}, Preprint, 2000.


\bibitem{melvin-shrestha}
P.~Melvin and S.~Shrestha, \emph{The nonuniqueness of {C}hekanov polynomials of
  {L}egendrian knots}, Geom. Topol. \textbf{9} (2005), 1221--1252.

\bibitem{milinkovic}
D.~Milinkovi{\'c}, \emph{Morse homology for generating functions of
  {L}agrangian submanifolds}, Trans. Amer. Math. Soc. \textbf{351} (1999),
  no.~10, 3953--3974.

\bibitem{kirill}
K.~Mishachev, \emph{The $n$-copy of a topologically trivial {L}egendrian knot},
  J. Symplectic Geom. \textbf{1} (2003), no.~4, 659--682.

\bibitem{lenny:computable}
L.~Ng, \emph{Computable {L}egendrian invariants}, Topology \textbf{42} (2003),
  no.~1, 55--82.

\bibitem{lenny:knots3}
\bysame, \emph{{Framed knot contact homology}}, Preprint available on arXiv as
  math.GT/0407071, 2004.

\bibitem{ns:augm-rulings}
L.~Ng and J.~Sabloff, \emph{The correspondence between augmentations and
  rulings for {L}egendrian knots}, Pacific J. Math. \textbf{224} (2006), no.~1,
  141--150.

\bibitem{lenny-lisa}
L.~Ng and L.~Traynor, \emph{Legendrian solid-torus links}, J. Symplectic Geom.
  \textbf{2} (2004), no.~3, 411--443.

\bibitem{robbin-salamon:maslov}
J.~Robbin and D.~Salamon, \emph{The {M}aslov index for paths}, Topology
  \textbf{32} (1993), 827--844.

\bibitem{rutherford:kauffman}
D.~Rutherford, \emph{Thurston-{B}ennequin number, {K}auffman polynomial, and
  ruling invariants of a {L}egendrian link: the {F}uchs conjecture and beyond},
  Int. Math. Res. Not. (2006), Art. ID 78591, 15.

\bibitem{rulings}
J.~Sabloff, \emph{Augmentations and rulings of {L}egendrian knots}, Int. Math.
  Res. Not. (2005), no.~19, 1157--1180.

\bibitem{duality}
\bysame, \emph{Duality for {L}egendrian contact homology}, Geom. Topol.
  \textbf{10} (2006), 2351--2381 (electronic).

\bibitem{schwarz}
M.~Schwarz, \emph{Morse homology}, Progress in Mathematics, vol. 111,
  Birkh\"auser Verlag, Basel, 1993.

\bibitem{schwarz:equivalence}
\bysame, \emph{Equivalences for {M}orse homology}, Geometry and topology in
  dynamics (Winston-Salem, NC, 1998/San Antonio, TX, 1999), Contemp. Math.,
  vol. 246, Amer. Math. Soc., Providence, RI, 1999, pp.~197--216.

\bibitem{smale}
S.~Smale, \emph{On gradient dynamical systems}, Ann. of Math. (2) \textbf{74}
  (1961), 199--206.

\end{thebibliography}
\end{document}